\documentclass[12pt]{amsart}
\usepackage{amsmath}
\usepackage{amsfonts}
\usepackage[T1]{fontenc} 
\usepackage{color}
\usepackage{url}
\usepackage{lmodern}		

\usepackage{babel} \usepackage[babel=true]{microtype} \usepackage[hidelinks]{hyperref} 
\hypersetup{colorlinks=true, urlcolor=red,} 

\usepackage{a4wide}
\usepackage{verbatim}

\usepackage{graphicx}
\usepackage{amssymb}
\usepackage{epstopdf}
\usepackage{psfrag}

\graphicspath{{}}

\newcommand {\bea}{\begin{align}}
\newcommand {\ea}{\end{align}}
\newtheorem{proposition}{Proposition}[section]
\newtheorem{theorem}{Theorem}[section]

\newtheorem{lemma}{Lemma}[section]

\newtheorem{remark}{Remark}[section]

\newtheorem{corollary}{Corollary}[section]

\usepackage{xspace}

\usepackage{subfigure}

\newcommand{\thmref}[1]{{Theorem~\ref{#1}}}
\newcommand{\lemref}[1]{{Lemma~\ref{#1}}}
\newcommand{\secref}[1]{{Section~\ref{#1}}}

\newcommand{\propref}[1]{{Proposition~\ref{#1}}}
\newcommand{\coref}[1]{{Corollary~\ref{#1}}}

\newcommand{\rmref}[1]{{Remark~\ref{#1}}}

\newcommand{\red}{\color{red}}

\newcommand{\eps}{\varepsilon}
\newcommand{\revdr}[1]{{\color{red} #1}}

\begin{document}
\title[A posteriori error analysis of the SCHE with rough noise]
{Robust a posteriori error analysis of the stochastic Cahn-Hilliard equation with rough noise}

\author{\v{L}ubom\'{i}r Ba\v{n}as}
\address{Department of Mathematics, Bielefeld University, 33501 Bielefeld, Germany}
\email{banas@math.uni-bielefeld.de}
\author{Jean Daniel Mukam}
\address{Department of Mathematics and Informatic, School of Mathematics and Natural Sciences, University of Wuppertal, 42119 Wuppertal, Germany}
\email{mukam@uni-wuppertal.de}

\begin{abstract}

We derive a posteriori error estimate for a fully discrete adaptive ﬁnite
element approximation of the stochastic Cahn-Hilliard equation with rough noise.
The considered model is derived from the stochastic Cahn-Hilliard equation with additive space-time white noise through suitable spatial regularization of the white noise.
The a posteriori estimate is robust with respect to the interfacial width parameter as well as the noise regularization parameter.
We propose a practical adaptive algorithm for the considered problem and perform numerical simulations to illustrate the theoretical findings.
\end{abstract}

\maketitle



\section{Introduction}

The stochastic Cahn-Hilliard equation with additive space-time white noise reads as
\begin{align}
\label{model1}
du&=\Delta wdt+ dW &&\text{in}\; (0, T)\times\mathcal{D},\nonumber\\
w&=-\varepsilon\Delta u+\varepsilon^{-1}f(u) && \text{in}\; (0, T)\times\mathcal{D},\\
\partial_{\vec{n}}u&=\partial_{\vec{n}}w=0 &&\text{on}\; (0,T)\times\partial\mathcal{D},\nonumber\\
u(0)&=u^{\varepsilon}_0 &&\text{in}\; \mathcal{D},\nonumber
\end{align}
where $T>0$ is fixed, $\mathcal{D}\subset\mathbb{R}^d$, $d\geq 1$ is an open bounded domain with boundary $\partial\mathcal{D}$ and $\vec{n}$ denotes the outer unit normal vector to $\partial \mathcal{D}$.
The constant $0<\varepsilon\ll 1$ is called the interfacial with parameter.
The nonlinearity in \eqref{model1} is given by $f(u)=F'(u)=u^3-u$, where  $F(u)=\frac{1}{4}(u^2-1)^2$ is the double-well potential.

The term  $W$ in \eqref{model1} represents the space-time white noise
which  can be formally  expressed as
\begin{align}
\label{SpaceTimeWhiteNoise}
W(t,x)=\sum_{j\in\mathbb{N}^d}\beta_j(t)e_j(x),
\end{align}
where the $\beta_j$, $j\in\mathbb{N}^d$, are independent and identically distributed Brownian motions on a filtered probability space $(\Omega, \mathcal{F}, \{\mathcal{F}_t\}_t, \mathbb{P})$ and $\{e_j\}_{j\in\mathbb{N}^d}$ are the eigenvectors of the Neumann Laplacian $-\Delta$ with domain $ D(-\Delta)=\{u\in \mathbb{H}^2:\; \partial_{\vec{n}}u=0\;\text{on}\; \partial\mathcal{D}\}$.

For simplicity we take $\mathcal{D}=(0, 1)^d$ to be the unit cube in $\mathbb{R}^d$, $d=1,2,3$.
To avoid technicalities we assume that the initial data $u^{\varepsilon}_0\in \mathbb{H}^1$ and has zero mean, i.e., $\int_{\mathcal{D}}u^{\varepsilon}_0dx=0$.
Furthermore, we assume that the noise is mean-value preserving, i.e., $\int_{\mathcal{D}} W(t, x) dx=0$ for a.a. $t\in [0,T]$, $\mathbb{P}$-a.s. (i.e., we drop the constant mode in \eqref{SpaceTimeWhiteNoise}, cf. \cite{Debussche1}).
The zero mean conditions on the initial data and the noise imply that $\int_{\mathcal{D}} u(t, x) dx=0$ for $t\in [0,T]$, $\mathbb{P}$-a.s.

Recently, a posteriori estimates for adaptive  finite element approximation of linear stochastic partial differential equations (SPDEs) with $\mathbb{H}^2\cap\mathbb{W}^{1, \infty}$-trace class noise were investigated in \cite{MajeeProhl22}, generalizing the variational concepts  of the residual-based estimators for deterministic parabolic PDEs (cf. e.g., \cite{ChenFeng2004}) to linear SPDEs.
Due to the lack of differentiability in time of solutions to SPDEs, \cite{MajeeProhl22} employs a linear transformation that transforms the (linear) SPDE into a (linear) random PDE (RPDE) which is amenable to a posteriori analysis.
This approach was recently generalized to nonlinear SPDEs in \cite{BanasVieth_a_post23}, \cite{bw21}.
 The work \cite{BanasVieth_a_post23} derives robust a posteriori estimates for the stochastic Cahn-Hilliard equation with additive $\mathbb{H}^4\cap\mathbb{W}^{1, \infty}$-trace class noise
and \cite{bw21} considers a posteriori estimate for the stochastic total variation flow requiring $\mathbb{H}^2$-regularity of the noise.

The stochastic Cahn-Hilliard equation with space-time white noise (\ref{model1}) is not amenable to a posteriori error analysis since its solution
does not posses enough (spatial) regularity to formulate a suitable error equation for the numerical approximation. I.e., the order parameter $u$ is not $\mathbb{H}^2$-regular in space, cf. \cite{Debussche1}, \cite{BanasDaniel23}, \cite{scarpa18}. In addition, the chemical potential $w$ is not properly defined in the case of space-time white noise (cf. \cite{BanasDaniel23}, \cite{scarpa18}),
which prohibits the application of a suitable counterpart of the linear transformation from \cite{BanasVieth_a_post23} (see \eqref{lineartransform} below).
Hence, we consider the regularized stochastic Cahn-Hilliard equation \eqref{pb1}, which is obtained by replacing the space-time white noise \eqref{SpaceTimeWhiteNoise} in the original problem (\ref{model1}) by its piecewise linear approximation \eqref{Noiseapprox2}.
To derive the a posteriori error estimate for the numerical approximation of \eqref{pb1}, we adopt a similar approach as in \cite{BanasVieth_a_post23}. We split the solution as $u = \tilde{u} + \hat{u}$,
where $\tilde{u}$ solves the linear SPDE \eqref{model2} and $\hat{u}$ solves the random PDE (RPDE) \eqref{model3}.
Analogously to \cite{BanasVieth_a_post23}, to obtain estimate that are robust with respect to the interfacial width parameter $\varepsilon$
we make use of the (computable) principal eigenvalue \eqref{Eigenvalue1} (see also \cite{BartelsMueller2011}, \cite{Bartels_book2015}).
Our work differs from \cite{BanasVieth_a_post23} in the following aspects.
\begin{itemize}
\item To derive the a posteriori error estimate for the linear SPDE \eqref{model2} in the low-regularity setting requires the use of a modified linear transformation, see \rmref{RemLinearT} below,
along with an appropriate treatment of the regularized noise.
\item
We adopt a refined approach for the derivation of pathwise a posteriori estimate for the RPDE \eqref{model3}.
We derive the error estimate on a subspace
$\Omega_{\delta,\tilde{\varepsilon}} \cap \Omega_{\gamma, \tilde{\varepsilon}} \cap \Omega_{\tilde{\varepsilon}}$,
where, the set $\Omega_{\tilde{\varepsilon}}$ \eqref{SetOmegatilde} controls the approximation error of the linear SPDE,
the set $\Omega_{\delta, \tilde{\varepsilon}}$ \eqref{SetOmegadelta} corresponds to the subspace on which the $L^{\infty}(0, T; \mathbb{H}^{-1})$- and $L^4(0, T; \mathbb{L}^4)$-norms of the solution are bounded by a prescribed threshold, and
the set $\Omega_{\gamma, \tilde{\varepsilon}}$ \eqref{SetOmegagamma} corresponds to the subspace on which the $L^{\infty}(0, T; \mathbb{L}^4)$-norm of the solution to the linear SPDE is bounded by a prescribed threshold.
Using the new interpolation inequality in \lemref{Fundamentallemma}, in \thmref{mainresult2} we derive  pathwise a posteriori error estimate for the approximation of the random PDE on the subspace $\Omega_{\delta,\tilde{\varepsilon}} \cap \Omega_{\gamma, \tilde{\varepsilon}} \cap \Omega_{\tilde{\varepsilon}}$.
By combining variational and semigroup techniques, we prove that $\Omega_{\delta, \tilde{\varepsilon}}$ and $\Omega_{\gamma, \tilde{\varepsilon}}$ are subspaces of high probability (see Lemmas \ref{Normutilde} and \ref{LemmaRegularity}).
Furthermore, the approximation error of the linear SPDE on the set $\Omega_{\tilde{\varepsilon}}$ can be controlled owing to Corollary~\ref{RateLinearSPDE}.
Using the fact that $\Omega_{\delta, \tilde{\varepsilon}}$, $\Omega_{\gamma, \tilde{\varepsilon}}$ are subspaces of high probability
we combine the pathwise estimate in \thmref{mainresult2} with the error estimate for the linear SPDE in \lemref{LemmaErrortilde}
and obtain an error estimate for the numerical approximation
of \eqref{pb1} in \thmref{mainresult3}.
\item
As a byproduct of our analysis we obtain some additional new results.
The error estimate in \thmref{mainresult3} holds on the whole sample space $\Omega$. This improves the earlier work \cite{BanasVieth_a_post23},
where the derived a posteriori error estimate for the stochastic Cahn-Hilliard equation with smooth noise in spatial dimension $d=3$ was restricted to the subspace
$\Omega_{\infty}=\left\{\omega\in\Omega:\; \sup_{t\in(0,T)}\Vert u(t)\Vert_{\mathbb{L}^{\infty}} \leq C_{\infty}\right\}$. A rigorous estimate for this subspace $\Omega_{\infty}$ has not yet been established.
Furthermore, in \thmref{ThmRateLinearSPDE}, we obtained convergence rate for the fully discrete numerical approximation of the linear fourth order SPDE \eqref{model2} with $\mathbb{H}^1$-regular noise. This appears to be a new result.
\end{itemize}

The paper is organized as follows.
In Section~\ref{sec_notation} we introduce the notation and auxiliary results.
In \secref{Regularizedproblem}, we introduce the regularized problem and its fully discrete numerical approximation is given in \secref{sec_full}.
In \secref{ErrorlinearSPDE}, we derive the error estimate for the linear PDE. \secref{ErrorRandonPDE} is dedicated to the error analysis of the random PDE. In \secref{ErrorRegularizedCHE}, we combine the estimates from \secref{ErrorlinearSPDE} and \secref{ErrorRandonPDE} and derive the error estimate for the numerical approximation of the stochastic Cahn-Hilliard equation.
Numerical experiments are presented in \secref{sec_num}.
Auxiliary results are collected in Appendices~\ref{SectionRegularity}~and~\ref{Rateimplicit}.

\section{Notation and preliminaries}
\label{sec_notation}

For $p\in [1, \infty]$, we denote by $(\mathbb{L}^p, \Vert \cdot\Vert_{\mathbb{L}^p}):=(L^p(\mathcal{D}), \Vert \cdot\Vert_{L^p(\mathcal{D})})$ the space of  equivalence classes of functions on $\mathcal{D}$ that are  $p$-th order integrable. We denote by $(\cdot, \cdot)$ the inner product in $\mathbb{L}^2$, and by $\Vert \cdot\Vert:=\Vert \cdot\Vert_{\mathbb{L}^2}$ its associated norm. For any $k\in\mathbb{N}$, we denote by $(\mathbb{H}^k, \Vert\cdot\Vert_{\mathbb{H}^k}):=(H^k(\mathcal{D}),  \Vert\cdot\Vert_{H^k(\mathcal{D})})$  the standard Sobolev space of functions whose derivatives up to order $k$ belong to $\mathbb{L}^2$. For $r>0$, we denote by $\mathbb{H}^r$ the standard fractional Sobolev space. For  $r\geq 0$, $\mathbb{H}^{-r}:=(\mathbb{H}^r)^*$ stands for the dual space of $\mathbb{H}^r$.  We denote by $\langle\cdot, \cdot\rangle$  the duality pairing between $\mathbb{H}^1$ and $\mathbb{H}^{-1}$, with the norm  defined as
\begin{align}
\label{NormH-1}
\Vert u\Vert_{\mathbb{H}^{-1}} =\sup_{v\in \mathbb{H}^1}\frac{\langle u,v\rangle}{\| v\|_{\mathbb{H}^{1}}}.
\end{align}
Furthermore, we consider the space $\mathring{\mathbb{H}}^{-1} = \left\{v\in \mathbb{H}^{-1}:\,\, \langle v, 1\rangle =0 \right\}$.

For $v\in \mathbb{L}^2$, we denote by $m(v)$ the mean value of $v$, i.e.,
\begin{align*}
m(v):=\frac{1}{\vert \mathcal{D}\vert}\int_{\mathcal{D}}v(x)dx,\quad v\in \mathbb{L}^2,
\end{align*}
and define the space $\mathbb{L}^2_0=\{\varphi\in \mathbb{L}^2:\; m(\varphi)=0\}$.

We consider the inverse Neumann Laplacian $(-\Delta)^{-1}: \mathring{\mathbb{H}}^{-1} \rightarrow  \mathbb{H}^2\cap\mathbb{L}^2_0$, i.e.,
for  $v\in \mathring{\mathbb{H}}^{-1}$ we let $\tilde{v} :=(-\Delta)^{-1}v$ be the unique variational solution to:
\begin{align*}
-\Delta \tilde{v} & =v\quad \text{in}\; \mathcal{D} \\
 \partial_{\vec{n}} \tilde{v} & =0\quad \text{on}\; \partial\mathcal{D}.
\end{align*}
In particular for $\overline{v}\in \mathbb{L}^2_0$ it holds that  $(\nabla(-\Delta)^{-1}\overline{v}, \nabla\varphi)=(\overline{v}, \varphi)$ for all $\varphi\in \mathbb{H}^1$.

The inner product  on $\mathring{\mathbb{H}}^{-1}$ is defined by
\begin{align*}
(u,v)_{-1}:= (\nabla(-\Delta)^{-1}\overline{v}, \nabla (-\Delta)^{-1} v) \qquad \forall u, v\in \mathring{\mathbb{H}}^{-1}.
\end{align*}
Note that the norm associated to the above scalar product is equivalent to the $\mathbb{H}^{-1}$-norm (\ref{NormH-1}) on $\mathring{\mathbb{H}}^{-1}$.

\section{The regularized stochastic Cahn-Hilliard equation}\label{Regularizedproblem}
Let $\mathcal{T}_{\tilde{h}}$ be a quasi-uniform partition of $\mathcal{D}$ into simplices with mesh-size $\tilde{h}=\max_{K\in \mathcal{T}_{\tilde{h}}}\text{diam}(K)$.
Let $\mathbb{V}_{\tilde{h}}\equiv \mathbb{V}_{\tilde{h}}(\mathcal{T}_{\tilde{h}})\subset\mathbb{H}^1$  be the finite element space of piecewise affine, globally continuous functions on $\mathcal{D}$, that is,
 \begin{align*}
 \mathbb{V}_{\tilde{h}}:=\{v_{\tilde{h}}\in C(\bar{\mathcal{D}}):\; {v_{\tilde{h}}}_{|_K}\in \mathbb{P}_1(K)\quad \forall K\in \mathcal{T}_{\tilde{h}}\}.
\end{align*}
Let $\phi_{\ell}$, $\ell=1, \cdots, L$, be the basis functions of $\mathbb{V}_{\tilde{h}}$, s.t., $\mathbb{V}_{\tilde{h}}=\text{span}\{\phi_{\ell},\; \ell=1, \cdots, L\}$.
As in \cite{BanasBrzezniakProhl2013,sllg_book}, we introduce the following approximation of the space-time white noise (\ref{SpaceTimeWhiteNoise}):
\begin{align*}
\widehat{W}(t, x):=\sum_{\ell=1}^L\frac{\phi_{\ell}(x)}{\sqrt{(d+1)^{-1}\vert (\phi_{\ell}, 1)\vert}}\beta_{\ell}(t)\quad x\in \bar{\mathcal{D}}\subset\mathbb{R}^d,
\end{align*}
where   $(\beta_{\ell})_{\ell=1}^L$ are standard real-valued Brownian motions.
To ensure the zero mean-value property of the noise at the discrete level, we normalize the noise $\widehat{W}$ as:
\begin{align}
\label{Noiseapprox2}
\widetilde{W}(t):=\widehat{W}(t)-\frac{1}{\vert\mathcal{D}\vert}(\widehat{W}(t), 1) =\sum_{\ell}^L\frac{\phi_{\ell}-m(\phi_{\ell})}{\sqrt{(d+1)^{-1}\vert (\phi_{\ell}, 1)\vert}}\beta_{\ell}(t).
\end{align}

\begin{remark}
The discrete noise $\widehat{W}$ was considered in \cite{sllg_book,BanasBrzezniakProhl2013} as an approximation of the space-time white noise, cf. \cite[Remark A.1]{BanasBrzezniakProhl2013}.
The approximation $\widehat{W}$ can also be interpreted as the $\mathbb{L}^2$-projection onto $\mathbb{V}_{\tilde{h}}$ of the higher-dimensional  analogue of the piecewise constant approximation of the space-time white noise from \cite{AllenNovoselZhang1998}.
\end{remark}

The regularized stochastic Cahn-Hilliard equation is obtained by replacing the white noise $W$ in \eqref{model1} with the approximation $\widetilde{W}$ as
\begin{align}
\label{pb1}
du&=\Delta wdt+ d\widetilde{W}(t) && \text{in}\; (0, T)\times\mathcal{D},\nonumber\\
w&=-\varepsilon\Delta u+\varepsilon^{-1}f(u) && \text{in}\; (0, T)\times\mathcal{D},\\
\partial_{\vec{n}}u&=\partial_{\vec{n}}w=0 && \text{on}\; (0,T)\times\mathcal{D},\nonumber\\
u(0)&=u^{\varepsilon}_0 && \text{in}\; \mathcal{D}.\nonumber
\end{align}
The solution of \eqref{pb1} can be written as $u=\widetilde{u}+\widehat{u}$, where $\widetilde{u}$ solves the  linear SPDE
\begin{align}
\label{model2}
d\widetilde{u}&=\Delta \widetilde{w}dt+ d\widetilde{W}(t)&&\text{in}\; (0, T)\times\mathcal{D},\nonumber\\
\widetilde{w}&=-\varepsilon\Delta \widetilde{u}&& \text{in}\; (0, T)\times \mathcal{D},\\
\partial_{\vec{n}}\widetilde{u}&=\partial_{\vec{n}}\widetilde{w}=0&& \text{on} \; (0, T)\times\partial\mathcal{D},\nonumber\\
\widetilde{u}(0)&=0&&\text{in}\; \mathcal{D},\nonumber
\end{align}
and $\widehat{u}$ solves the following random PDE:
\begin{align}
\label{model3}
d\widehat{u}&=\Delta \widehat{w}dt && \text{in}\; (0, T)\times\mathcal{D},\nonumber\\
\widehat{w}&=-\varepsilon\Delta \widehat{u}+\frac{1}{\varepsilon}f(u) && \text{in}\; (0, T)\times \mathcal{D},\\
\partial_{\vec{n}}\widehat{u}&=\partial_{\vec{n}}\widehat{w}=0 && \text{on} \; (0, T)\times\partial\mathcal{D},\nonumber\\
\widehat{u}(0)&=u^{\varepsilon}_0&& \text{in}\; \mathcal{D}.\nonumber
\end{align}

The linear SPDE \eqref{model2} has a unique (analytically) weak solution, see e.g., \cite{Debussche1}, i.e., there exists $(\widetilde{u}, \widetilde{w})$ that satisfy for $t\in(0,T)$, $\mathbb{P}$-a.s.:
\begin{subequations}
\label{weak1}
\begin{align*}
(\widetilde{u}(t), \varphi)+\int_0^t(\nabla\widetilde{w}(s), \nabla\varphi)ds&=\left(\int_0^td\widetilde{W}(s), \varphi\right)&\forall\varphi\in\mathbb{H}^1,\\
(\widetilde{w}(t), \psi)&=\varepsilon(\nabla\widetilde{u}(t), \nabla\psi) & \forall\psi\in\mathbb{H}^1.
\end{align*}
\end{subequations}
We introduce the linear transformation
\begin{align}
\label{lineartransform}
y(t,x)=\widetilde{u}(t,x)-\int_0^td\widetilde{W}(s,x),
\end{align}
and note that $(y, \widetilde{w})$ $\mathbb{P}$-a.s. solves the random PDE
\begin{equation}
\label{weak2}
\begin{array}{rclll}
(y(t), \varphi)+\displaystyle\int_0^t(\nabla\widetilde{w}(s), \nabla\varphi)ds&=&0 &\quad\forall\varphi\in\mathbb{H}^1,\\
(\widetilde{w}(t), \psi)&=&\varepsilon(\nabla\widetilde{u}(t), \nabla\psi)&\quad \forall\psi\in\mathbb{H}^1,
\end{array}
\end{equation}
for all $t\in(0,T)$, with $y(0)=0$.

We remark standard arguments (e.g., note Lemma~\ref{LemmaDiscdisrete} and take $\tau_n\rightarrow 0$ in \eqref{scheme2})
imply that $\widetilde{w}\in L^2(0, T; \mathbb{H}^{1})$, $\mathbb{P}$-a.s., for fixed $\tilde{h}$.
Hence, cf. \cite{bw21}, it follows that $\partial_t y\in L^2(0, T; \mathbb{H}^{-1})$, $\mathbb{P}$-a.s.
and \eqref{weak2} is equivalent to
\begin{equation}
\label{Eqy}
\begin{array}{rclll}
\langle\partial_ty(t), \varphi\rangle+(\nabla\widetilde{w}(t), \nabla\varphi)&=&0 &\quad\forall\varphi\in\mathbb{H}^1,\\
(\widetilde{w}(t), \psi)&=&\varepsilon(\nabla\widetilde{u}(t), \nabla\psi)&\quad \forall\psi\in \mathbb{H}^1.
\end{array}
\end{equation}

\begin{remark}
\label{RemLinearT}
In \cite{BanasVieth_a_post23} the linear transformation (\ref{lineartransform}) is also applied
to the variable $\widetilde{w} = -\varepsilon \Delta \widetilde{u}$. Hence,
instead of \eqref{Eqy}, in \cite[Section 5]{BanasVieth_a_post23} a RPDE is formulated for
the transformed variables $(y, y_w)$ with $y_w(t) = \widetilde{w}(t) +\varepsilon \Delta \int_0^t d\widetilde{W}(s)$.
This transformation requires $\mathbb{H}^4$-regularity of the noise and is therefore not applicable in our setting where the noise is only $\mathbb{H}^1$-regular.
\end{remark}

\section{Fully discrete numerical approximation}\label{sec_full}
We consider a possibly non-uniform partition $0=t_0<t_1<\cdots<t_N=T$ of the time interval $[0, T]$ with step sizes $\tau_n=t_n-t_{n-1}$, $n=1,\cdots,N$.
Below, we denote $\tau:=\max\limits_{n=1,\cdots,N}\tau_n$.
At time level $t_n$, we consider a quasi-uniform partition $\mathcal{T}^n_h$ of the domain $\mathcal{D}$ into simplices and the associated finite element space of globally continuous piecewise linear functions
\begin{align*}
\mathbb{V}_h^n=\{\varphi_h\in C(\bar{\mathcal{D}}):\; \varphi_{h}\vert_K\in \mathcal{P}_1(K)\qquad \forall K\in\mathcal{T}^n_h\}.
\end{align*}

For an element $K\in\mathcal{T}^n_h$, we denote by $\mathcal{E}_K$ the set of all faces $e$ of $\partial K$. The set of all faces of the elements of the mesh $\mathcal{T}^n_h$ is denoted by $\mathcal{E}^n_h=\cup_{K\in \mathcal{T}^n_h}\mathcal{E}_K$. The diameters of $K\in \mathcal{T}^n_h$ and $e\in\mathcal{E}^n_h$ are denoted  by $h_K$ and $h_e$ respectively. We set $h:=\max_{K\in\mathcal{T}^n_h}h_K$. We split $\mathcal{E}^n_h$ into the set of all interior and boundary faces, $\mathcal{E}^n_h=\mathcal{E}^n_{h, \mathcal{D}}\cup\mathcal{E}^n_{h, \partial\mathcal{D}}$, where $\mathcal{E}^n_{h, \partial\mathcal{D}}=\{e\in\mathcal{E}^n_h,\; e\subset\partial\mathcal{D}\}$.
For $K\in\mathcal{T}^n_h$ and $e\in\mathcal{E}^n_h$, we define the local patches $w_K=\cup_{\mathcal{E}_K\cap\mathcal{E}_{K'}\neq \emptyset}K'$ and $w_e=\cup_{e\in\mathcal{E}_{K'}}K'$.

We define the $\mathbb{L}^2$-projection $P^n_h: \mathbb{L}^2\rightarrow \mathbb{V}^n_h$ such that:
\begin{align}
\label{Defprojection}
(P^n_hv-v, \varphi_h)=0\quad \forall\varphi\in \mathbb{V}^n_h. 
\end{align}
For   $s\in\{1,2\}$, the projection $P^n_h$ satisfies the following error estimate (cf. \cite{Bartels_book2015,BrennerScottbook2002,Ciarletbook2002}) 
 \begin{align}
 \label{ErrorPnh}
 \Vert v-P^n_hv\Vert+h\Vert\nabla(v-P^n_hv)\Vert\leq Ch^{s}\Vert v\Vert_{\mathbb{H}^s}\quad \forall v\in \mathbb{H}^s. 
 \end{align}
We consider the Cl\'{e}ment-Scott-Zhang interpolation operator $C^n_h: \mathbb{H}^1\rightarrow \mathbb{V}^n_h$, which satisfies the following local  error estimates: there exists a constant $C^*>0$  depending only on the minimum angle of the mesh $\mathcal{T}^n_h$ (cf. \cite[Definition 3.8]{Bartels_book2015}) such that for all $\psi\in \mathbb{H}^1$:
\begin{align}
\label{ErrorCnh}
\Vert\psi-C^n_h\psi\Vert_{L^2(K)}+h_K\Vert\nabla[\psi-C^n_h\psi]\Vert_{L^2(K)}\leq C^*h_K\Vert\nabla\psi\Vert_{L^2(w_K)}\quad \forall K\in\mathcal{T}^n_h,\\
\Vert \psi-C^n_h\psi\Vert_{L^2(e)}\leq C^*h_e^{\frac{1}{2}}\Vert\nabla\psi\Vert_{L^2(w_e)}\hspace{3cm} \forall e\in\mathcal{E}^n_h.\nonumber
\end{align}

We consider the following fully discrete numerical approximation of the Cahn-Hilliard equation \eqref{pb1}: set $u^0_h=P^0_hu^{\varepsilon}_0\in \mathbb{V}^0_h$
and for $n=1,\dots, N$ find $(u^n_h,w^n_h) \in \mathbb{V}^n_h\times \mathbb{V}^n_h$ as the solution of
 \begin{align}
  \label{scheme1}
\frac{1}{\tau_n}(u^n_h-u^{n-1}_h, \varphi_h)+(\nabla w^n_h, \nabla\varphi_h)&= \frac{1}{\tau_n}(\Delta_n\widetilde{W}, \varphi_h)& \varphi_h\in\mathbb{V}^n_h,\nonumber\\
\varepsilon(\nabla u^n_h, \nabla\psi_h)+\displaystyle\frac{1}{\varepsilon}(f(u^n_h), \psi_h)&=(w^n_h, \psi_h)& \psi_h\in\mathbb{V}^n_h,
\end{align}
where $\Delta_n\widetilde{W}$ denotes the time-increment of the noise \eqref{Noiseapprox2} on $(t_{n-1}, t_n)$, i.e.,
\begin{align*}
\Delta_n\widetilde{W}:= \widetilde{W}(t_n)- \widetilde{W}(t_{n-1}) = \Delta_n\widehat{W}-\frac{1}{\vert\mathcal{D}\vert}(\Delta_n\widehat{W}, 1).
\end{align*}
 
We define the piecewise  linear time interpolant $u_{h,\tau}$ of the numerical solution $\{u^n_h\}_{n=0}^N$ as:
\begin{align}
\label{Interpolant1}
u_{h,\tau}(t)=\frac{t-t_{n-1}}{\tau_n}u^n_h+\left(1-\frac{t-t_{n-1}}{\tau_n}\right)u^{n-1}_h\quad \text{for}\; t\in[t_{n-1}, t_n].
\end{align}
Analogously, we define  the piecewise  linear time interpolant  $w_{h,\tau}$ of the numerical solution $\{w^n_h\}_{n=0}^N$. 

The numerical solution $u^n_h$ can be expressed as $u^n_h=\widetilde{u}^n_h+\widehat{u}^n_h$, where $(\widetilde{u}^n_h, \widetilde{w}^n_h)$ solves:
\begin{align}
\label{scheme2}
\frac{1}{\tau_n}(\widetilde{u}^n_h-\widetilde{u}^{n-1}_h, \varphi_h)+(\nabla \widetilde{w}^n_h, \nabla\varphi_h)&=\frac{1}{\tau_n}(\Delta_n\widetilde{W}, \varphi_h) &\varphi_h\in\mathbb{V}^n_h,\nonumber\\
(\widetilde{w}^n_h, \psi_h)&=\varepsilon(\nabla \widetilde{u}^n_h, \nabla\psi_h)&\psi_h\in\mathbb{V}^n_h,\\
\widetilde{u}^0_h&=0,\nonumber
\end{align}
and $(\widehat{u}^n_h,\widehat{w}^n_h)$ solves:
\begin{align}
\label{scheme3}
\frac{1}{\tau_n}(\widehat{u}^n_h-\widehat{u}^{n-1}_h, \varphi_h)+(\nabla \widehat{w}^n_h, \nabla\varphi_h)&= 0 & \varphi_h\in\mathbb{V}^n_h,\nonumber\\
\varepsilon(\nabla \widehat{u}^n_h, \nabla\psi_h)+\displaystyle\frac{1}{\varepsilon}(f(u^n_h), \psi_h)&=(\widehat{w}^n_h, \psi_h) &\psi_h\in\mathbb{V}^n_h,\\
\widehat{u}^0_h=u^0_h&=P^0_hu^{\varepsilon}_0.\nonumber
\end{align}
Analogously to \eqref{Interpolant1}, we define the  interpolants $\widetilde{u}_{h,\tau}$, $\widetilde{w}_{h,\tau}$, $\widehat{u}_{h,\tau}$ and $\widehat{w}_{h,\tau}$ of the numerical solutions $\{\widetilde{u}^n_h\}_n$, $\{\widetilde{w}^n_h\}_n$, $\{\widehat{u}^n_h\}_n$ and $\{\widehat{w}^n_h\}_n$ respectively.

\section{Error estimate for the linear SPDE}
\label{ErrorlinearSPDE}

In this section we derive error estimates for the numerical approximation \eqref{scheme2} of \eqref{model2}.
To derive the error estimates we first consider the following approximation of \eqref{Eqy}:
\begin{equation}
\label{schemey}
\begin{array}{rclll}
\left(\frac{y^n_h-y^{n-1}_h}{\tau_n}, \varphi_h\right)+(\nabla\widetilde{w}^n_h, \nabla\varphi_h)&=&0 &\quad \forall \varphi_h\in\mathbb{V}_h^n,\\
(\widetilde{w}^n_h,\varphi_h)&=&\varepsilon(\nabla\widetilde{u}^n_h,\nabla\psi_h)& \quad \forall\psi_h\in \mathbb{V}^n_h,
\end{array}
\end{equation}
with $y^0_h = 0$ and $\{\widetilde{w}^n_h\}_{n=1}^N$ is the solution of \eqref{scheme2}.

In the following lemma, we derive a discrete analogue of the transformation \eqref{lineartransform}, which relates the solution of \eqref{schemey} to the solution of \eqref{scheme2}.
The lemma holds under an additional (mild) noise ''compatibility'' condition $\mathbb{V}_{\tilde{h}}\subset\mathbb{V}^n_h$ for all  $n=1,\cdots,N$, which is assumed to hold for the remainder of the paper.
The proof of the lemma follows as \cite[Lemma 3.1]{bw21} and \cite[Lemma 4.1]{BanasVieth_a_post23} and is therefore omitted. We note that the noise compatibility condition relaxes
the condition $\mathbb{V}^{n-1}_h \subset \mathbb{V}^n_h$ which was assumed in \cite{BanasVieth_a_post23}, cf. \cite[Remark 5.2]{BanasVieth_a_post23}.
\begin{lemma}
\label{LemmaLubo}
Suppose that  $\mathbb{V}_{\tilde{h}}\subset\mathbb{V}^n_h$ for all $n=1,\cdots,N$. Then it holds that:
\begin{align*}
 y^n_h=\widetilde{u}^n_h-\sum_{j=1}^n\int_{t_{j-1}}^{t_j}d\widetilde{W}(s). 
 \end{align*}
\end{lemma}

Similarly to \eqref{Interpolant1}, we define the  piecewise linear time interpolant $y_{h, \tau}$  of the numerical solution $(y^n_h)$.  It follows that:
\begin{align}
\label{DerivInterpoly}
\partial_ty_{h,\tau}(t)=\frac{y^n_h-y^{n-1}_h}{\tau_n}\quad \text{for} \quad t_{n-1}<t<t_n,\quad n=1,\cdots,N. 
\end{align}
It follows from \eqref{schemey} that $(y_{h,\tau}, \widetilde{w}_{h, \tau})$ satisfies:
\begin{equation}
\label{schemeintery}
\begin{array}{rclll}
(\partial_ty_{h,\tau}(t), \varphi)+(\nabla\widetilde{w}_{h,\tau}(t), \nabla\varphi)&=&\langle\mathcal{R}_y(t),\varphi\rangle &\quad \forall \varphi\in \mathbb{H}^1,\\
\varepsilon(\nabla\widetilde{u}_{h,\tau}(t), \nabla\psi)-(\widetilde{w}_{h,\tau}(t), \psi)&=&\langle\mathcal{S}_{y}(t), \psi\rangle &\quad \forall\psi\in \mathbb{H}^1, 
\end{array}
\end{equation}
with the residuals $\mathcal{R}_y(t)$, $\mathcal{S}_y(t)$ given as
\begin{align*}
\langle\mathcal{R}_y(t),\varphi\rangle=(\partial_ty_{h,\tau}(t), \varphi)+(\nabla\widetilde{w}_{h,\tau}(t), \nabla\varphi), & \quad\forall\varphi\in \mathbb{H}^1, \\
\langle\mathcal{S}_{y}(t), \psi\rangle=-(\widetilde{w}_{h,\tau}(t), \psi)+\varepsilon(\nabla\widetilde{u}_{h,\tau}(t), \nabla\psi) &\quad \forall\psi\in \mathbb{H}^1.
\end{align*}
We define the spatial error indicators $\eta^n_{\text{SPACE},i}$, for $i=1,2,3$, as follows
\begin{align*}
\eta^n_{\text{SPACE},1}&=\left(\sum_{K\in \mathcal{T}^n_h}h_K^2\Vert\tau_n^{-1}(y^n_h-y^{n-1}_h)\Vert^2_{L^2(K)}\right)^{1/2}+\left(\sum_{e\in \mathcal{E}^n_h}h_e\Vert[\nabla\widetilde{w}^n_h.\vec{n}_e]_e\Vert^2_{L^2(e)}\right)^{1/2},\\
\eta^n_{\text{SPACE},2}&=\left(\sum_{K\in\mathcal{T}^n_h}h_K^2\Vert\widetilde{w}^n_h\Vert^2_{L^2(K)}\right)^{1/2},\\
\eta^n_{\text{SPACE},3}&=\left(\varepsilon\sum_{e\in\mathcal{E}^n_h}h_e\Vert[\nabla\widetilde{u}^n_h.\vec{n}_e]_e\Vert^2_{L^2(e)}\right)^{1/2},
\end{align*}
where $[\nabla u.\vec{n}_e]:=\nabla u\vert_{K_1}.\vec{n}_1+\nabla u\vert_{K_2}.\vec{n}_2$ for $e=\bar{K}_1\cap\bar{K}_2$, and the vectors $\vec{n}_1$ and $\vec{n}_2$ are  the outer unit normals to the elements $K_1, K_2\in\mathcal{T}^n_h$ at $e\in\mathcal{E}^n_h$.
Furthermore, we define the time error indicators $\eta^n_{\text{TIME},i}$, $i=1,2,3$ as
\begin{align*}
\eta^n_{\text{TIME},1} & =\Vert\nabla[\widetilde{w}^{n-1}_h-\widetilde{w}^n_h]\Vert,
\\
\eta^n_{\text{TIME},2} & =\Vert\widetilde{w}^{n-1}_h-\widetilde{w}^n_h\Vert,
\\
\eta^n_{\text{TIME},3} & =\varepsilon\Vert\nabla[\widetilde{u}^{n-1}_h-\widetilde{u}^n_h]\Vert.
\end{align*}
To simplify the notation below we denote
\begin{align*}
\mu_{-1}(t) & =C^*\eta^n_{\text{SPACE},1}+\eta^n_{\text{TIME},1},
\\
\mu_0(t) & =\eta^n_{\text{TIME},2},
\\
\mu_1(t) & =\eta^n_{\text{TIME},3}+\eta^n_{\text{SPACE},2}+C^*\eta^n_{\text{SPACE},3},
\end{align*}
where $C^*>0$ is the constant from \eqref{ErrorCnh}.
\begin{lemma}
\label{Residuelestimate}
The following bounds on the residuals hold:
\begin{align*}
\langle\mathcal{R}_y(t), \varphi\rangle\leq \mu_{-1}(t)\Vert\nabla\varphi\Vert\quad \text{and}\quad 
\langle\mathcal{S}_y(t), \varphi\rangle\leq \mu_0(t)\Vert\varphi\Vert+\mu_1(t)\Vert\nabla\varphi\Vert.
\end{align*}
\end{lemma}
\begin{proof}
Using \eqref{schemey}, we can express $\mathcal{R}_y$ and $\mathcal{S}_y$  as follows:
\begin{align*}
\langle \mathcal{R}_y(t),\varphi\rangle=&(\partial_ty_{h,\tau}(t), \varphi-\varphi_h)+(\nabla\widetilde{w}_h^n, \nabla[\varphi-\varphi_h])+(\nabla[\widetilde{w}_{h,\tau}(t)-\widetilde{w}^n_h], \nabla\varphi), \\
\langle\mathcal{S}_y(t), \varphi\rangle=&(\widetilde{w}^n_h-\widetilde{w}_{h,\tau}(t), \varphi)+(\widetilde{w}^n_h, \varphi_h-\varphi)+\varepsilon(\nabla[\widetilde{u}_{h,\tau}(t)-\widetilde{u}^n_h], \nabla\varphi)\nonumber\\
&+\varepsilon(\nabla\widetilde{u}^n_h, \nabla[\varphi-\varphi_h]). 
\end{align*}
By setting $\varphi_h=C^n_h\varphi\in \mathbb{V}^n_h$, and applying element-wise integration by parts, together with \eqref{ErrorPnh} and \eqref{ErrorCnh}, as done in the proof of \cite[Proposition 6.3]{Bartels_book2015}, we obtain the desired results.
\end{proof}
To simplify the notation, we respectively denote the stochastic integral and its time-discrete counterpart by:
\begin{align}
\label{Sigma1}
\Sigma(t) & =\int_0^td\widetilde{W}(s),
\\
\label{Sigma1b}
\Sigma^n_{\tilde{h}} & =\sum_{i=1}^{n}\Delta_i\widetilde{W}=\int_0^{t_n}d\widetilde{W}(s).
\end{align}
Analogously to \eqref{Interpolant1}, we define the continuous piecewise linear time-interpolant of $\{\Sigma^n_{\tilde{h}}\}_{n=0}^N$ as follows:
\begin{align}
\label{Sigma2} 
\Sigma_{\tilde{h},\tau}(t)=\frac{t-t_{n-1}}{\tau_n}\Sigma^n_{\tilde{h}}+\frac{t_n-t}{\tau_n}\Sigma^{n-1}_{\tilde{h}}=\sum_{i=1}^{n-1}\Delta_i\widetilde{W}+\frac{t-t_{n-1}}{\tau_n}\Delta_n\widetilde{W},\quad t\in[t_{n-1}, t_n].
\end{align}

We recall in the following lemma some  basic properties of the nodal basis functions $(\phi_{\ell})_{\ell=1}^L$ of the finite element space $\mathbb{V}_{\tilde{h}}$ for a quasi-uniform triangulation, see, e.g., \cite[Chapter 3]{Bartels_book2015}. 
\begin{lemma}
\label{Lemmabasis}
The following properties hold  for all $\phi_{\ell}\in \mathbb{V}_{\tilde{h}}$, uniformly in $\tilde{h}$ and for all $\ell\in\{1, \cdots, L\}$:
\begin{itemize}
\item[(i)]    $C_1\tilde{h}^d\leq \vert(\phi_{\ell}, 1)\vert\leq C_2\tilde{h}^{d}$, $L=dim(\mathbb{V}_{\tilde{h}})\leq C\tilde{h}^{-d}$,
\item[(ii)] $\Vert \phi_{\ell}\Vert\leq C\tilde{h}^{\frac{d}{2}}$ and $\Vert\nabla \phi_{\ell}\Vert\leq C\tilde{h}^{-1}\Vert \phi_{\ell}\Vert$.
\end{itemize} 
\end{lemma} 

We define the noise error indicator as
\begin{align}
\label{Noise1}
\eta^n_{\text{NOISE}}&:=\tau_n^2\sum_{\ell=1}^L\frac{\Vert\nabla\phi_{\ell}\Vert^2}{(d+1)^{-1}\vert (\phi_{\ell}, 1)\vert}.
\end{align}
\begin{remark}
Using \lemref{Lemmabasis}, it can be shown that:
\begin{align*}
\eta^n_{\text{NOISE}}\leq C\tau_n^2\tilde{h}^{-2}L\leq C\tau_n^2\tilde{h}^{-2-d}.
\end{align*}
By choosing $\tau_n$ such that $\tau_n\leq C\tilde{h}^{2+d+\sigma}$ for some $\sigma>0$, it follows that $\eta^n_{\text{NOISE}}\leq C\tau_n\tilde{h}^{\sigma}$.
Hence, since $\tilde{h}$ is fixed, the size of the noise error indicator can be controlled by the time step size.
\end{remark}
The following lemma relates the noise error indicator (\ref{Noise1}) to the error due to the time-discretization of the noise.
\begin{lemma}
\label{ErrorNoiseLemma}
The following estimate holds:
\begin{align*}
\int_{0}^{T}\mathbb{E}[\Vert\nabla(\Sigma_{\tilde{h},\tau}(s)-\Sigma(s))\Vert^2]ds\leq C\sum_{n=1}^N\eta^n_{\text{NOISE}}.
\end{align*}
\end{lemma}
\begin{proof}
Using the definitions of $\Sigma_{\tilde{h},\tau}$ and $\Sigma$ (see \eqref{Sigma1} and \eqref{Sigma1b}), we obtain
\begin{align*}
\text{I}&:=\int_{0}^{T}\mathbb{E}[\Vert\nabla(\Sigma_{\tilde{h},\tau}(t)-\Sigma(t))\Vert^2]ds=\sum_{n=1}^N\int_{t_{n-1}}^{t_n}\mathbb{E}[\Vert\nabla(\Sigma_{\tilde{h},\tau}(t)-\Sigma(t))\Vert^2]dt\nonumber\\
&=\mathbb{E}\left[\sum_{n=1}^N\int_{t_{n-1}}^{t_n}\left\Vert\nabla\left(\int_0^td\widetilde{W}(s)-\frac{t-t_{n-1}}{\tau_n}\Delta_n\widetilde{W}-\sum_{i=1}^{n-1}\Delta_i\widetilde{W}\right)\right\Vert^2\right]dt\nonumber\\
&=\mathbb{E}\left[\sum_{n=1}^N\int_{t_{n-1}}^{t_n}\left\Vert\nabla\left(\int_{t_{n-1}}^td\widetilde{W}(s)-\frac{t-t_{n-1}}{\tau_n}\int_{t_{n-1}}^{t_n}d\widetilde{W}(s)\right)\right\Vert^2\right]dt.
\end{align*}
Using the It\^{o} isometry, the fact that  $\mathbb{E}[(\Delta_n\beta_{\ell})^2]=\tau_n$, and $\mathbb{E}[(\Delta_n\beta_{\ell})(\Delta_n\beta_k)]=0$ for $k\neq \ell$, we obtain
\begin{align*}
\text{I}&\leq  C\sum_{n=1}^N\int_{t_{n-1}}^{t_n}\mathbb{E}\left[\left\Vert\nabla\left(\int_{t_{n-1}}^td\widetilde{W}(s)\right)\right\Vert^2dt\right]+C\sum_{n=1}^N\int_{t_{n-1}}^{t_n}\mathbb{E}\left[\left\Vert\frac{t-t_{n-1}}{\tau_n}\nabla\left(\Delta_n\widetilde{W}\right)\right\Vert^2dt\right]\nonumber\\
&\leq  C\sum_{n=1}^N\tau_n\int_{t_{n-1}}^{t_n}\sum_{\ell=1}^L\frac{\Vert\nabla\phi_{\ell}\Vert^2}{(d+1)^{-1}\vert (\phi_{\ell},1)\vert}dt+C\sum_{n=1}^N\tau_n^2\sum_{\ell=1}^L\frac{\Vert\nabla\phi_{\ell}\Vert^2}{(d+1)^{-1}\vert (\phi_{\ell}, 1)\vert}\leq C\sum_{n=1}^N\eta^n_{\text{NOISE}}.
\end{align*}
\end{proof}

\begin{proposition}
\label{mainresult1}
 Let $y$ be given by \eqref{Eqy} and  $y_{h,\tau}$ be  the time-interpolant of the numerical solution $\{y^n_h\}_n$, given by \eqref{schemey}. Then, the following error estimate holds:
\begin{align*}
&\sup_{t\in[0, T]}\mathbb{E}[\Vert y_{h,\tau}(t)-y(t)\Vert^2_{\mathbb{H}^{-1}}]+\varepsilon\int_0^T\mathbb{E}[\Vert \nabla(y_{h,\tau}(s)-y(s))]\Vert^2]ds\nonumber\\
&\qquad\leq C\int_0^T\mathbb{E}\left[T\mu_{-1}^2(s)+\frac{T}{\varepsilon}\mu_0^2(s)+\varepsilon^{-1}\mu_1^2(s)\right]ds+C\varepsilon \sum_{n=1}^N\eta^n_{\text{NOISE}}.
\end{align*}
\end{proposition}
\begin{proof}
We subtract \eqref{Eqy} from \eqref{schemeintery} and take $\varphi=(-\Delta)^{-1}\left(y_{h,\tau}(t)-y(t)\right)$ to obtain
\begin{align*}
&\frac{1}{2}\frac{d}{dt}\Vert y_{h,\tau}(t)-y(t)\Vert^2_{\mathbb{H}^{-1}}+(\widetilde{w}_{h,\tau}(t)-\widetilde{w}(t), y_{h,\tau}(t)-y(t))\\
&\qquad=\langle\mathcal{R}_y(t), (-\Delta)^{-1}\left(y_{h,\tau}(t)-y(t)\right)\rangle.
\end{align*}
Subtracting \eqref{Eqy} from \eqref{schemeintery} and taking  $\psi=\widetilde{u}_{h,\tau}(t)-\widetilde{u}(t)$, yields:
\begin{align*}
&-(\widetilde{w}_{h,\tau}(t)-\widetilde{w}, \widetilde{u}_{h,\tau}-\widetilde{u}(t))=-\varepsilon\Vert\nabla\left(\widetilde{u}_{h,\tau}(t)-\widetilde{u}(t)\right)\Vert^2+\langle\mathcal{S}_y(t), \widetilde{u}_{h,\tau}(t)-\widetilde{u}(t)\rangle.
\end{align*}
By summing the two preceding identities, integrating the resulting equation over the interval $(0, t)$, noting that $y_{h,\tau}(0)=y(0)=0$ and then taking the expectation, we obtain:
\begin{align}
\label{errory1}
&\frac{1}{2}\mathbb{E}[\Vert y_{h,\tau}(t)-y(t)\Vert^2_{\mathbb{H}^{-1}}]+\varepsilon\int_0^t\mathbb{E}[\Vert\nabla\left(\widetilde{u}_{h,\tau}(s)-\widetilde{u}(s)\right)\Vert^2]ds\nonumber\\
&\quad=\int_0^t\mathbb{E}[(\widetilde{w}_{h,\tau}(s)-\widetilde{w}(s), \widetilde{u}_{h,\tau}(s)-\widetilde{u}(s))]ds\nonumber\\
&\quad\quad-\int_0^t\mathbb{E}[(\widetilde{w}_{h,\tau}(s)-\widetilde{w}(s), y_{h,\tau}(s)-y(s))]ds\\
&\quad\quad+\int_0^t\mathbb{E}[\langle\mathcal{R}_y(s), (-\Delta)^{-1}\left(y_{h,\tau}(s)-y(s)\right)\rangle]ds+\int_0^t\mathbb{E}[\langle\mathcal{S}_y(s), \widetilde{u}_{h,\tau}(s)-\widetilde{u}(s)\rangle]ds.\nonumber
\end{align}
Subtracting the second equation of \eqref{weak2} from the second equation of \eqref{schemeintery} yields:
\begin{align}
\label{Inter1}
(\widetilde{w}_{h,\tau}(t)-\widetilde{w}(t), \psi)=\varepsilon(\nabla\left(\widetilde{u}_{h, \tau}(t)-\widetilde{u}(t)\right), \nabla\psi)-\langle\mathcal{S}_y(t), \psi\rangle\quad \forall\psi\in\mathbb{H}^1. 
\end{align}
Taking $\psi=\widetilde{u}_{h,\tau}-\widetilde{u}$ in \eqref{Inter1} and substituting the resulting equation into \eqref{errory1} yields:
\begin{align}
\label{errory2}
\frac{1}{2}\mathbb{E}[\Vert y_{h,\tau}(t)-y(t)\Vert^2_{\mathbb{H}^{-1}}]=&-\int_0^t\mathbb{E}[(\widetilde{w}_{h,\tau}(s)-\widetilde{w}(s), y_{h,\tau}(s)-y(s))]ds\nonumber\\
&+\int_0^t\mathbb{E}[\langle\mathcal{R}_y(s), (-\Delta)^{-1}\left(y_{h,\tau}(s)-y(s)\right)\rangle]ds.
\end{align}
Taking $\psi=y_{h,\tau}-y$ in \eqref{Inter1} and recalling that $y_{h,\tau}(t)=\widetilde{u}_{h,\tau}(t)-\Sigma_{\tilde{h},\tau}(t)$, yields:
\begin{align}
\label{errory2a}
&(\widetilde{w}_{h,\tau}(t)-\widetilde{w}(t), y_{h,\tau}(t)-y(t))\nonumber\\
&\quad=\varepsilon\Vert\nabla\left(y_{h,\tau}(t)-y(t)\right)\Vert^2+\varepsilon\left(\nabla[\Sigma_{\tilde{h},\tau}(t)-\Sigma(t)], \nabla[y_{h,\tau}(t)-y(t)]\right)\\
&\quad\quad-\langle\mathcal{S}_y(t), y_{h,\tau}(t)-y(t)\rangle.\nonumber
\end{align}
Substituting \eqref{errory2a} into \eqref{errory2} leads to:
\begin{align*}
&\frac{1}{2}\mathbb{E}\left[\Vert y_{h,\tau}(t)-y(t)\Vert^2_{\mathbb{H}^{-1}}\right]+\varepsilon\int_0^t\mathbb{E}\left[\Vert\nabla\left(y_{h,\tau}(s)-y(s)\right)\Vert^2\right]ds\nonumber\\
&\quad=\varepsilon\int_0^t\mathbb{E}\left[(\nabla\left(\Sigma_{\tilde{h},\tau}(s)-\Sigma(s)\right), \nabla\left(y_{h,\tau}(s)-y(s)\right))\right]ds\nonumber\\
&\quad\quad+\int_0^t\mathbb{E}\left[\left\langle\mathcal{R}_y(s), (-\Delta)^{-1}[y_{h,\tau}(s)-y(s)\right]\right\rangle]ds+\int_0^t\mathbb{E}\left[\left\langle\mathcal{S}_y(s), y_{h,\tau}(s)-y(s)\right\rangle\right]ds.\nonumber\\
\end{align*}
Using \lemref{Residuelestimate}, we obtain:
\begin{align}
\label{errory3}
&\frac{1}{2}\mathbb{E}\left[\Vert y_{h,\tau}(t)-y(t)\Vert^2_{\mathbb{H}^{-1}}\right]+\varepsilon\int_0^t\mathbb{E}\left[\Vert\nabla\left(y_{h,\tau}(s)-y(s)\right)\Vert^2\right]ds\nonumber\\
&\quad\leq \varepsilon\int_0^t\mathbb{E}\left[(\nabla\left(\Sigma_{\tilde{h},\tau}(s)-\Sigma(s)\right), \nabla\left(y_{h,\tau}(s)-y(s)\right))\right]ds\\
&\quad\quad+\int_0^t\mathbb{E}\left[\mu_0(s)\Vert y_{h,\tau}(s)-y(s)\Vert\right] ds+\int_0^t\mathbb{E}\left[\mu_{-1}(s)\Vert y_{h,\tau}(s)-y(s)\Vert_{\mathbb{H}^{-1}}\right]ds\nonumber\\
&\quad\quad+\int_0^t\mathbb{E}\left[\mu_1(s)\Vert\nabla\left(y_{h,\tau}(s)-y(s)\right)\Vert\right]ds\nonumber\\
&\quad=: \text{II}_1+\text{II}_2+\text{II}_3+\text{II}_4. \nonumber
\end{align}
Using Cauchy-Schwarz's inequality, Young's inequality, and \lemref{ErrorNoiseLemma}, we obtain:
\begin{align}
\label{EstimationI1}
\text{II}_1&\leq C\varepsilon\int_0^t\mathbb{E}[\Vert\nabla[\Sigma_{\tilde{h},\tau}(s)-\Sigma(s)]\Vert^2]ds+\frac{\varepsilon}{8}\int_0^t\mathbb{E}[\Vert\nabla\left(y_{h,\tau}(s)-y(s)\right)\Vert^2]ds\nonumber\\
&\leq C\varepsilon\sum_{n=1}^N\eta^n_{\text{NOISE}}+\frac{\varepsilon}{8}\int_0^t\mathbb{E}[\Vert\nabla\left(y_{h,\tau}(s)-y(s)\right)\Vert^2]ds.
\end{align}
Using Young's inequality, we estimate $\text{II}_2$ and $\text{II}_4$ as follows: 
\begin{align}
\label{EstimationI2}
\text{II}_2&\leq 2T\int_0^t\mathbb{E}\left[\mu_{-1}^2(s)\right]ds+\frac{1}{8}\sup_{s\in[0, t]}\mathbb{E}\left[\Vert y_{h,\tau}-y(s)\Vert^2_{\mathbb{H}^{-1}}\right],\\
\label{EstimationI4}
\text{II}_4&\leq 2\varepsilon^{-1}\int_0^t\mathbb{E}\left[\mu_1^2(s)\right]ds+\frac{\varepsilon}{8}\int_0^t\mathbb{E}\left[\Vert\nabla\left(y_{h,\tau}(s)-y(s)\right)\Vert^2\right]ds. 
\end{align}
Using the interpolation inequality $\Vert u\Vert^2_{\mathbb{L}^2}\leq \Vert u\Vert_{\mathbb{H}^{-1}}\Vert\nabla u\Vert_{\mathbb{L}^2}$ and Young's inequality, we obtain:
\begin{align}
\label{EstimationI3}
\text{II}_3&\leq C\sqrt{\frac{T}{\varepsilon}}\int_0^t\mathbb{E}[\mu_0^2(s)]ds+\frac{\varepsilon}{8}\int_0^t\mathbb{E}[\Vert\nabla\left(y_{h,\tau}(s)-y(s)\right)\Vert^2]ds\\
&\quad+\frac{1}{8}\sup_{s\in[0, t]}\mathbb{E}\left[\Vert y_{h,\tau}(s)-y(s)\Vert^2_{\mathbb{H}^{-1}}\right]. \nonumber
\end{align}
Substituting \eqref{EstimationI1}, \eqref{EstimationI2}, \eqref{EstimationI3} and \eqref{EstimationI4} into \eqref{errory3}
 completes the proof. 
\end{proof}

\begin{corollary}
\label{Corollary1}
 Let $y$ be given by \eqref{lineartransform}, and let $y_{h,\tau}$ be the time-interpolant of the numerical solution $\{y^n_h\}_n$ satisfying \eqref{schemey}. The following estimate holds:
\begin{align*}
&\mathbb{E}\left[\sup_{t\in[0, T]}\Vert y_{h,\tau}(t)-y(t)\Vert^2_{\mathbb{H}^{-1}}\right]+\varepsilon\mathbb{E}\left[\int_0^T\Vert \nabla(y_{h,\tau}(s)-y(s))]\Vert^2ds\right]\nonumber\\
&\qquad\leq C\int_0^T\mathbb{E}\left[T\mu_{-1}^2(s)+\frac{T}{\varepsilon}\mu_0^2(s)+\varepsilon^{-1}\mu_1^2(s)\right]ds+C\varepsilon\sum_{n=1}^N\eta^n_{\text{NOISE}}.
\end{align*}
\end{corollary}
\begin{proof}
The proof follows along the same lines as that of \propref{mainresult1}. We first take the supremum on $[0, T]$ and then the expectation and obtain (cf. \eqref{errory3})
\begin{align*}
&\frac{1}{2}\mathbb{E}\left[\sup_{t\in[0,T]}\Vert y_{h,\tau}(t)-y(t)\Vert^2_{\mathbb{H}^{-1}}\right]+\varepsilon\mathbb{E}\left[\int_0^T\Vert\nabla\left(y_{h,\tau}(s)-y(s)\right)\Vert^2ds\right]\nonumber\\
&\qquad\leq \varepsilon\mathbb{E}\left[\int_0^T\left\vert(\nabla\left(\Sigma_{h,\tau}(s)-\Sigma(s)\right), \nabla\left(y_{h,\tau}(s)-y(s)\right))\right\vert ds\right]\nonumber\\
&\qquad\qquad+\mathbb{E}\left[\int_0^T\vert\mu_0(s)\vert\Vert y_{h,\tau}(s)-y(s)\Vert  ds\right]+\mathbb{E}\left[\int_0^T\vert\mu_{-1}(s)\vert\Vert y_{h,\tau}(s)-y(s)\Vert_{\mathbb{H}^{-1}}ds\right]\nonumber\\
&\qquad\qquad+\mathbb{E}\left[\int_0^T\vert\mu_1(s)\vert\Vert\nabla\left(y_{h,\tau}(s)-y(s)\right)\Vert ds\right].\nonumber
\end{align*}
The rest of the proof follows the same lines as that of \propref{mainresult1}.
\end{proof}

The following lemma provides an estimate of the error  $\widetilde{u}(t)-\widetilde{u}_{h,\tau}$ in the $L^{\infty}(0, T; L^2(\Omega, \mathbb{H}^{-1}))$-norm.
\begin{lemma}
 Let $\widetilde{u}$ be the solution to \eqref{weak1}, and let  $\widetilde{u}_{h,\tau}$  be the time-interpolant of the numerical solution of \eqref{scheme2}. The following  error estimate holds:
\begin{align*}
&\sup_{t\in[0, T]}\mathbb{E}\left[\Vert \widetilde{u}_{h,\tau}(t)-\widetilde{u}(t)\Vert^2_{\mathbb{H}^{-1}}\right]+\varepsilon\int_0^T\mathbb{E}[\Vert\nabla\left(\widetilde{u}_{h,\tau}(s)-\widetilde{u}(s)\right)\Vert^2]ds\nonumber\\
&\quad\leq C\varepsilon\sum_{n=1}^N\eta^n_{\text{NOISE}}+C\int_0^T\mathbb{E}\left[T\mu^2_{-1}(s)+\sqrt{\frac{T}{\varepsilon}}\mu_0^2(s)+\varepsilon^{-1}\mu_1^2(s)\right]ds\nonumber\\
&\qquad+ C\max_{n=1,\cdots,N}\left(\mathbb{E}[\Vert \widetilde{u}^{n-1}_h-\widetilde{u}^n_h\Vert^2_{\mathbb{H}^{-1}}+\mathbb{E}[\Vert y^{n-1}_h-y^n_h\Vert^2_{\mathbb{H}^{-1}}]\right)+C\tau\sum_{\ell=1}^L\frac{\Vert\phi_{\ell}-m(\phi_{\ell})\Vert^2_{\mathbb{H}^{-1}}}{(d+1)^{-1}\vert(\phi_{\ell}, 1)\vert}\nonumber\\
&\qquad+C\varepsilon\sum_{n=1}^N\tau_n\left(\mathbb{E}\left[\left\Vert\nabla\left(\widetilde{u}^{n-1}_h-\widetilde{u}^n_h\right)\right\Vert^2\right]+\mathbb{E}\left[\left\Vert \nabla(y^{n-1}_h-y^n_h)\right\Vert^2\right]\right).
\end{align*}
\end{lemma}
\begin{proof}
Recalling that $\widetilde{u}(t)=y(t)+\int_0^td\widetilde{W}(s)$ and using the triangle inequality, we obtain:
\begin{align*}
\Vert\widetilde{u}(t)-\widetilde{u}_{h,\tau}(t)\Vert^2_{\mathbb{H}^{-1}}&=\left\Vert y(t)+\int_0^td\widetilde{W}(s)-y_{h,\tau}(t)+y_{h,\tau}(t)-\widetilde{u}_{h,\tau}(t)\right\Vert^2_{\mathbb{H}^{-1}}\nonumber\\
&\leq2 \Vert y(t)-y_{h,\tau}(t)\Vert^2_{\mathbb{H}^{-1}}+2\left\Vert y_{h,\tau}(t)+\int_0^td\widetilde{W}(s)-\widetilde{u}_{h,\tau}(t)\right\Vert^2_{\mathbb{H}^{-1}}.
\end{align*}
A similar estimate  for $\int_0^t\Vert\nabla(\widetilde{u}_{h,\tau}(s)-\widetilde{u}(s))\Vert^2]ds$ holds. Consequently, we have:
\begin{align}
\label{Errorutilde1}
&\sup_{t\in[0, T]}\mathbb{E}\left[\Vert \widetilde{u}_{h,\tau}(t)-\widetilde{u}(t)\Vert^2_{\mathbb{H}^{-1}}\right]+\varepsilon\int_0^T\mathbb{E}\left[\Vert\nabla(\widetilde{u}_{h,\tau}(s)-\widetilde{u}(s))\Vert^2\right]ds\nonumber\\
&\quad\leq 2\sup_{t\in[0, T]}\mathbb{E}\left[\Vert y_{h,\tau}(t)-y(t)\Vert^2_{\mathbb{H}^{-1}}\right]+2\sup_{t\in[0, T]}\mathbb{E}\left[\left\Vert y_{h,\tau}(t)+\int_0^td\widetilde{W}(s)-\widetilde{u}_{h,\tau}(t)\right\Vert^2_{\mathbb{H}^{-1}}\right]\\
&\quad\quad+2\varepsilon\int_0^T\mathbb{E}[\Vert\nabla(y_{h,\tau}(t)-y(t))\Vert^2]dt\nonumber\\
&\quad\quad+2\varepsilon\int_0^T\mathbb{E}\left[\left\Vert\nabla\left(y_{h,\tau}(t)+\int_0^td\widetilde{W}(s)-\widetilde{u}_{h,\tau}(t)\right)\right\Vert^2\right]dt\nonumber\\
&\quad=:\text{III}_1+\text{III}_2+\text{III}_3+\text{III}_4. \nonumber
\end{align}
The terms $\text{III}_1$ and $\text{III}_3$ are estimated in \propref{mainresult1}. It remains to estimate $\text{III}_2$ and $\text{III}_4$.  Using the triangle inequality,  we have:
\begin{align}
\label{EstimateII2}
\text{III}_2 &\leq \max_{n=1,\cdots,N}\sup_{t\in[t_{n-1}, t_n]}\left\{\mathbb{E}\left[\Vert \widetilde{u}_{h,\tau}(t)-\widetilde{u}^n_h\Vert^2_{\mathbb{H}^{-1}}\right]+\mathbb{E}\left[\Vert y_{h,\tau}(t)-y^n_h\Vert^2_{\mathbb{H}^{-1}}\right]\right\}\nonumber\\
&\qquad\max_{n=1,\cdots,N}\sup_{t\in[t_{n-1}, t_n]}\mathbb{E}\left[\left\Vert\sum_{j=1}^n\int_{t_{j-1}}^{t_j}d\widetilde{W}(s)-\int_0^td\widetilde{W}(s) \right\Vert^2_{\mathbb{H}^{-1}}\right]\nonumber\\
&\qquad+\max_{n=1,\cdots,N}\mathbb{E}\left[\left\Vert \widetilde{u}^n_h-y^n_h-\sum_{j=1}^n\int_{t_{j-1}}^{t_j}d\widetilde{W}(s)\right\Vert^2_{\mathbb{H}^{-1}}\right]\\
&=: \text{III}_{2,1}+\text{III}_{2,2}+\text{III}_{2,3}+\text{III}_{2,4}. \nonumber
\end{align}
From \lemref{LemmaLubo}, we have $\text{III}_{2,4}=0$. Using the definitions of $\widetilde{u}_{h,\tau}$ and $y_{h, \tau}$, we obtain:
\begin{align}
\label{EstimateII21}
\text{III}_{2,1}\leq \max_{n=1,\cdots,N}\mathbb{E}\left[\Vert \widetilde{u}^{n-1}_h-\widetilde{u}^n_h\Vert^2_{\mathbb{H}^{-1}}\right]\quad \text{and}\quad \text{III}_{2,2}\leq \max_{n=1,\cdots,N}\mathbb{E}\left[\Vert y^{n-1}_h-y^n_h\Vert^2_{\mathbb{H}^{-1}}\right].
\end{align}
Next, 
using the It\^{o} isometry, we estimate
\begin{align}
\label{EstimateII23}
\text{III}_{2,3}&=\max_{n=1,\cdots,N}\sup_{t\in[t_{n-1}, t_n]}\mathbb{E}\left[\left\Vert\int_0^{t_n}d\widetilde{W}(s)-\int_0^td\widetilde{W}(s)\right\Vert^2_{\mathbb{H}^{-1}}\right]\nonumber\\
&=\max_{n=1,\cdots,N}\sup_{t\in[t_{n-1}, t_n]}\mathbb{E}\left[\left\Vert \int_t^{t_n}d\widetilde{W}(s)\right\Vert^2_{\mathbb{H}^{-1}}\right]\\
&\leq C\max_{n=1,\cdots,N}\sup_{t\in[t_{n-1}, t_n]}(t_n-t)\sum_{\ell=1}^L\frac{\Vert\phi_{\ell}-m(\phi_{\ell})\Vert^2_{\mathbb{H}^{-1}}}{(d+1)^{-1}\vert (\phi_{\ell},1)\vert}=C\tau\sum_{\ell=1}^L\frac{\Vert\phi_{\ell}-m(\phi_{\ell})\Vert^2_{\mathbb{H}^{-1}}}{(d+1)^{-1}\vert(\phi_{\ell}, 1)\vert}\nonumber. 
\end{align}
Substituting \eqref{EstimateII23} and \eqref{EstimateII21} into \eqref{EstimateII2} yields
\begin{align}
\label{EstimateII2a}
\text{III}_2&\leq C\max_{n=1,\cdots,N}\mathbb{E}\left[\Vert \widetilde{u}^{n-1}_h-\widetilde{u}^n_h\Vert^2_{\mathbb{H}^{-1}}\right]+C\max_{n=1,\cdots,N}\mathbb{E}\left[\Vert y^{n-1}_h-y^n_h\Vert^2_{\mathbb{H}^{-1}}\right]\nonumber\\
&\qquad+ C\tau\sum_{\ell=1}^L\frac{\Vert\phi_{\ell}-m(\phi_{\ell})\Vert^2_{\mathbb{H}^{-1}}}{(d+1)^{-1}\vert(\phi_{\ell}, 1)\vert}.
\end{align}
We can rewrite $\text{III}_4$ as follows:
\begin{align*}
\text{III}_4\leq 2\varepsilon\sum_{n=1}^N\int_{t_{n-1}}^{t_n}\mathbb{E}\left[\left\Vert\nabla\left(y_{h,\tau}(t)+\int_0^td\widetilde{W}(s)-\widetilde{u}_{h,\tau}(t)\right)\right\Vert^2\right]dt=:2\varepsilon\sum_{n=1}^N\int_{t_{n-1}}^{t_n}\text{III}_4^n.
\end{align*}
Along the same lines as the estimate of $\text{III}_2$, one obtains the following estimate:
\begin{align*}
\text{III}_4^n \leq C\tau_n\mathbb{E}\left[\Vert\nabla(\widetilde{u}^{n-1}_h-\widetilde{u}^n_h)\Vert^2\right]+C\tau_n\mathbb{E}\left[\Vert\nabla(y^{n-1}_h-y^n_h)\Vert^2\right]+C\tau_n\sum_{\ell=1}^L\frac{\Vert\nabla\phi_{\ell}\Vert^2}{(d+1)^{-1}\vert (\phi_{\ell}, 1)\vert}.
\end{align*}
We therefore obtain the following estimate for $\text{III}_4$:
\begin{align}
\label{EstimateII4}
\text{III}_4&\leq C\varepsilon\sum_{n=1}^N\tau_n\mathbb{E}\left[\Vert\nabla(\widetilde{u}^{n-1}_h-\widetilde{u}^n_h)\Vert^2\right]+C\varepsilon\sum_{n=1}^N\tau_n\mathbb{E}\left[\Vert\nabla(y^{n-1}_h-y^n_h)\Vert^2\right]\nonumber\\
&\qquad+C\varepsilon\sum_{n=1}^N\tau_n^2\sum_{\ell=1}^L\frac{\Vert\nabla\phi_{\ell}\Vert^2}{(d+1)^{-1}\vert (\phi_{\ell}, 1)\vert}. 
\end{align}
Substituting \eqref{EstimateII4} and \eqref{EstimateII2a} into \eqref{Errorutilde1} and using \propref{mainresult1} completes the proof. 
\end{proof}

The next lemma provides an estimate for the error $\widetilde{u}_{h,\tau}-\widetilde{u}$ in the $L^{2}(\Omega; L^{\infty}(0, T;  \mathbb{H}^{-1}))$-norm.

\begin{lemma}
\label{LemmaErrortilde}
 Let $\widetilde{u}$ be the solution to \eqref{weak1}, and let  $\widetilde{u}_{h,\tau}$  be the time-interpolant of the numerical solution of \eqref{scheme2}. The following  error estimate holds:
\begin{align*}
&\mathbb{E}\left[\sup_{t\in[0, T]}\Vert\widetilde{u}_{h,\tau}(t)-\widetilde{u}(t)\Vert^2_{\mathbb{H}^{-1}}\right]+\varepsilon\int_0^T\mathbb{E}\left[\Vert\nabla(\widetilde{u}_{h,\tau}(t)-\widetilde{u}(t))\Vert^2\right]dt\nonumber\\
&\leq C\varepsilon\sum_{n=1}^N\eta^n_{\text{NOISE}}+C\tau\sum_{l=1}^L\frac{\Vert\phi_l\Vert^2_{\mathbb{H}^{-1}}}{(d+1)^{-1}\vert(\phi_l, 1)\vert}+C\mathbb{E}\left[\max_{n=1,\cdots,N}\Vert\widetilde{u}^{n-1}_h-\widetilde{u}^n_h\Vert^2_{\mathbb{H}^{-1}}\right]\nonumber\\
&\;\;+C\mathbb{E}\left[\max_{n=1,\cdots,N}\Vert y^{n-1}_h-y^n_h\Vert^2_{\mathbb{H}^{-1}}\right]+C\varepsilon\sum_{n=1}^N\tau_n\left(\mathbb{E}[\Vert\nabla(\widetilde{u}^{n-1}_h-\widetilde{u}^n_h)\Vert^2]+\mathbb{E}[\Vert\nabla(y^{n-1}_h-y^n_h)\Vert^2]\right)
\nonumber\\
&\;\;+C\int_0^T\mathbb{E}\left[T\mu_{-1}^2(s)+\sqrt{\frac{T}{\varepsilon}}\mu_0^2(s)+\varepsilon^{-1}\mu_1^2(s)\right]ds+C_p\tau^{2\lambda}\left(\sum_{\ell=1}^L\frac{\Vert\phi_{\ell}-m(\phi_{\ell})\Vert^a_{\mathbb{H}^{-1}}}{(d+1)^{-1}\vert (\phi_{\ell}, 1)\vert}\right)^{\frac{2}{a}}\nonumber,
\end{align*}
for any $\lambda=q-\frac{1}{p}$, with $a, p\in(2, \infty)$, $a\geq p$, $q>\frac{1}{p}$, such that $\frac{1}{p}+q<\frac{1}{2}-\frac{1}{a}$. 
\end{lemma}

\begin{proof}
Using the identity $\widetilde{u}(t)=y(t)+\int_0^td\widetilde{W}(s)$ and the triangle inequality, we obtain:
{\small
\begin{align}
\label{ErrorII}
\mathbb{E}\left[\sup_{t\in[0, T]}\Vert\widetilde{u}(t)-\widetilde{u}_{h,\tau}(t)\Vert^2_{\mathbb{H}^{-1}}\right]&=\mathbb{E}\left[\sup_{t\in[0, T]}\left\Vert y(t)+\int_0^td\widetilde{W}(s)-y_{h,\tau}(t)+y_{h,\tau}(t)-\widetilde{u}_{h,\tau}(t)\right\Vert^2_{\mathbb{H}^{-1}}\right]\nonumber\\
&\leq 2 \mathbb{E}\left[\sup_{t\in[0,T]}\Vert y(t)-y_{h,\tau}(t)\Vert^2_{\mathbb{H}^{-1}}\right]\\
&\;\;+2\mathbb{E}\left[\sup_{t\in[0,T]}\left\Vert y_{h,\tau}(t)+\int_0^td\widetilde{W}(s)-\widetilde{u}_{h,\tau}(t)\right\Vert^2_{\mathbb{H}^{-1}}\right]=: 2\text{IV}_1+2\text{IV}_2.\nonumber
\end{align}
}
An estimate of $\text{IV}_1$ can be found in \coref{Corollary1}.  By the triangle inequality, we obtain:
\begin{align}
\label{ErrorII2}
\text{IV}_2&\leq \mathbb{E}\left[\max_{n=1,\cdots,N}\sup_{t\in[t_{n-1}, t_n]}\Vert \widetilde{u}_{h,\tau}(t)-\widetilde{u}^n_h\Vert^2_{\mathbb{H}^{-1}}\right]+\mathbb{E}\left[\max_{n=1,\cdots,N}\sup_{t\in[t_{n-1}, t_n]}\Vert y_{h,\tau}(t)-y^n_h\Vert^2_{\mathbb{H}^{-1}}\right]\nonumber\\
&\qquad+\mathbb{E}\left[\max_{n=1,\cdots,N}\left\Vert\sum_{j=1}^n\int_{t_{j-1}}^{t_j}d\widetilde{W}(s)-\int_0^td\widetilde{W}(s) \right\Vert^2_{\mathbb{H}^{-1}}\right]\\
&\qquad+\mathbb{E}\left[\max_{n=1,\cdots,N}\left\Vert \widetilde{u}^n_h-y^n_h-\sum_{j=1}^n\int_{t_{j-1}}^{t_j}d\widetilde{W}(s)\right\Vert^2_{\mathbb{H}^{-1}}\right]\nonumber\\
&=:\text{IV}_{2,1}+\text{IV}_{2,2}+\text{IV}_{2,3}+\text{IV}_{2,4}. \nonumber
\end{align}
From \lemref{LemmaLubo}, we have $\text{IV}_{2,4}=0$. Using the definitions of  $\widetilde{u}_{h,\tau}$ and $y_{h, \tau}$, we  obtain:
\begin{align}
\label{ErrorII212}
\text{IV}_{2,1}\leq \mathbb{E}\left[\max_{n=1,\cdots,N}\Vert \widetilde{u}^{n-1}_h-\widetilde{u}^n_h\Vert^2_{\mathbb{H}^{-1}}\right]\quad \text{and}\quad \text{IV}_{2,2}\leq \mathbb{E}\left[\max_{n=1,\cdots,N}\Vert y^{n-1}_h-y^n_h\Vert^2_{\mathbb{H}^{-1}}\right].
\end{align}
The term $\text{IV}_{2,3}$ can be estimated along the same lines as the term $\text{I}_{1,6}$ in the proof of \cite[Lemma 5.7]{BanasVieth_a_post23}:
\begin{align}
\label{ErrorII23}
\text{IV}_{2,3}\leq C\tau^{p\lambda}\left(\sum_{\ell=1}^L\frac{\Vert \phi_{\ell}-m(\phi_{\ell})\Vert^a_{\mathbb{H}^{-1}}}{(d+1)^{-1}\vert (\phi_{\ell}, 1)\vert}\right)^{\frac{p}{a}} 
\end{align}
for $\lambda=k-\frac{1}{p} > 0$  where  $k>\frac{1}{p}$ and $\frac{1}{p}+k<\frac{1}{2}-\frac{1}{a}$ for some $a, p\in(2, \infty)$, $a\geq p$.

Recalling that $\text{IV}_{2,4}=0$, substituting  \eqref{ErrorII23} and \eqref{ErrorII212} into \eqref{ErrorII2},
yields an estimate of $\text{IV}_2$. The term $\text{IV}_1$ is estimated in \coref{Corollary1}.
By combining these estimates we bound \eqref{ErrorII} as:
\begin{align}
\label{ErrorD}
\mathbb{E}\left[\sup_{t\in[0, T]}\Vert \widetilde{u}_{h,\tau}(t)-\widetilde{u}(t)\Vert^2_{\mathbb{H}^{-1}}\right]
&\leq C\varepsilon\sum_{n=1}^N\mathbb{E}[\eta^n_{\text{NOISE}}]+C\tau^{2\lambda}\left(\sum_{\ell=1}^L\frac{\Vert\phi_{\ell}-m(\phi_{\ell})\Vert^a_{\mathbb{H}^{-1}}}{(d+1)^{-1}\vert (\phi_{\ell}, 1)\vert}\right)^{\frac{2}{a}}\nonumber\\
&\quad+C\mathbb{E}\left[\max_{n=1,\cdots,N}\Vert \widetilde{u}^{n-1}_h-\widetilde{u}^n_h\Vert^2_{\mathbb{H}^{-1}}+\max_{n=1,\cdots,N}\Vert y^{n-1}_h-y^n_h\Vert^2_{\mathbb{H}^{-1}}\right]\\
&\quad+C\int_0^T\mathbb{E}\left[T\mu_{-1}^2(s)+\frac{T}{\varepsilon}\mu_0^2(s)+\varepsilon^{-1}\mu_1^2(s)\right]ds.\nonumber
\end{align}
Using the triangle inequality and the  inequality $(a+b)^2\leq 2a^2+2b^2$, yields
\begin{align}
\label{EstimateIII}
\varepsilon\int_0^T\mathbb{E}\left[\left\Vert\nabla\left(\widetilde{u}_{h,\tau}(t)-\widetilde{u}(t)\right)\right\Vert^2\right]dt&\leq 2\varepsilon\int_0^T\mathbb{E}\left[\left\Vert\nabla\left(y_{h,\tau}(t)-y(t)\right)\right\Vert^2\right]dt\nonumber\\
&\quad+2\varepsilon\int_0^T\mathbb{E}\left[\left\Vert\nabla\left(y_{h,\tau}(t)+\int_0^td\widetilde{W}(s)-\widetilde{u}_{h,\tau}(t)\right)\right\Vert^2\right]dt\nonumber\\
&=:\text{V}_1+\text{V}_2.
\end{align}
The term $\text{V}_2$ is the same as  $\text{III}_2$ in \eqref{Errorutilde1}. Hence, from \eqref{EstimateII4} we have:
\begin{align}
\label{EstimateIII2}
\text{V}_2
&\leq C\varepsilon\sum_{n=1}^N\tau_n\left(\mathbb{E}\left[\left\Vert\nabla(\widetilde{u}^{n-1}_h-\widetilde{u}^n_h)\right\Vert^2\right]+\mathbb{E}\left[\left\Vert\nabla(y^{n-1}_h-y^n_h)\right\Vert^2\right]\right)\\
&\qquad+C\varepsilon\sum_{n=1}^N\tau_n^2\sum_{\ell=1}^L\frac{\Vert\nabla \phi_{\ell}\Vert^2}{(d+1)^{-1}\vert (\phi_{\ell}, 1)\vert}. \nonumber
\end{align}
Substituting \eqref{EstimateIII2} and the estimate of $\text{V}_1$ (obtained from \coref{Corollary1}) into \eqref{EstimateIII}, we deduce the following estimate:
\begin{align}
\label{EstiIV}
\varepsilon\int_0^T\mathbb{E}\left[\left\Vert\nabla\left(\widetilde{u}_{h,\tau}(t)-\widetilde{u}(t)\right)\right\Vert^2\right]dt &\leq  C\varepsilon\sum_{n=1}^N\tau_n\mathbb{E}\left[\left\Vert\nabla(\widetilde{u}^{n-1}_h-\widetilde{u}^n_h)\right\Vert^2\right]+C\varepsilon\sum_{n=1}^N\eta^n_{\text{NOISE}}\nonumber\\
&\qquad+ C\varepsilon\sum_{n=1}^N\tau_n\mathbb{E}\left[\left\Vert\nabla(y^{n-1}_h-y^n_h)\right\Vert^2\right]\\
&\qquad +C\int_0^T\mathbb{E}\left[T\mu_{-1}^2(t)+\frac{T}{\varepsilon}\mu_0^2(t)+\varepsilon^{-1}\mu_1^2(t)\right]dt.\nonumber
\end{align} 
Collecting the estimates \eqref{EstiIV} and \eqref{ErrorD} concludes the proof.
\end{proof}

\section{Error estimate for the  random PDE}
\label{ErrorRandonPDE}
In this section we derive an a posteriori error estimate for the random PDE \eqref{model3}.
The analysis below follows roughly along the lines of \cite[Section 6]{BanasVieth_a_post23}, with several crucial modifications.
In particular, \lemref{LemmaNormL3} is necessary to compensate the lack of $\tilde{h}$-independent $\mathbb{H}^1$-energy bound as well as to avoid the restriction \cite[eq. (37)]{BanasVieth_a_post23}
in spatial dimension $d=3$.

We consider weak formulation of \eqref{model3} as
\begin{align}
\label{Weak1}
\langle \partial_t\widehat{u}(t), \varphi\rangle +(\nabla\widehat{w}(t), \nabla\varphi)=0 &\quad \forall\varphi\in\mathbb{H}^1,\\
\varepsilon(\nabla\widehat{u}(t), \nabla\psi)+\frac{1}{\varepsilon}(f(u(t)), \psi)-(\widehat{w}(t), \psi)=0&\quad \forall\psi\in\mathbb{H}^1.\nonumber
\end{align}
Let us recall that from the definition of the  time interpolant $\widehat{u}_{h,\tau}$ it holds that:
\begin{align*}
\partial_t\widehat{u}_{h,\tau}(t)=\frac{\widehat{u}^n_h-\widehat{u}^{n-1}_h}{\tau_n}\quad \text{for}\; t\in (t_{n-1}, t_n). 
\end{align*}
It follows from \eqref{scheme3} that the time interpolants $\widehat{u}_{h,\tau}$ and $\widehat{w}_{h,\tau}$ satisfy the following:
\begin{align}
\label{Weak1a}
\langle\partial_t\widehat{u}_{h,\tau}(t), \varphi\rangle + (\nabla\widehat{w}_{h,\tau}(t), \nabla\varphi)=\langle\widehat{\mathcal{R}}(t),\varphi\rangle &\quad \forall\varphi\in\mathbb{H}^1,\\
\varepsilon(\nabla\widehat{u}_{h,\tau}(t),\nabla\psi)+\frac{1}{\varepsilon}(f(u_{h,\tau}(t)),\psi)-(\widehat{w}_{h,\tau}(t), \psi)=\langle\widehat{\mathcal{S}}(t),\psi)&\quad \forall\psi\in\mathbb{H}^1,\nonumber
\end{align}
where the residuals $\widehat{\mathcal{R}}(t)$ and $\widehat{\mathcal{S}}(t)$ are defined for $t\in(0, T]$ as follows:
\begin{align*}
\langle\widehat{\mathcal{R}}(t),\varphi\rangle&=(\partial_t\widehat{u}_{h,\tau}(t), \varphi)+(\nabla\widehat{w}_{h,\tau}(t), \nabla\varphi)&\forall\varphi\in\mathbb{H}^1,\\
\langle\widehat{\mathcal{S}}(t),\psi\rangle&=-(\widehat{w}_{h,\tau}(t), \psi)+\varepsilon(\nabla\widehat{u}_{h,\tau}(t),\nabla\psi)+\frac{1}{\varepsilon}(f(u_{h,\tau}(t)),\psi)&\forall\psi\in\mathbb{H}^1.
\end{align*}
We define the space indicator errors $\eta^n_{\text{SPACE},i}$, for $i=4,5,6$,  as follows:
\begin{align*}
\eta^n_{\text{SPACE},4}&:=\left(\sum_{K\in\mathcal{T}^n_h}h_K^2\tau_n^{-2}\Vert\widehat{u}^n_h-\widehat{u}^{n-1}_h\Vert^2_{L^2(K)}\right)^{\frac{1}{2}}+\left(\sum_{e\in\mathcal{E}^n_h}h_e\Vert[\nabla\widehat{w}^n_h.\vec{n}_e]_e\Vert^2_{L^2(e)}\right)^{\frac{1}{2}},\\
\eta^n_{\text{SPACE},5}&:=\left(\sum_{K\in\mathcal{T}^n_h}h_K^2\Vert\widehat{w}^n_h+\varepsilon^{-1}f(u^n_h)\Vert^2_{L^2(T)}\right)^{\frac{1}{2}},\\
\eta^n_{\text{SPACE},6}&:=\left(\sum_{e\in\mathcal{E}^n_h}h_e\Vert[\nabla\widehat{u}^n_h.\vec{n}_e]_e\Vert^2_{L^2(e)}\right)^{\frac{1}{2}}.
\end{align*}
We define the time indicator errors $\eta^n_{\text{TIME},i}$, for $i=4,5$, as follows:
\begin{align*}
\eta^n_{\text{TIME},4}&:=\Vert\nabla(\widehat{w}^n_h-\widehat{w}^{n-1}_h)\Vert,\quad \eta^n_{\text{TIME},6}:=\varepsilon\Vert\nabla(\widehat{u}^n_h-\widehat{u}^{n-1}_h)\Vert,\nonumber\\
\eta^n_{\text{TIME},5}&:=\Vert\widehat{w}^n_h-\widehat{w}^{n-1}_h\Vert+\varepsilon^{-1}\Vert f(u^n_h)-f(u^{n-1}_h)\Vert.
\end{align*}
To simplify the notation we define
\begin{align*}
\widehat{\mu}_{-1}(t)& :=C^*\eta^n_{\text{SPACE},4}+\eta^n_{\text{TIME},4},
\\
\widehat{\mu}_0(t) & :=\eta^n_{\text{TIME},5},
\\
\widehat{\mu}_1(t) & :=\eta^n_{\text{TIME},6}+\eta^n_{\text{SPACE},5}+C^*\eta^n_{\text{SPACE},6}.
\end{align*}
\begin{lemma}
\label{Residual}
For all $\varphi\in \mathbb{H}^1$, the following estimates hold for the residuals $\widehat{\mathcal{R}}$ and $\widehat{\mathcal{S}}$:
\begin{align*}
\langle\widehat{\mathcal{R}}(t), \varphi\rangle&\leq \widehat{\mu}_{-1}(t)\Vert\nabla\varphi\Vert\quad \text{and}\quad
\langle\widehat{\mathcal{S}}(t), \varphi\rangle\leq \widehat{\mu}_0(t)\Vert\varphi\Vert+\widehat{\mu}_1(t)\Vert\nabla\varphi\Vert.
\end{align*}
\end{lemma}
\begin{proof}
For  $\varphi\in\mathbb{H}^1$, $\varphi_h\in\mathbb{V}^n_h$, and  $t\in (t_{n-1}, t_n]$, the  residuals can be expressed as follows:
\begin{align*}
\langle\widehat{\mathcal{R}}(t), \varphi\rangle=&\left(\frac{\widehat{u}^n_h-\widehat{u}^{n-1}_h}{\tau_n}, \varphi-\varphi_h\right)+(\nabla\widehat{w}^n_h,\nabla[\varphi-\varphi_h])+(\nabla[\widehat{w}_{h,\tau}(t)-\widehat{w}^n_h], \nabla\varphi), \\
\langle\widehat{\mathcal{S}}(t), \varphi\rangle=&(\widehat{w}^n_h-\widehat{w}_{h,\tau}(t), \varphi)+(\widehat{w}^n_h, \varphi_h-\varphi)+\varepsilon(\nabla[\widehat{u}_{h,\tau}(t)-\widehat{u}^n_h], \nabla\varphi)+\varepsilon(\nabla\widehat{u}^n_h, \nabla[\varphi-\varphi_h])\nonumber\\
&+\frac{1}{\varepsilon}\left(f(u_{h,\tau}(t))-f(u^n_h), \varphi\right)+\frac{1}{\varepsilon}\left(f(u^n_h), \varphi-\varphi_h\right).
\end{align*}
Taking $\varphi_h=C^n_h\varphi\in\mathbb{V}^n_h$ and applying element-wise integration by parts as in the proof of  \cite[Proposition 6.3]{Bartels_book2015}, along with using \eqref{ErrorPnh} and \eqref{ErrorCnh}, yields the desired result.
\end{proof}

For $\delta>0$, we consider the following subspace of $\Omega$:
\begin{align}
\label{SetOmegadelta}
\Omega_{\delta,\tilde{\varepsilon}}:=\left\{\omega\in\Omega:\; \sup_{t\in[0, T]}\Vert u(t)\Vert^2_{\mathbb{H}^{-1}}+\frac{1}{\varepsilon}\int_0^T\Vert u(s)\Vert^4_{\mathbb{L}^4}ds\leq C\tilde{\varepsilon}^{-\delta}\right\}.
\end{align}
 Using Markov's inequality and  \lemref{LemmaRegularity} one can verify that $\mathbb{P}[\Omega_{\delta, \tilde{\varepsilon}}]>0$ for sufficiently small $\tilde{\varepsilon}$, and $\mathbb{P}[\Omega_{\delta, \tilde{\varepsilon}}]\rightarrow 1$ as $\tilde{\varepsilon}\rightarrow 0$. 
 
 For $\gamma>0$, we consider the following subspace of $\Omega$:
\begin{align}
\label{SetOmegagamma}
\Omega_{\gamma, \tilde{\varepsilon}}:=\left\{\omega\in\Omega:\; \sup_{t\in[0, T]}\Vert \tilde{u}(t)\Vert^2_{\mathbb{L}^4}\leq C\tilde{\varepsilon}^{-\gamma}\right\}. 
\end{align}
Using Markov's inequality and  \lemref{Normutilde} one can verify that $\mathbb{P}[\Omega_{\gamma, \tilde{\varepsilon}}]>0$ for sufficiently small $\tilde{\varepsilon}$, and $\mathbb{P}[\Omega_{\gamma, \tilde{\varepsilon}}]\rightarrow 1$ as $\tilde{\varepsilon}\rightarrow 0$.

Next, we  introduce the discrete principal eigenvalue (cf. \cite{abc94,fp04,BartelsMueller2011,BanasVieth_a_post23})
\begin{align}
\label{Eigenvalue1}
\Lambda_{CH}(t):=\inf_{\underset{\int_{\mathcal{D}}vdx=0}{v\in \mathbb{H}^1\setminus\{0\}}}\frac{\varepsilon\Vert\nabla v\Vert^2+\varepsilon^{-1}\left(f'(u_{h,\tau}(t))v,v)\right)}{\Vert\nabla(-\Delta)^{-1}v\Vert^2}.
\end{align}
The discrete principal eigenvalue $\Lambda_{CH}(t)$ involves the numerical approximation $u_{h,\tau}$ of the stochastic Cahn-Hilliard equation and it is therefore computable for every  $\omega\in \Omega$. 

For an arbitrary $\tilde{\varepsilon}>0$, we define the following subspace of $\Omega$:
\begin{align}
\label{SetOmegatilde}
\Omega_{\tilde{\varepsilon}}:=\left\{\omega\in \Omega:\; \sup_{t\in[0, T]}\Vert \tilde{e}(t)\Vert^2_{\mathbb{H}^{-1}}+\varepsilon\int_0^T\Vert \nabla\tilde{e}(s)\Vert^2ds\leq \tilde{\varepsilon}\right\},
\end{align}
where $\tilde{e}(t):=\tilde{u}(t)-\tilde{u}_{h,\tau}(t)$.

For an appropriate choice of $\tilde{\varepsilon}$,  the $\Omega_{\tilde{\varepsilon}}$ has high probability. In fact, the size of $\Omega_{\tilde{\varepsilon}}$ can be controlled by the accuracy of the numerical approximation of the linear SPDE, see Corollary~\ref{RateLinearSPDE} below.  Taking  $\tilde{\varepsilon}=C(h^{\alpha}+\tau^{\gamma})$  for sufficiently small $0<\alpha<2$ and $0< \gamma<1$, and using  Markov's inequality together with Corollary~\ref{RateLinearSPDE} implies that $\mathbb{P} [\Omega_{\tilde{\varepsilon}}]\rightarrow 1$ as $\tilde{\varepsilon}\rightarrow 0$. In addition, $\mathbb{P}[\Omega_{\tilde{\varepsilon}}]>0$   for sufficiently small $\tau=\tau(\tilde{\varepsilon})$ and $h=h(\tilde{\varepsilon})$.


The lemma below is used to deal with the cubic nonlinearity in the proof of the error estimate for the approximation of the RPDE \eqref{model3} in \thmref{mainresult2} below.
\begin{lemma}
\label{LemmaNormL3}
The following estimate holds $\mathbb{P}$-a.s. on  $\Omega_{\tilde{\varepsilon}}\cap\Omega_{\gamma, \tilde{\varepsilon}}$
\begin{align*}
6\varepsilon^{-1}C_{h,\infty}\int_0^t\Vert e(s)\Vert^3_{\mathbb{L}^3}ds\leq& C\left[C_{h,\infty}^4\varepsilon^{-6}+\varepsilon^3+C_{h,\infty}^4\varepsilon^{-3}+C_{h,\infty}^6\varepsilon^{-8}\right]\tilde{\varepsilon}+C_{h,\infty}^4\varepsilon^{-8}\tilde{\varepsilon}^{1-\gamma}\nonumber\\
&+ \frac{7}{4\varepsilon}\int_0^t\Vert e(s)\Vert^4_{\mathbb{L}^4}ds+\frac{3\varepsilon^4}{4}\int_0^t\Vert\nabla\widehat{e}(s)\Vert^2ds+C\int_0^t\Vert\widehat{e}(s)\Vert^2_{\mathbb{H}^{-1}}ds\nonumber\\
&+\varepsilon^{-1}CC_{h,\infty}^2\int_0^t\Vert \widehat{e}(s)\Vert^{\frac{2}{3}}_{\mathbb{H}^{-1}}\Vert\nabla\widehat{e}(s)\Vert^2ds,
\end{align*}
where $C_{h,\infty}:= \displaystyle \sup_{t\in(0,T)}\|u_{h,\tau}(t)\|_{\mathbb{L}^\infty}$, $e(t):=u(t)-u_{h,\tau}(t)$ and $\hat{e}(t):=\hat{u}(t)-\hat{u}_{h,\tau}(t)$.
\end{lemma}

\begin{proof}
Using \lemref{Fundamentallemma} with $r=\frac{8}{3}$, we obtain:
\begin{align}
\label{long8a}
6\varepsilon^{-1}C_{h,\infty}\Vert e(s)\Vert^3_{\mathbb{L}^3}&\leq \varepsilon^{-1}\Vert e(s)\Vert^4_{\mathbb{L}^4}+\varepsilon^{-1}C C_{h,\infty}^{\frac{4}{3}}\Vert e(s)\Vert^{\frac{2}{3}}_{\mathbb{H}^{-1}}\Vert \nabla e(s)\Vert^{\frac{2}{3}}\Vert  e(s)\Vert^{\frac{4}{3}}_{\mathbb{L}^4}\nonumber\\
&=:\varepsilon^{-1}\Vert e(s)\Vert^4_{\mathbb{L}^4}+\text{VI}.
\end{align} 
Recalling that $e=\widetilde{e}+\widehat{e}$, 
and using the triangle and  Cauchy-Schwarz inequalities, yields
\begin{align}
\label{long9}
\text{VI}&\leq \varepsilon^{-1}CC_{h,\infty}^{\frac{4}{3}}\left(\Vert \widetilde{e}(s)\Vert^{\frac{2}{3}}_{\mathbb{H}^{-1}}+\Vert\widehat{e}(s)\Vert^{\frac{2}{3}}_{\mathbb{H}^{-1}}\right)\Vert\nabla e(s)\Vert^{\frac{2}{3}}\Vert e(s)\Vert^{\frac{4}{3}}_{\mathbb{L}^4}\nonumber\\
&\leq \varepsilon^{-1}CC_{h,\infty}^{\frac{4}{3}}\Vert \widetilde{e}(s)\Vert^{\frac{2}{3}}_{\mathbb{H}^{-1}}\Vert\nabla e(s)\Vert^{\frac{2}{3}}\Vert e(s)\Vert^{\frac{4}{3}}_{\mathbb{L}^4}+\varepsilon^{-1}CC_{h,\infty}^{\frac{4}{3}}\Vert \widehat{e}(s)\Vert^{\frac{2}{3}}_{\mathbb{H}^{-1}}\Vert \nabla e(s)\Vert^{\frac{2}{3}}\Vert e(s)\Vert^{\frac{4}{3}}_{\mathbb{L}^4}\\
&=:\text{VI}_1+\text{VI}_2.\nonumber
\end{align}
Using the definition of $\Omega_{\tilde{\varepsilon}}$, the triangle inequality, and Young's inequality, it follows that
\begin{align}
\label{long10}
\text{VI}_1&=\varepsilon^{-1}CC_{h,\infty}^{\frac{4}{3}}\Vert \widetilde{e}(s)\Vert^{\frac{2}{3}}_{\mathbb{H}^{-1}}\Vert\nabla e(s)\Vert^{\frac{2}{3}}\Vert e(s)\Vert^{\frac{4}{3}}_{\mathbb{L}^4}\nonumber\\
&\leq \varepsilon^{-1}CC_{h,\infty}^{\frac{4}{3}}\tilde{\varepsilon}^{\frac{1}{3}}\Vert\nabla e(s)\Vert^{\frac{2}{3}}\Vert e(s)\Vert^{\frac{4}{3}}_{\mathbb{L}^4}\leq  \varepsilon^{-1}CC_{h,\infty}^{\frac{4}{3}}\tilde{\varepsilon}^{\frac{1}{3}}\left(\Vert \nabla \widetilde{e}(s)\Vert^{\frac{2}{3}}+\Vert \nabla\widehat{e}(s)\Vert^{\frac{2}{3}}\right)\Vert e(s)\Vert^{\frac{4}{3}}_{\mathbb{L}^4}\nonumber\\
&\leq  \varepsilon^{-1}CC_{h,\infty}^{\frac{4}{3}}\tilde{\varepsilon}^{\frac{1}{3}}\Vert \nabla \widetilde{e}(s)\Vert^{\frac{2}{3}}\Vert e(s)\Vert^{\frac{4}{3}}_{\mathbb{L}^4}+  \varepsilon^{-1}CC_{h,\infty}^{\frac{4}{3}}\tilde{\varepsilon}^{\frac{1}{3}}\Vert \nabla\widehat{e}(s)\Vert^{\frac{2}{3}}\Vert e(s)\Vert^{\frac{4}{3}}_{\mathbb{L}^4}\\
&\leq CC_{h,\infty}^{2}\varepsilon^{-1}\tilde{\varepsilon}^{\frac{1}{2}}\Vert \nabla\widetilde{e}(s)\Vert+\frac{\varepsilon^{-1}}{4}\Vert e(s)\Vert^4_{\mathbb{L}^4}+CC_{h,\infty}^2\varepsilon^{-1}\tilde{\varepsilon}^{\frac{1}{2}}\Vert \nabla \widehat{e}(s)\Vert+\frac{\varepsilon^{-1}}{4}\Vert e(s)\Vert^4_{\mathbb{L}^4}\nonumber\\
&\leq CC_{h,\infty}^4\varepsilon^{-6}\tilde{\varepsilon}+\varepsilon^4\Vert\nabla\widetilde{e}(s)\Vert^2+\frac{\varepsilon^4}{4}\Vert \nabla\widehat{e}(s)\Vert^2+\frac{\varepsilon^{-1}}{2}\Vert e(s)\Vert^4_{\mathbb{L}^4}. \nonumber
\end{align}
Using the triangle inequality and Young's inequality, we conclude that
\begin{align*}
\text{VI}_2&=\varepsilon^{-1}CC_{h,\infty}^{\frac{4}{3}}\Vert \widehat{e}(s)\Vert^{\frac{2}{3}}_{\mathbb{H}^{-1}}\Vert \nabla e(s)\Vert^{\frac{2}{3}}\Vert e(s)\Vert^{\frac{4}{3}}_{\mathbb{L}^4}\nonumber\\
&\leq \varepsilon^{-1}CC_{h,\infty}^{\frac{4}{3}}\Vert \widehat{e}(s)\Vert^{\frac{2}{3}}_{\mathbb{H}^{-1}}\left(\Vert \nabla \widetilde{e}(s)\Vert^{\frac{2}{3}}+\Vert \nabla\widehat{e}(s)\Vert^{\frac{2}{3}}\right)\Vert e(s)\Vert^{\frac{4}{3}}_{\mathbb{L}^4}\nonumber\\
&\leq \varepsilon^{-1}CC_{h,\infty}^{\frac{4}{3}}\Vert \widehat{e}(s)\Vert^{\frac{2}{3}}_{\mathbb{H}^{-1}}\Vert \nabla \widetilde{e}(s)\Vert^{\frac{2}{3}}\Vert e(s)\Vert^{\frac{4}{3}}_{\mathbb{L}^4}+\varepsilon^{-1}CC_{h,\infty}^{\frac{4}{3}}\Vert \widehat{e}(s)\Vert^{\frac{2}{3}}_{\mathbb{H}^{-1}}\Vert \nabla \widehat{e}(s)\Vert^{\frac{2}{3}}\Vert e(s)\Vert^{\frac{4}{3}}_{\mathbb{L}^4}\\
&\leq CC_{h,\infty}^{2}\varepsilon^{-1}\Vert \widehat{e}(s)\Vert_{\mathbb{H}^{-1}}\Vert \nabla\widetilde{e}(s)\Vert+\frac{1}{4\varepsilon}\Vert e(s)\Vert^4_{\mathbb{L}^4}\nonumber\\
&\quad+\varepsilon^{-1}CC_{h,\infty}^{\frac{4}{3}}\Vert \widehat{e}(s)\Vert^{\frac{2}{3}}_{\mathbb{H}^{-1}}\Vert \nabla \widehat{e}(s)\Vert^{\frac{2}{3}}\left(\Vert \widetilde{e}(s)\Vert^{\frac{4}{3}}_{\mathbb{L}^4}+\Vert \widehat{e}(s)\Vert^{\frac{4}{3}}_{\mathbb{L}^4}\right).
\end{align*}
Using the  inequality $(a+b)^2\leq 2a^2+2b^2$, we derive from the preceding estimate that
\begin{align*}
\text{VI}_2\leq& C\Vert \widehat{e}(s)\Vert^2_{\mathbb{H}^{-1}}+CC_{h,\infty}^4\varepsilon^{-2}\Vert \nabla\widetilde{e}(s)\Vert^2+\frac{1}{4\varepsilon}\Vert e(s)\Vert^4_{\mathbb{L}^4}\nonumber\\
&+\varepsilon^{-1}CC_{h,\infty}^{\frac{4}{3}}\Vert \widehat{e}(s)\Vert^{\frac{2}{3}}_{\mathbb{H}^{-1}}\Vert \nabla \widehat{e}(s)\Vert^{\frac{2}{3}}\Vert \widetilde{e}(s)\Vert^{\frac{4}{3}}_{\mathbb{L}^4}+\varepsilon^{-1}CC_{h,\infty}^{\frac{4}{3}}\Vert \widehat{e}(s)\Vert^{\frac{2}{3}}_{\mathbb{H}^{-1}}\Vert \nabla \widehat{e}(s)\Vert^{\frac{2}{3}}\Vert \widehat{e}(s)\Vert^{\frac{4}{3}}_{\mathbb{L}^4}.
\end{align*}
Using Young's inequality, the Sobolev embedding $\mathbb{H}^1\hookrightarrow\mathbb{L}^4$, and Poincar\'{e}'s inequality, it follows from the estimate above that
\begin{align}
\label{long11}
\text{VI}_2&\leq C\Vert \widehat{e}(s)\Vert^2_{\mathbb{H}^{-1}}+CC_{h,\infty}^4\varepsilon^{-2}\Vert \nabla\widetilde{e}(s)\Vert^2+\frac{1}{4\varepsilon}\Vert e(s)\Vert^4_{\mathbb{L}^4}+C\varepsilon^2\Vert \widehat{e}(s)\Vert_{\mathbb{H}^{-1}}\Vert\nabla\widehat{e}(s)\Vert\nonumber\\
&\qquad+C\varepsilon^{-7}C_{h,\infty}^4\Vert \widetilde{e}(s)\Vert^4_{\mathbb{L}^4}+\varepsilon^{-1}CC_{h,\infty}^{2}\Vert\widehat{e}(s)\Vert^{\frac{2}{3}}_{\mathbb{H}^{-1}}\Vert \nabla \widehat{e}(s)\Vert^2\nonumber\\
&\leq C\Vert \widehat{e}(s)\Vert^2_{\mathbb{H}^{-1}}+CC_{h,\infty}^4\varepsilon^{-2}\Vert \nabla\widetilde{e}(s)\Vert^2+\frac{1}{4\varepsilon}\Vert e(s)\Vert^4_{\mathbb{L}^4}+C\Vert \widehat{e}(s)\Vert^2_{\mathbb{H}^{-1}}+\frac{\varepsilon^4}{4}\Vert\nabla\widehat{e}(s)\Vert^2\\
&\qquad+C\varepsilon^{-7}C_{h,\infty}^4\Vert \widetilde{e}(s)\Vert^4_{\mathbb{L}^4}+\varepsilon^{-1}CC_{h,\infty}^{2}\Vert\widehat{e}(s)\Vert^{\frac{2}{3}}_{\mathbb{H}^{-1}}\Vert \nabla \widehat{e}(s)\Vert^2.\nonumber
\end{align}
Substituting \eqref{long11} and \eqref{long10} into \eqref{long9}, we obtain:
\begin{align}
\label{long12}
\text{VI}&\leq CC_{h,\infty}^4\varepsilon^{-6}\tilde{\varepsilon}+\left(\varepsilon^4+CC_{h,\infty}^4\varepsilon^{-2}\right)\Vert\nabla\widetilde{e}(s)\Vert^2+CC_{h,\infty}^4\varepsilon^{-7}\Vert \widetilde{e}(s)\Vert^4_{\mathbb{L}^4}+\frac{3\varepsilon^4}{4}\Vert\widehat{e}(s)\Vert^2\\
&\qquad+\frac{3}{4\varepsilon}\Vert e(s)\Vert^4_{\mathbb{L}^4}+C\Vert \widehat{e}(s)\Vert^2_{\mathbb{H}^{-1}}+\varepsilon^{-1}CC_{h,\infty}^2\Vert\widehat{e}(s)\Vert^{\frac{2}{3}}_{\mathbb{H}^{-1}}\Vert \nabla\widehat{e}(s)\Vert^2.\nonumber
\end{align}
Substituting \eqref{long12} into \eqref{long8a}, integrating over $(0, t)$ and using the definition of $\Omega_{\tilde{\varepsilon}}$  (and interpolating the $\mathbb{L}^2$-norm $\frac{3\varepsilon^4}{4}\Vert\widehat{e}(s)\Vert^2$) we deduce
\begin{align}
\label{long13}
6\varepsilon^{-1}C_{h,\infty}\int_0^t\Vert e(s)\Vert^3_{\mathbb{L}^3}ds&\leq CC_{h,\infty}^4\varepsilon^{-6}\tilde{\varepsilon}+C(\varepsilon^4+C_{h,\infty}^4\varepsilon^{-2})\varepsilon^{-1}\tilde{\varepsilon}+CC_{h,\infty}^4\varepsilon^{-7}\int_0^t\Vert \widetilde{e}(s)\Vert^4_{\mathbb{L}^4}ds\nonumber\\
&\quad+ \frac{{5}}{4\varepsilon}\int_0^t\Vert e(s)\Vert^4_{\mathbb{L}^4}ds+\frac{3\varepsilon^4}{4}\int_0^t\Vert\nabla\widehat{e}(s)\Vert^2ds+ {\frac{3\varepsilon^4}{4}}\int_0^t\Vert\widehat{e}(s)\Vert^2_{\mathbb{H}^{-1}}ds\\
&\quad+\varepsilon^{-1}CC_{h,\infty}^2\int_0^t\Vert \widehat{e}(s)\Vert^{\frac{2}{3}}_{\mathbb{H}^{-1}}\Vert\nabla\widehat{e}(s)\Vert^2ds.\nonumber
\end{align}
Using  the embeddings $\mathbb{H}^1\hookrightarrow\mathbb{L}^4$, $\mathbb{L}^{\infty}\hookrightarrow\mathbb{L}^4$, the definitions of $\Omega_{{\gamma, \tilde{\varepsilon}}}$ and $\Omega_{\tilde{\varepsilon}}$ we conclude
\begin{align}
\label{long13a}
\int_0^t\Vert \widetilde{e}(s)\Vert^4_{\mathbb{L}^4}ds&\leq \sup_{s\in[0, t]}\Vert \widetilde{e}(s)\Vert^2_{\mathbb{L}^4}\int_0^t\Vert \widetilde{e}(s)\Vert^2_{\mathbb{L}^4}ds\nonumber\\
&\leq \sup_{s\in[0, T]}\left(\Vert \widetilde{u}(s)\Vert^2_{\mathbb{L}^4}+\Vert \widetilde{u}_{h,\tau}(s)\Vert^2_{\mathbb{L}^{\infty}}\right)\int_0^T\Vert\nabla\widetilde{e}(s)\Vert^2ds\\
&\leq \varepsilon^{-1}\left(\tilde{\varepsilon}^{-\gamma}+C_{h,\infty}^2\right)\tilde{\varepsilon}, \nonumber
\end{align}
$\mathbb{P}$-a.s. on $\Omega_{\gamma, \tilde{\varepsilon}}\cap\Omega_{\tilde{\varepsilon}}$.
Substituting \eqref{long13a} into \eqref{long13} completes the proof. 
\end{proof}


The following theorem provides an estimate for the error $\widehat{e}(t):=\widehat{u}_{h,\tau}(t)-\widehat{u}(t)$ on the subspace $\Omega_{\delta, \tilde{\varepsilon}}\cap\Omega_{\gamma, \tilde{\varepsilon}}\cap\Omega_{\tilde{\varepsilon}}$.
\begin{theorem}
\label{mainresult2}
Assume that  $\Lambda_{CH}\in L^1(0, T)$. Let $\beta=2/3$,
$\alpha(t)=(20+4(1-\varepsilon^3)\Lambda_{CH}(t))^{+}$,  $B=CC_{h,\infty}^2\varepsilon^{-5}$, $E=\exp\left(\int_0^T\alpha(s)ds\right)$ and let
\begin{align*}
A=&C\left\{\left[C_{h,\infty}^4\varepsilon^{-6}+\varepsilon^3+C_{h,\infty}^4\varepsilon^{-3}+C_{h,\infty}^6\varepsilon^{-8}\right]\tilde{\varepsilon}+C_{h,\infty}^4\varepsilon^{-8}\tilde{\varepsilon}^{1-\gamma}+\left(\tilde{\varepsilon}^{-\frac{\gamma}{4}}+C_{h,\infty}^{\frac{1}{2}}\right)\varepsilon^{-2}\tilde{\varepsilon}^{\frac{1}{4}-\frac{\delta}{4}}\right.\nonumber\\
&+\varepsilon(\varepsilon+1)\tilde{\varepsilon}+\left(C_{h,\infty}\varepsilon^{-\frac{5}{4}}\tilde{\varepsilon}^{\frac{\delta}{4}}+\varepsilon^{-\frac{1}{2}}\tilde{\varepsilon}^{\frac{1}{2}}\right)\tilde{\varepsilon}^{\frac{1}{2}-\frac{\delta}{4}}\left(2\varepsilon(1-\varepsilon^3)+ 8\varepsilon^{-2}(1-\varepsilon^3)^2\right)\varepsilon^{-1}\tilde{\varepsilon}\nonumber\\
&+\left.+\int_0^t\left( \widehat{\mu}_{-1}(s)^2+\varepsilon^{-2}\widehat{\mu}_0(s)^2+2\varepsilon^{-4}\widehat{\mu}_1(s)^2\right)ds+(1-\varepsilon^3)\int_0^T\Lambda_{CH}(s)ds\right\}.
\end{align*}
If $8AE\leq \left(8B(1+T)E\right)^{-1/\beta}$, then  it holds $\mathbb{P}$-a.s. on the subspace $\Omega_{\delta, \tilde{\varepsilon}}\cap\Omega_{\gamma, \tilde{\varepsilon}}\cap\Omega_{\tilde{\varepsilon}}$  that:
\begin{align*}
&\sup_{t\in[0, T]}\Vert\widehat{e}(t)\Vert^2_{\mathbb{H}^{-1}}+\frac{\varepsilon^4}{4}\int_0^T\Vert\nabla\widehat{e}(s)\Vert^2ds+\frac{1}{4\varepsilon}\int_0^T\Vert e(s)\Vert^4_{\mathbb{L}^4}ds\nonumber\\
&\quad\leq C\left\{\left[C_{h,\infty}^4\varepsilon^{-6}+\varepsilon^3+C_{h,\infty}^4\varepsilon^{-3}+C_{h,\infty}^6\varepsilon^{-8}\right]\tilde{\varepsilon}+C_{h,\infty}^4\varepsilon^{-8}\tilde{\varepsilon}^{1-\gamma}+\left(\tilde{\varepsilon}^{-\frac{\gamma}{4}}+C_{h,\infty}^{\frac{1}{2}}\right)\varepsilon^{-2}\tilde{\varepsilon}^{\frac{1}{4}-\frac{\delta}{4}}\right.\nonumber\\
&\quad\quad+\varepsilon(\varepsilon+1)\tilde{\varepsilon}+\left(C_{h,\infty}\varepsilon^{-\frac{5}{4}}\tilde{\varepsilon}^{\frac{\delta}{4}}+\varepsilon^{-\frac{1}{2}}\tilde{\varepsilon}^{\frac{1}{2}}\right)\tilde{\varepsilon}^{\frac{1}{2}-\frac{\delta}{4}}+\left(2\varepsilon(1-\varepsilon^3)+ 8\varepsilon^{-2}(1-\varepsilon^3)^2\right)\varepsilon^{-1}\tilde{\varepsilon}\nonumber\\
&\quad\quad\left.+\int_0^T\left( \widehat{\mu}_{-1}(s)^2+\varepsilon^{-2}\widehat{\mu}_0(s)^2+2\varepsilon^{-4}\widehat{\mu}_1(s)^2\right)ds+C(1-\varepsilon^3)\int_0^T\Lambda_{CH}(s)ds\right\}\nonumber\\
&\qquad\qquad\times\exp\left(\int_0^T(20+4(1-\varepsilon^3)\Lambda_{CH}(s))^{+}ds\right),
\end{align*}
  where $e(t)=u_{h,\tau}(t)-u(t)=\widehat{e}(t)+\widetilde{e}(t)$ and $a^{+}:=\max\{a,0\}$.
\end{theorem}
\begin{proof}
Setting $\widehat{e}_w(t):=\widehat{w}_{h,\tau}(t)-\widehat{w}(t)$, subtracting \eqref{Weak1} from \eqref{Weak1a}, and taking $\varphi=(-\Delta)^{-1}\widehat{e}(t)$ and $\psi=\widehat{e}(t)$ in the resulting equations, we derive:
\begin{align*}
&(\partial_t\widehat{e}(t), (-\Delta)^{-1}\widehat{e}(t))+(\nabla\widehat{e}(t), \nabla(-\Delta)^{-1}\widehat{e}(t))=\langle\widehat{\mathcal{R}}(t), (-\Delta)^{-1}\widehat{e}(t)\rangle\\
&-(\widehat{e}_w(t), \widehat{e}(t))+\varepsilon(\nabla\widehat{e}(t), \nabla\widehat{e}(t))=-\varepsilon^{-1}\left(f(u_{h,\tau}(t))-f(u(t)), \widehat{e}(t)\right)+\langle\widehat{\mathcal{S}}(t), \widehat{e}(t)\rangle.
\end{align*}
Summing the preceding two equations, we obtain:
\begin{align}
\label{ErrorRandom1}
&\frac{1}{2}\frac{d}{dt}\Vert\widehat{e}(t)\Vert^2_{\mathbb{H}^{-1}}+\varepsilon\Vert\nabla\widehat{e}(t)\Vert^2+\varepsilon^{-1}\left(f(u_{h,\tau}(t))-f(u(t)), \widehat{e}(t)\right)\nonumber\\
&\qquad=\langle\widehat{\mathcal{R}}(t), (-\Delta)^{-1}\widehat{e}(t)\rangle+\langle\widehat{\mathcal{S}}(t), \widehat{e}(t)\rangle.
\end{align}
Using the fact that $e(t)=\widehat{e}(t)+\widetilde{e}(t)$, we  split the term involving $f$ in \eqref{ErrorRandom1} as follows:
\begin{align}
\label{decompof1}
&f(u_{h,\tau}(t))-f(u(t)), \widehat{e}(t))\nonumber\\
&\qquad=(f(u_{h,\tau}(t))-f(u(t)), e(t))-(f(u_{h,\tau}(t))-f(u(t)), \widetilde{e}(t)). 
\end{align}
We use the identity
\begin{align}
\label{identityf1}
f(a)-f(b)&=(a-b)f'(a)+(a-b)^3-3(a-b)^2a\nonumber\\
&=3(a-b)a^2-(a-b)+(a-b)^3-3(a-b)^2a\quad a,b\in\mathbb{R},
\end{align}
and note that $e(t)=u_{h,\tau}(t)-u(t)$ to obtain
\begin{align}
\label{decompof1a}
\left(f(u_{h,\tau}(t))-f(u(t)), e(t)\right)&=3(u^2_{h,\tau}(t), e(t)^2)-\Vert e(t)\Vert^2+\Vert e(t)\Vert^4_{\mathbb{L}^4}-3(u_{h,\tau}(t), e(t)^3)\nonumber\\
&\geq\Vert e(t)\Vert^4_{\mathbb{L}^4}-\Vert e(t)\Vert^2-3(u_{h,\tau}(t), e(t)^3).
\end{align}
Substituting \eqref{decompof1a} into \eqref{decompof1} and noting $e(t)=\widetilde{e}(t)+\widehat{e}(t)$ yields
\begin{align}
\label{decompof1b}
&\left(f(u_{h,\tau}(t))-f(u(t)), \widehat{e}(t)\right)\nonumber\\
&\qquad\geq \Vert e(t)\Vert^4_{\mathbb{L}^4}-\Vert e(t)\Vert^2-3(u_{h,\tau}(t), e(t)^3)-\left(f(u_{h,\tau}(t))-f(u(t)), \widetilde{e}(t)\right).
\end{align}
Substituting \eqref{decompof1b} into \eqref{ErrorRandom1} yields
\begin{align*}
&\frac{1}{2}\frac{d}{dt}\Vert \widehat{e}(t)\Vert^2_{\mathbb{H}^{-1}}+\varepsilon\Vert\nabla\widehat{e}(t)\Vert^2+\frac{1}{\varepsilon}\Vert e(t)\Vert^4_{\mathbb{L}^4}\nonumber\\
&\qquad\leq \varepsilon^{-1}\Vert e(t)\Vert^2+3\varepsilon^{-1}(u_{h,\tau}(t), e(t)^3)+\varepsilon^{-1}\left(f(u_{h,\tau}(t))-f(u(t)), \widetilde{e}(t)\right)\\
&\qquad\qquad+\langle\widehat{\mathcal{R}}(t), (-\Delta)^{-1}\widehat{e}(t)\rangle+\langle\widehat{\mathcal{S}}(t), \widehat{e}(t)\rangle.\nonumber
\end{align*}
Multiplying the preceding estimate by $\varepsilon^3$ yields:
\begin{align}
\label{ErrorRandom2}
&\frac{\varepsilon^3}{2}\frac{d}{dt}\Vert \widehat{e}(t)\Vert^2_{\mathbb{H}^{-1}}+\varepsilon^4\Vert\nabla\widehat{e}(t)\Vert^2+\varepsilon^2\Vert e(t)\Vert^4_{\mathbb{L}^4}\nonumber\\
&\qquad\leq \varepsilon^{2}\Vert e(t)\Vert^2+3\varepsilon^{2}(u_{h,\tau}(t), e(t)^3)+\varepsilon^{2}\left(f(u_{h,\tau}(t))-f(u(t)), \widetilde{e}(t)\right)\\
&\qquad\qquad+\varepsilon^3\langle\widehat{\mathcal{R}}(t), (-\Delta)^{-1}\widehat{e}(t)\rangle+\varepsilon^3\langle\widehat{\mathcal{S}}(t), \widehat{e}(t)\rangle.\nonumber
\end{align}
Using \eqref{identityf1} we estimate
\begin{align*}
&\left(f(u_{h,\tau}(t))-f(u(t)), \widehat{e}(t)\right)\nonumber\\
&=\left(f(u_{h,\tau}(t))-f(u(t)), e(t)\right)-\left(f(u_{h,\tau}(t))-f(u(t)), \widetilde{e}(t)\right)\nonumber\\
&\geq \left(f'(u_{h,\tau}(t))e(t), e(t)\right)+\Vert e(t)\Vert^4_{\mathbb{L}^4}-3(u_{h,\tau}(t), e(t)^3)-\left(f(u_{h,\tau}(t))-f(u(t)), \widetilde{e}(t)\right).
\end{align*}
We substitute the preceding estimate into \eqref{ErrorRandom1} and get
\begin{align}
\label{ErrorRandom3}
&\frac{1}{2}\frac{d}{dt}\Vert\widehat{e}(t)\Vert^2_{\mathbb{H}^{-1}}+\varepsilon\Vert\nabla\widehat{e}(t)\Vert^2+\frac{1}{\varepsilon}\Vert e(t)\Vert^4_{\mathbb{L}^4}\nonumber\\
&\leq  -\varepsilon^{-1}\left(f'(u_{h,\tau}(t))e(t), e(t)\right)+3\varepsilon^{-1}(u_{h,\tau}(t), e(t)^3)+\varepsilon^{-1}\left(f(u_{h,\tau}(t))-f(u(t)), \widetilde{e}(t)\right)\\
&\qquad+\langle\mathcal{\widehat{R}}(t), (-\Delta)^{-1}\widehat{e}(t)\rangle+\langle \mathcal{S}(t), \widehat{e}(t)\rangle. \nonumber
\end{align}
Using the spectral estimate \eqref{Eigenvalue1} and the triangle inequality, we estimate
\begin{align*}
-\varepsilon^{-1}\left(f'(u_{h,\tau}(t))e(t), e(t)\right)&\leq \Lambda_{CH}(t)\Vert e(t)\Vert^2_{\mathbb{H}^{-1}}+\varepsilon\Vert \nabla e(t)\Vert^2\nonumber\\
&\leq  \Lambda_{CH}(t)\Vert e(t)\Vert^2_{\mathbb{H}^{-1}}+ \varepsilon\left(\Vert\nabla\widetilde{e}(t)\Vert+\Vert\nabla\widehat{e}(t)\Vert\right)^2\\
&\leq\Lambda_{CH}(t)\Vert e(t)\Vert^2_{\mathbb{H}^{-1}}+\varepsilon\Vert \nabla\widetilde{e}(t)\Vert^2+\varepsilon\Vert\nabla\widehat{e}(t)\Vert^2+2\varepsilon\Vert\nabla\widehat{e}(t)\Vert\Vert\nabla\widetilde{e}(t)\Vert.\nonumber
\end{align*}
Substituting the preceding estimate into \eqref{ErrorRandom3} yields
\begin{align}
\label{ErrorRandom4}
&\frac{1}{2}\frac{d}{dt}\Vert\widehat{e}(t)\Vert^2_{\mathbb{H}^{-1}}+\varepsilon\Vert\nabla\widehat{e}(t)\Vert^2+\frac{1}{\varepsilon}\Vert e(t)\Vert^4_{\mathbb{L}^4}\nonumber\\
&\;\;\;\leq \Lambda_{CH}(t)\Vert e(t)\Vert^2_{\mathbb{H}^{-1}}+3\varepsilon^{-1}(u_{h,\tau}(t), e(t)^3)+\varepsilon\Vert \nabla\widetilde{e}(t)\Vert^2+\varepsilon\Vert\nabla\widehat{e}(t)\Vert^2\\
&\quad\;\;\;+2\varepsilon\Vert\nabla\widehat{e}(t)\Vert\Vert\nabla\widetilde{e}(t)\Vert+\varepsilon^{-1}\left(f(u_{h,\tau}(t))-f(u(t)), \widetilde{e}(t)\right)\nonumber\\
&\quad\;\;\;+\langle\mathcal{\widehat{R}}(t), (-\Delta)^{-1}\widehat{e}(t)\rangle+\langle \mathcal{S}(t), \widehat{e}(t)\rangle. \nonumber
\end{align}
Multiplying \eqref{ErrorRandom4} by $1-\varepsilon^3$  and adding the resulting estimate to  \eqref{ErrorRandom2} yields:
\begin{align}
\label{ErrorRandom5}
&\frac{1}{2}\frac{d}{dt}\Vert\widehat{e}(t)\Vert^2_{\mathbb{H}^{-1}}+\varepsilon^{4}\Vert\nabla\widehat{e}(t)\Vert^2+\frac{1}{\varepsilon}\Vert e(t)\Vert^4_{\mathbb{L}^4}\nonumber\\
&\qquad\leq \varepsilon^2\Vert e(t)\Vert^2+(1-\varepsilon^3)\Lambda_{CH}(t)\Vert e(t)\Vert^2_{\mathbb{H}^{-1}}+\varepsilon^{-1}\left(f(u_{h,\tau}(t))-f(u(t)), \widetilde{e}(t)\right)\\
&\qquad\qquad+3\varepsilon^{-1}(u_{h,\tau}(t), e(t)^3)+\varepsilon(1-\varepsilon^3)\Vert\nabla\widetilde{e}(t)\Vert^2+2\varepsilon(1-\varepsilon^3)\Vert\nabla\widehat{e}(t)\Vert\Vert\nabla\widetilde{e}(t)\Vert\nonumber\\
&\qquad\qquad+\langle\mathcal{\widehat{R}}(t), (-\Delta)^{-1}\widehat{e}(t)\rangle+\langle \mathcal{S}(t), \widehat{e}(t)\rangle,\nonumber
\end{align}
where we also used the fact that $0<\varepsilon<1$. 

Using \lemref{Residual} and Young's inequality, we obtain:
\begin{align}
\label{inter1a}
&2\langle\mathcal{\widehat{R}}(t), (-\Delta)^{-1}\widehat{e}(t)\rangle+2\langle \mathcal{S}(t), \widehat{e}(t)\rangle\nonumber\\
&\qquad\leq \widehat{\mu}_{-1}(t)^2+\varepsilon^{-2}\widehat{\mu}_0(t)^2+2\varepsilon^{-4}\widehat{\mu}_1(t)^2+\Vert \widehat{e}(t)\Vert^2_{\mathbb{H}^{-1}}+\varepsilon^2\Vert\widehat{e}(t)\Vert^2+\frac{\varepsilon^4}{2}\Vert\nabla\widehat{e}(t)\Vert^2.
\end{align}
Using the interpolation inequality $\Vert\cdot\Vert^2\leq\Vert\cdot\Vert_{\mathbb{H}^{-1}}\Vert\nabla \cdot\Vert$ and Young's inequality, we derive: 
\begin{align}
\label{inter1b}
4\varepsilon^2\Vert \widehat{e}(t)\Vert^2\leq 4\varepsilon^2\Vert\widehat{e}(t)\Vert_{\mathbb{H}^{-1}}\Vert\nabla\widehat{e}(t)\Vert\leq \frac{\varepsilon^4}{2}\Vert\nabla\widehat{e}(t)\Vert^2+18\Vert\widehat{e}(t)\Vert^2_{\mathbb{H}^{-1}}. 
\end{align}
Using \eqref{inter1b} in \eqref{inter1a}, we obtain:
\begin{align*}
&2\langle\mathcal{\widehat{R}}(t), (-\Delta)^{-1}\widehat{e}(t)\rangle+2\langle \mathcal{S}(t), \widehat{e}(t)\rangle+2\varepsilon^2\Vert\widehat{e}(t)\Vert^2\nonumber\\
&\leq  \widehat{\mu}_{-1}(t)^2+\varepsilon^{-2}\widehat{\mu}_0(t)^2+2\varepsilon^{-4}\widehat{\mu}_1(t)^2+18\Vert \widehat{e}(t)\Vert^2_{\mathbb{H}^{-1}}+\frac{\varepsilon^4}{2}\Vert\nabla\widehat{e}(t)\Vert^2. 
\end{align*}
Using Young's inequality, yields:
\begin{align}
\label{inter1c}
4\varepsilon(1-\varepsilon^3)\Vert \nabla\widehat{e}(t)\Vert\Vert\nabla\widetilde{e}(t)\Vert\leq \frac{\varepsilon^4}{2}\Vert\nabla\widehat{e}(t)\Vert^2+8\varepsilon^{-2}(1-\varepsilon^3)^2\Vert\nabla\widetilde{e}(t)\Vert^2. 
\end{align} 
We substitute \eqref{inter1c} and \eqref{inter1a} into \eqref{ErrorRandom5} and integrate  over $(0, t)$ to get
\begin{align}
\label{ErrorRandom6}
&\Vert\widehat{e}(t)\Vert^2_{\mathbb{H}^{-1}}+\varepsilon^4\int_0^t\Vert\nabla\widehat{e}(s)\Vert^2ds+2\varepsilon^{-1}\int_0^t\Vert e(s)\Vert^4_{\mathbb{L}^4}ds\nonumber\\
&\leq\int_0^t(19+4(1-\varepsilon^3)\Lambda_{CH}(s))\Vert\widehat{e}(s)\Vert^2_{\mathbb{H}^{-1}}ds+{6}\varepsilon^{-1}\int_0^t\vert (u_{h,\tau}(s), e(s)^3)\vert ds\\
&\quad+\int_0^t\left( \widehat{\mu}_{-1}(s)^2+\varepsilon^{-2}\widehat{\mu}_0(s)^2+2\varepsilon^{-4}\widehat{\mu}_1(s)^2\right)ds+4(1-\varepsilon^3)\int_0^t\Lambda_{CH}(s)\Vert\widetilde{e}(s)\Vert^2_{\mathbb{H}^{-1}}ds\nonumber\\
&\quad+4\varepsilon^2\int_0^t\Vert\widetilde{e}(s)\Vert^2ds+\left(2\varepsilon(1-\varepsilon^3)+8\varepsilon^{-2}(1-\varepsilon^3)^2\right)\int_0^t\Vert\nabla\widetilde{e}(s)\Vert^2ds+\text{VII}, \nonumber
\end{align}
where $\text{VII}$ is defined as:
\begin{align*}
\text{VII}:=2\varepsilon^{-1}\int_0^t\left\vert \Big(f(u_{h,\tau}(s))-f(u(s)), \widetilde{e}(s)\Big)\right\vert ds.
\end{align*}
Next, we estimate $\text{VII}$. Applying the triangle inequality gives:
\begin{align}
\label{long1}
\text{VII}\leq 2\varepsilon^{-1}\int_0^t\left\vert \big(f(u_{h,\tau}(s), \widetilde{e}(s)\big)\right\vert ds+2\varepsilon^{-1}\int_0^t\left\vert \big(f(u(s)), \widetilde{e}(s)\big)\right\vert ds
=:\text{VII}_1+\text{VII}_2.
\end{align}
Using the Cauchy-Schwarz inequality, the interpolation inequality $\Vert \cdot \Vert^2 \leq \Vert \cdot \Vert_{\mathbb{H}^{-1}} \Vert \nabla \cdot \Vert$, Hölder's inequality, and the definition of $\Omega_{\tilde{\varepsilon}}$, it follows $\mathbb{P}$-a.s. on $\Omega_{\tilde{\varepsilon}}$ that
\begin{align}
\label{long1a}
\text{VII}_1&=2\varepsilon^{-1}\int_0^t\left\vert \big(f(u_{h,\tau}(s)), \widetilde{e}(s)\big)\right\vert ds\leq \varepsilon^{-1}C\int_0^t\Vert u_{h,\tau}(s)\Vert_{\mathbb{L}^{\infty}}\Vert \widetilde{e}(s)\Vert ds\nonumber\\
&\leq\varepsilon^{-1} C\sup_{s\in[0, T]}\Vert u_{h,\tau}(s)\Vert_{\mathbb{L}^{\infty}}\int_0^t\Vert \widetilde{e}(s)\Vert_{\mathbb{H}^{-1}}^{\frac{1}{2}}\Vert \nabla \widetilde{e}(s)\Vert^{\frac{1}{2}}ds\nonumber\\
&\leq \varepsilon^{-1}C\sup_{t\in [0, T]}\Vert u_{h,\tau}(t)\Vert_{\mathbb{L}^{\infty}}\sup_{t\in[0, T]}\Vert \widetilde{e}(t)\Vert_{\mathbb{H}^{-1}}^{\frac{1}{2}}\int_0^T\Vert\nabla\widetilde{e}(s)\Vert^{\frac{1}{2}}ds\\
&\leq \varepsilon^{-1}C T^{\frac{3}{4}}\sup_{t\in [0, T]}\Vert u_{h,\tau}(t)\Vert_{\mathbb{L}^{\infty}}\sup_{t\in[0, T]}\Vert \widetilde{e}(t)\Vert_{\mathbb{H}^{-1}}^{\frac{1}{2}}\left(\int_0^T\Vert\nabla\widetilde{e}(s)\Vert^{2}ds\right)^{\frac{1}{4}}\leq CC_{h,\infty}\varepsilon^{-\frac{5}{4}}\tilde{\varepsilon}^{\frac{1}{2}}.\nonumber
\end{align}
Recalling that $f(a) = a^3 - a$ and applying the triangle inequality, we estimate
\begin{align}
\label{long1b}
\text{VII}_2\leq 2\varepsilon^{-1}\int_0^t\left\vert \big(u(s)^3, \widetilde{e}(s)\big)\right\vert ds+2\varepsilon^{-1}\int_0^t\left\vert \big(u(s), \widetilde{e}(s)\big)\right\vert ds=:\text{VII}_{21}+\text{VII}_{22}.
\end{align}
Using the Cauchy-Schwarz and Hölder inequalities, the embedding $\mathbb{L}^q \hookrightarrow \mathbb{L}^p$ ($1 \leq p \leq q$), and the interpolation inequality $\Vert \cdot \Vert^2 \leq \Vert \cdot \Vert_{\mathbb{H}^{-1}} \Vert \nabla \cdot \Vert$, it holds $\mathbb{P}$-a.s. on $\Omega_{\delta, \tilde{\varepsilon}} \cap \Omega_{\tilde{\varepsilon}}$ that:
\begin{align}
\label{long1c}
\text{VII}_{22}&\leq 2\varepsilon^{-1}\int_0^t\Vert u(s)\Vert\Vert \widetilde{e}(s)\Vert ds\leq 2\varepsilon^{-1}\left(\int_0^t\Vert u(s)\Vert^2ds\right)^{\frac{1}{2}}\left(\int_0^t\Vert\widetilde{e}(s)\Vert^2ds\right)^{\frac{1}{2}}\nonumber\\
&\leq C\varepsilon^{-1}\left(\int_0^t\Vert u(s)\Vert^2_{\mathbb{L}^4} ds\right)^{\frac{1}{2}}\left(\int_0^t\Vert\widetilde{e}(s)\Vert_{\mathbb{H}^{-1}}\Vert\nabla\widetilde{e}(s)\Vert ds\right)^{\frac{1}{2}}\nonumber\\
&\leq C\varepsilon^{-1}\left(\int_0^t\Vert u(s)\Vert^4_{\mathbb{L}^4} ds\right)^{\frac{1}{4}}\sup_{t\in[0, T]}\Vert\widetilde{e}(s)\Vert_{\mathbb{H}^{-1}}^{\frac{1}{2}}\left(\int_0^T\Vert\nabla\widetilde{e}(s)\Vert ds\right)^{\frac{1}{2}}\\
&\leq C\varepsilon^{-1}\left(\int_0^t\Vert u(s)\Vert^4_{\mathbb{L}^4} ds\right)^{\frac{1}{4}}\sup_{t\in[0, T]}\Vert\widetilde{e}(s)\Vert_{\mathbb{H}^{-1}}^{\frac{1}{2}}\left(\int_0^T\Vert\nabla\widetilde{e}(s)\Vert^2 ds\right)^{\frac{1}{4}}\leq C\varepsilon^{-\frac{1}{2}}\tilde{\varepsilon}^{1-\frac{\delta}{4}}.\nonumber
\end{align}
Using Hölder's and Young's inequalities, we estimate
\begin{align}
\label{long1d}
\text{VII}_{21}&  =2\varepsilon^{-1}\int_0^t\vert (u(s)^3, \widetilde{e}(s))\vert ds \leq 2\varepsilon^{-1}\int_0^t\Vert u(s)\Vert^3_{\mathbb{L}^4}\Vert \widetilde{e}(s)\Vert_{\mathbb{L}^4}ds\nonumber\\
&\leq2\varepsilon^{-1}\sup_{t\in[0, T]}\Vert \widetilde{e}(t)\Vert^{\frac{1}{2}}_{\mathbb{L}^4}\int_0^T\Vert u(s)\Vert^3_{\mathbb{L}^4}\Vert \widetilde{e}(s)\Vert_{\mathbb{L}^4}^{\frac{1}{2}}ds\\
&\leq 2\varepsilon^{-1}\sup_{t\in[0, T]}\Vert \widetilde{e}(t)\Vert^{\frac{1}{2}}_{\mathbb{L}^4}\left(\int_0^t\Vert u(s)\Vert^4_{\mathbb{L}^4}ds\right)^{\frac{3}{4}}\left(\int_0^t\Vert\widetilde{e}(s)\Vert^2_{\mathbb{L}^4}ds\right)^{\frac{1}{4}}.\nonumber
\end{align}
Using the Sobolev embeddings $\mathbb{H}^1 \hookrightarrow \mathbb{L}^4$ and $\mathbb{L}^\infty \hookrightarrow \mathbb{L}^4$, Poincaré's inequality, and the definitions of $\Omega_{\tilde{\varepsilon}}$, $\Omega_{\delta, \tilde{\varepsilon}}$, and $\Omega_{\gamma, \tilde{\varepsilon}}$, it follows from \eqref{long1d} that:
\begin{align}
\label{long3}
\text{VII}_{21}&\leq C\varepsilon^{-1}\sup_{t\in[0, T]}\left[\Vert \widetilde{u}(t)\Vert^{\frac{1}{2}}_{\mathbb{L}^4}+\Vert {\red \widetilde{u}}_{h,\tau}(t)\Vert_{\mathbb{L}^4}^{\frac{1}{2}}\right]\left(\int_0^t\Vert u(s)\Vert^4_{\mathbb{L}^4}ds\right)^{\frac{3}{4}}\left(\int_0^t\Vert\nabla\widetilde{e}(s)\Vert^2ds\right)^{\frac{1}{4}}\nonumber\\
&\leq C\left(\tilde{\varepsilon}^{-\frac{\gamma}{4}}+C_{h,\infty}^{\frac{1}{2}}\right)\varepsilon^{-2}\tilde{\varepsilon}^{\frac{1}{4}-\frac{\delta}{4}}.
\end{align}
Substituting \eqref{long3} and \eqref{long1c} into \eqref{long1b} gives:
\begin{align}
\label{long3aa}
\text{VII}_2\leq C\left(\tilde{\varepsilon}^{-\frac{\gamma}{4}}+C_{h,\infty}^{\frac{1}{2}}\right)\varepsilon^{-2}\tilde{\varepsilon}^{\frac{1}{4}-\frac{\delta}{4}}+C\varepsilon^{-\frac{1}{2}}\tilde{\varepsilon}^{1-\frac{\delta}{4}}.
\end{align} 
Substituting \eqref{long1a} and \eqref{long3aa}  into \eqref{long1} yields
\begin{align}
\label{long3a}
\text{VII}\leq C\left(C_{h,\infty}\varepsilon^{-\frac{5}{4}}\tilde{\varepsilon}^{\frac{\delta}{4}}+\varepsilon^{-\frac{1}{2}}\tilde{\varepsilon}^{\frac{1}{2}}\right)\tilde{\varepsilon}^{\frac{1}{2}-\frac{\delta}{4}}+C\left(\tilde{\varepsilon}^{-\frac{\gamma}{4}}+C_{h,\infty}^{\frac{1}{2}}\right)\varepsilon^{-2}\tilde{\varepsilon}^{\frac{1}{4}-\frac{\delta}{4}}.
\end{align}
Substituting \eqref{long3a} into \eqref{ErrorRandom6} yields:
\begin{align}
\label{ErrorRandom7}
&\Vert\widehat{e}(t)\Vert^2_{\mathbb{H}^{-1}}+3\varepsilon^4\int_0^t\Vert\nabla\widehat{e}(s)\Vert^2ds+2\varepsilon^{-1}\int_0^t\Vert e(s)\Vert^4_{\mathbb{L}^4}ds\nonumber\\
&\quad\leq\int_0^t(19+4(1-\varepsilon^3)\Lambda_{CH}(s))\Vert\widehat{e}(s)\Vert^2_{\mathbb{H}^{-1}}ds+6\varepsilon^{-1}\int_0^t\vert (u_{h,\tau}(s), e(s))^3)\vert ds\\
&\quad\quad+C\left(C_{h,\infty}\varepsilon^{-\frac{5}{4}}\tilde{\varepsilon}^{\frac{\delta}{4}}+\varepsilon^{-\frac{1}{2}}\tilde{\varepsilon}^{\frac{1}{2}}\right)\tilde{\varepsilon}^{\frac{1}{2}-\frac{\delta}{4}}+C\left(\tilde{\varepsilon}^{-\frac{\gamma}{4}}+C_{h,\infty}^{\frac{1}{2}}\right)\varepsilon^{-2}\tilde{\varepsilon}^{\frac{1}{4}-\frac{\delta}{4}}\nonumber\\
&\quad\quad+\int_0^t\left( \widehat{\mu}_{-1}(s)^2+\varepsilon^{-2}\widehat{\mu}_0(s)^2+2\varepsilon^{-4}\widehat{\mu}_1(s)^2\right)ds+C(1-\varepsilon^3)\tilde{\varepsilon}\int_0^t\Lambda_{CH}(s)ds\nonumber\\
&\quad\quad+4\varepsilon^2\int_0^t\Vert\widetilde{e}(s)\Vert^2ds+\left(2\varepsilon(1-\varepsilon^3)+8\varepsilon^{-2}(1-\varepsilon^3)^2\right)\int_0^t\Vert\nabla\widetilde{e}(s)\Vert^2ds.\nonumber
\end{align}
Next, we note that
\begin{align}
\label{long5}
6\varepsilon^{-1}\int_0^t\vert(u_{h,\tau}(s), e(s)^3)\vert ds\leq 6\varepsilon^{-1}\sup_{t\in[0, T]}\Vert u_{h,\tau}(t)\Vert_{\mathbb{L}^{\infty}}\int_0^t\Vert e(s)\Vert^3_{\mathbb{L}^3}ds.
\end{align}
Using the interpolation inequality $\Vert \cdot \Vert^2 \leq \Vert \cdot \Vert_{\mathbb{H}^{-1}} \Vert \nabla \cdot \Vert$, it holds $\mathbb{P}$-a.s. on $\Omega_{\tilde{\varepsilon}}$ that:
\begin{align}
\label{long6}
2\varepsilon^2\int_0^t\Vert \widetilde{e}(s)\Vert^2ds\leq \varepsilon^{2}\sup_{s\in[0, T]}\Vert\widetilde{e}(s)\Vert^2_{\mathbb{H}^{-1}}+\varepsilon^2\int_0^t\Vert \nabla\widetilde{e}(s)\Vert^2ds\leq C\varepsilon^2\tilde{\varepsilon}+C\varepsilon\tilde{\varepsilon}. 
\end{align}
Substituting \eqref{long6} and \eqref{long5} into \eqref{ErrorRandom7}, it follows that $\mathbb{P}$-a.s. on $\Omega_{\delta, \tilde{\varepsilon}} \cap \Omega_{\gamma, \tilde{\varepsilon}} \cap \Omega_{\tilde{\varepsilon}}$, we have:
\begin{align}
\label{long7}
&\Vert\widehat{e}(t)\Vert^2_{\mathbb{H}^{-1}}+\varepsilon^4\int_0^t\Vert\nabla\widehat{e}(s)\Vert^2ds+2\varepsilon^{-1}\int_0^t\Vert e(s)\Vert^4_{\mathbb{L}^4}ds\nonumber\\
&\quad\leq\int_0^t(19+4(1-\varepsilon^3)\Lambda_{CH}(s))\Vert\widehat{e}(s)\Vert^2_{\mathbb{H}^{-1}}ds+6\varepsilon^{-1}C_{h,\infty}\int_0^t\Vert e(s)\Vert^3_{\mathbb{L}^3}ds\\
&\quad\quad+C\left(C_{h,\infty}\varepsilon^{-\frac{5}{4}}\tilde{\varepsilon}^{\frac{\delta}{4}}+\varepsilon^{-\frac{1}{2}}\tilde{\varepsilon}^{\frac{1}{2}}\right)\tilde{\varepsilon}^{\frac{1}{2}-\frac{\delta}{4}}+C\left(\tilde{\varepsilon}^{-\frac{\gamma}{4}}+C_{h,\infty}^{\frac{1}{2}}\right)\varepsilon^{-2}\tilde{\varepsilon}^{\frac{1}{4}-\frac{\delta}{4}}\nonumber\\
&\quad\quad+C(1-\varepsilon^3)\tilde{\varepsilon}\int_0^t\Lambda_{CH}(s)ds+\int_0^t\left( \widehat{\mu}_{-1}(s)^2+\varepsilon^{-2}\widehat{\mu}_0(s)^2+2\varepsilon^{-4}\widehat{\mu}_1(s)^2\right)ds\nonumber\\
&\quad\quad+C\varepsilon(\varepsilon+1)\tilde{\varepsilon}+\left(2\varepsilon(1-\varepsilon^3)+8\varepsilon^{-2}(1-\varepsilon^3)^2\right)\varepsilon^{-1}\tilde{\varepsilon}.\nonumber
\end{align}
We  use \lemref{LemmaNormL3} to estimate the $\mathbb{L}^3$-term on the right-hand side of \eqref{long7} and obtain that
 \begin{align*}
&\Vert\widehat{e}(t)\Vert^2_{\mathbb{H}^{-1}}+\frac{\varepsilon^4}{4}\int_0^t\Vert\nabla\widehat{e}(s)\Vert^2ds+\frac{1}{4\varepsilon}\int_0^t\Vert e(s)\Vert^4_{\mathbb{L}^4}ds\nonumber\\
&\quad\leq C\left[C_{h,\infty}^4\varepsilon^{-6}+\varepsilon^3+C_{h,\infty}^4\varepsilon^{-3}+C_{h,\infty}^6\varepsilon^{-8}\right]\tilde{\varepsilon}+CC_{h,\infty}^4\varepsilon^{-8}\tilde{\varepsilon}^{1-\gamma}\\
&\quad\quad+C\left(C_{h,\infty}\varepsilon^{-\frac{5}{4}}\tilde{\varepsilon}^{\frac{\delta}{4}}+\varepsilon^{-\frac{1}{2}}\tilde{\varepsilon}^{\frac{1}{2}}\right)\tilde{\varepsilon}^{\frac{1}{2}-\frac{\delta}{4}}+C\left(\tilde{\varepsilon}^{-\frac{\gamma}{4}}+C_{h,\infty}^{\frac{1}{2}}\right)\varepsilon^{-2}\tilde{\varepsilon}^{\frac{1}{4}-\frac{\delta}{4}}+C\varepsilon(\varepsilon+1)\tilde{\varepsilon}\nonumber\\
&\quad\quad+\left(2\varepsilon(1-\varepsilon^3)+8\varepsilon^{-2}(1-\varepsilon^3)^2\right)\varepsilon^{-1}\tilde{\varepsilon}+\int_0^t\left( \widehat{\mu}_{-1}(s)^2+\varepsilon^{-2}\widehat{\mu}_0(s)^2+2\varepsilon^{-4}\widehat{\mu}_1(s)^2\right)ds\nonumber\\
&\quad\quad+C(1-\varepsilon^3)\int_0^T\Lambda_{CH}(s)ds+C\int_0^t(20+4(1-\varepsilon^3)\Lambda_{CH}(s))\Vert\widehat{e}(s)\Vert^2_{\mathbb{H}^{-1}}ds\nonumber\\
&\quad\quad+CC_{h,\infty}^2\varepsilon^{-1}\int_0^t\Vert \widehat{e}(s)\Vert^{\frac{2}{3}}_{\mathbb{H}^{-1}}\Vert\nabla\widehat{e}(s)\Vert^2ds,\nonumber
\end{align*}
$\mathbb{P}$-a.s. on $\Omega_{\delta, \tilde{\varepsilon}} \cap \Omega_{\gamma, \tilde{\varepsilon}} \cap \Omega_{\tilde{\varepsilon}}$.

Applying the generalized Gronwall lemma (\lemref{Gronwall}) to the above estimate, with 
\begin{align*}
y_1(t)=\Vert \widehat{e}(t)\Vert^2_{\mathbb{H}^{-1}},\; y_2(t)=\varepsilon^4\Vert\nabla\widehat{e}(t)\Vert^2+\varepsilon^{-1}\Vert e(t)\Vert^4_{\mathbb{L}^4},\; y_3(t)=CC_{h,\infty}^2\varepsilon^{-1}\Vert \widehat{e}(t)\Vert^{\frac{2}{2}}_{\mathbb{H}^{-1}}\Vert\nabla\widehat{e}(t)\Vert^2
\end{align*}
and $\beta=2/3$ completes the proof.
\end{proof}

 \section{Error estimate for the approximation of the stochastic Cahn-Hilliard equation}
\label{ErrorRegularizedCHE}
In this section, we combine the estimates for the linear SPDE in Section~\ref{ErrorlinearSPDE} and the estimates for the nonlinear RPDE in Section~\ref
{ErrorRandonPDE} to derive an a posteriori error estimate for the fully discrete approximation scheme \eqref{scheme1}.

We denote the subspace of functions from $\mathbb{V}^n_h$ with zero mean as
\begin{align*}
\mathring{\mathbb{V}}^n_h:=\{v_h\in \mathbb{V}^n_h: (v_h, 1)=0\}.
\end{align*}
We introduce the discrete inverse Laplace operator $\Delta_h^{-1}: \mathring{\mathbb{V}}_h^n\to \mathring{\mathbb{V}}_h^n$ as follows: given $u_h\in \mathring{\mathbb{V}}^n_h$, define $\Delta_h^{-1}u_h\in \mathring{\mathbb{V}}_h^n$ such that
\begin{align*}
(\nabla(-\Delta_h^{-1}u_h), \nabla v_h)=(u_h, v_h)\quad \forall v_h\in \mathbb{V}_h^n.
\end{align*}

For $u_h, v_h\in \mathring{\mathbb{V}}_h^n$, we define the discrete $\mathbb{H}^{-1}$ inner product as follows:
\begin{align*}
(u_h, v_h)_{-1,h}:=(\nabla(-\Delta_h)^{-1}u_h, \nabla(-\Delta_h)^{-1}v_h), 
\end{align*}
and the corresponding discrete $\mathbb{H}^{-1}$-norm for $v_h\in \mathring{\mathbb{V}}_h^n$ is given by:
\begin{align*}
\Vert v_h\Vert_{-1,h}:= \|\nabla(-\Delta_h)^{-1}v_h\|.
\end{align*}
There exists a constant $C\geq 0$ such that  $\Vert v_h\Vert_{\mathbb{H}^{-1}}\leq C\Vert v_h\Vert_{-1,h}$ for all $v_h\in\mathring{\mathbb{V}}_h^n$. In fact,
noting the definition of the projection operator $P_h^n$ and its stability $\Vert \nabla P_h^n\psi\Vert\leq C\Vert\nabla\psi\Vert$
and the definition of $-\Delta_h^{-1}$, we deduce that
\begin{align}\label{normeq}
\Vert v_h\Vert_{\mathbb{H}^{-1}}
& =\sup_{\psi\in \mathbb{H}^1}\frac{(v_h, \psi)}{\|\nabla \psi\|_{\mathbb{H}^1}}
=\sup_{\psi\in \mathbb{H}^1}\frac{(v_h, P_h^n\psi)}{\|\nabla \psi\|_{\mathbb{H}^1}}
\leq C\sup_{\psi\in \mathbb{H}^1}\frac{(v_h, P_h^n\psi)}{\|\nabla P_h^n\psi\|_{\mathbb{H}^1}}
\\
\nonumber
& = C\sup_{\psi\in \mathbb{H}^1}\frac{(\nabla (-\Delta_h^{-1})v_h, \nabla P_h^n\psi)}{\|\nabla P_h^n\psi\|_{\mathbb{H}^1}}
\leq C\Vert v_h\Vert_{-1,h}.
\end{align}

To simplify the notation, we formulate the a posteriori estimate from \lemref{LemmaErrortilde} as follows:
\begin{align*}
\mathbb{E}\left[\sup_{t\in[0,T]}\Vert\widetilde{e}(t)\Vert^2_{\mathbb{H}^{-1}}+\varepsilon\int_0^T\Vert\nabla\widetilde{e}(s)\Vert^2ds\right]\leq \mathcal{\widetilde{R}},
\end{align*}
and the estimate from \thmref{mainresult2} as follows:
\begin{align*}
\sup_{t\in[0,T]}\Vert\widehat{e}(t)\Vert^2_{\mathbb{H}^{-1}}+\frac{\varepsilon^4}{4}\int_0^T\Vert\nabla\widehat{e}(s)\Vert^2ds+\frac{1}{4\varepsilon}\int_0^T\Vert e(s)\Vert^4_{\mathbb{L}^4}ds\leq \mathcal{\widehat{R}},
\end{align*}
$\mathbb{P}$-a.s. on $\Omega_{\delta, \tilde{\varepsilon}}\cap\Omega_{\gamma, \tilde{\varepsilon}}\cap\Omega_{\tilde{\varepsilon}}$ (recall the definitions in \eqref{SetOmegadelta}, \eqref{SetOmegagamma}, and \eqref{SetOmegatilde}).

The next lemma is used in Theorem~\ref{mainresult3} to control the  error on the complement of the probability subspace $\Omega_{\delta, \tilde{\varepsilon}}\cap\Omega_{\gamma, \tilde{\varepsilon}}\cap\Omega_{\tilde{\varepsilon}}$.
\begin{lemma}
\label{LemmaErrorhat2}
The following estimate holds for the approximation error $\widehat{e}(t)=\widehat{u}(t)-\widehat{u}_{h,\tau}$
\begin{align*}
\mathbb{E}\left[\left(\sup_{t\in[0, T]}\Vert\widehat{e}(t)\Vert^2_{\mathbb{H}^{-1}}+\varepsilon\int_0^T\Vert\nabla\widehat{e}(s)\Vert^2ds+\frac{1}{\varepsilon}\int_0^T\Vert\widehat{e}(s)\Vert^4_{\mathbb{L}^4}ds\right)^2\right]\leq \widehat{C}_{0,h}^2+(\mathbb{E}[\widehat{\mathcal{R}}_{\mu}])^2,
\end{align*}
where  $\widehat{C}_{0,h}$ and $\widehat{\mathcal{R}}_{\mu}$ are defined respectively in \eqref{ConstantC0} and \eqref{Residualmu} below.
\end{lemma}
\begin{proof}
We recall that $\widehat{e}(t)=\widehat{u}_{h,\tau}(t)-\widehat{u}(t)$, where from \eqref{model3}, it follows that $\widehat{u}(t)$   solves:
\begin{align}
\label{Hat3}
\frac{d \widehat{u}}{dt}(t)=-\varepsilon\Delta^2\widehat{u}(t)+\frac{1}{\varepsilon}\Delta f(\widehat{u}(t)+\widetilde{u}(t))\quad t\in(0, T],\; \widehat{u}(0)=u_0. 
\end{align}
Testing \eqref{Hat3} with $(-\Delta)^{-1}\widehat{u}(t)$ and following the same approach as in \cite[Theorem 3.1]{BanasDaniel23}, we obtain:
\begin{align}
\label{Hat4}
\sup_{t\in[0, T]}\Vert \widehat{u}(t)\Vert^2_{\mathbb{H}^{-1}}+\varepsilon\int_0^T\Vert\nabla\widehat{u}(s)\Vert^2ds+\frac{1}{4\varepsilon}\int_0^T\Vert u(s)\Vert^4_{\mathbb{L}^4}ds\leq \frac{C}{\varepsilon}+\frac{C}{\varepsilon}\int_0^T\Vert \widetilde{u}(s)\Vert^4_{\mathbb{L}^4}ds.
\end{align}
Squaring \eqref{Hat4}, using the embedding $\mathbb{L}^q\hookrightarrow\mathbb{L}^p$ ($1\leq p\leq q$), taking the expectation in the resulting inequality, and applying \lemref{Normutilde}, we derive:
{\small
\begin{align}
\label{Trala1}
\mathbb{E}\left[\left(\sup_{t\in[0, T]}\Vert \widehat{u}(t)\Vert^2_{\mathbb{H}^{-1}}+\varepsilon\int_0^T\Vert\nabla\widehat{u}(s)\Vert^2ds+\frac{1}{4\varepsilon}\int_0^T\Vert u(s)\Vert^4_{\mathbb{L}^4}ds\right)^2\right]\leq C\varepsilon^{-2}+C\tilde{h}^{-12d}\varepsilon^{-4}. 
\end{align}
}
It remains to estimate the term involving $\widehat{u}_{h,\tau}$. First, let us recall that $\widehat{u}_{h,\tau}$ satifisfies:
\begin{subequations}
\begin{align}
\label{final1}
(\partial_t\widehat{u}_{h,\tau}(t), \varphi_h)+(\nabla\widehat{w}_{h,\tau}(t), \nabla\varphi_h)=\langle \widehat{\mathcal{R}}(t), \varphi_h\rangle & \qquad \varphi_h\in\mathbb{V}^n_h,\\
\label{final2}
\varepsilon(\nabla\widehat{u}_{h,\tau}(t), \nabla\psi_h)=(\widehat{w}_{h,\tau}(t), \psi_h)-\varepsilon^{-1}\left(f(u_{h,\tau}(t)), \psi_h\right)+\langle \widehat{\mathcal{S}}(t), \psi_h\rangle & \qquad \psi_h\in\mathbb{V}^n_h,
\end{align}
\end{subequations}
 for all $t\in [t_{n-1}, t_{n}]$.
 
  Taking $\varphi=(-\Delta_h)^{-1}\widehat{u}_{h,\tau}(t)$ in \eqref{final1} and $\psi=\widehat{u}_{h,\tau}(t)$ in \eqref{final2} yields:
 \begin{align}
 \label{final2a}
& \frac{1}{2}\frac{d}{dt}\Vert \widehat{u}_{h,\tau}(t)\Vert^2_{-1,h}+\varepsilon\Vert\nabla\widehat{u}_{h,\tau}(t)\Vert^2+\frac{1}{\varepsilon}\left(f(u_{h,\tau}(t)), \widehat{u}_{h,\tau}(t)\right)\nonumber\\
&\quad=\langle \widehat{\mathcal{R}}(t), (-\Delta_h)^{-1}\widehat{u}_{h,\tau}(t)\rangle+\langle \widehat{\mathcal{S}}(t), \widehat{u}_{h,\tau}(t)\rangle. 
 \end{align}
 Using Cauchy-Schwarz, Poincar\'{e}, and Young's inequalities, it follows from \eqref{final2a} that:
 \begin{align*}
 \frac{1}{2}\frac{d}{dt}\Vert \widehat{u}_{h,\tau}(t)\Vert^2_{-1,h}+\frac{3\varepsilon}{4}\Vert\nabla\widehat{u}_{h,\tau}(t)\Vert^2&\leq \frac{C}{\varepsilon^3}\Vert f(u_{h,\tau}(t))\Vert^2\nonumber\\
 &\quad+\langle \widehat{\mathcal{R}}(t), (-\Delta_h)^{-1}\widehat{u}_{h,\tau}(t)\rangle+\langle \widehat{\mathcal{S}}(t), \widehat{u}_{h,\tau}(t)\rangle. 
 \end{align*}
 Using \lemref{Residual} and noting $\Vert \nabla(-\Delta_h^{-1}u_h)\Vert\leq \Vert u_h\Vert$  we obtain
 \begin{align*}
& \frac{1}{2}\frac{d}{dt}\Vert \widehat{u}_{h,\tau}(t)\Vert^2_{-1,h}+\frac{3\varepsilon}{4}\Vert\nabla\widehat{u}_{h,\tau}(t)\Vert^2\nonumber\\
 &\leq \frac{C}{\varepsilon^3}\Vert f(u_{h,\tau}(t))\Vert^2+\widehat{\mu}_{-1}(t)\Vert \nabla(-\Delta)^{-1}\widehat{u}_{h,\tau}(t)\Vert+\widehat{\mu}_0(t)\Vert \widehat{u}_{h,\tau}(t)\Vert +\widehat{\mu}_1(t)\Vert\nabla\widehat{u}_{h,\tau}(t)\Vert\\
 &\leq \frac{C}{\varepsilon^3}\Vert f(u_{h,\tau}(t))\Vert^2+C\widehat{\mu}_{-1}(t)\Vert \widehat{u}_{h,\tau}(t)\Vert+\widehat{\mu}_0(t)\Vert \widehat{u}_{h,\tau}(t)\Vert +\widehat{\mu}_1(t)\Vert\nabla\widehat{u}_{h,\tau}(t)\Vert. \nonumber
 \end{align*}
 Using Poincar\'{e}'s and Young's inequalities, it follows from the preceding estimate that
 \begin{align}
 \label{final2b}
\frac{1}{2}\frac{d}{dt}\Vert \widehat{u}_{h,\tau}(t)\Vert^2_{-1,h}+\frac{\varepsilon}{2}\Vert\nabla\widehat{u}_{h,\tau}(t)\Vert^2\leq \frac{C}{\varepsilon^3}\Vert f(u_{h,\tau}(t))\Vert^2+\frac{C}{\varepsilon}\left(\widehat{\mu}_{-1}^2(t)+\widehat{\mu}^2_0(t) +\widehat{\mu}_1^2(t)\right). 
 \end{align}
 
 Integrating \eqref{final2b} over $(0, t)$, taking the supremum over $[0, T]$, squaring both sides of the resulting inequality, using the embedding $\mathbb{L}^{\infty}\hookrightarrow\mathbb{L}^{2}$, and applying \eqref{normeq}, we obtain:
 \begin{align}
 \label{Trala2}
&\mathbb{E}\left[\left(\sup_{t\in[0, T]}\Vert \widehat{u}_{h, \tau}(t)\Vert^2_{\mathbb{H}^{-1}}+\frac{\varepsilon}{2}\int_0^T\Vert\nabla\widehat{u}_{h,\tau}(s)\Vert^2ds\right)^2\right]\nonumber\\
&\leq C\varepsilon^{-6}\mathbb{E}\left[\sup_{t\in[0, T]}\Vert u_{h,\tau}(t)\Vert^{12}_{\mathbb{L}^{\infty}}\right]+C\varepsilon^{-2}\left(\int_0^T\mathbb{E}\left[\widehat{\mu}_{-1}^2(t)+\widehat{\mu}^2_0(t) +\widehat{\mu}_1^2(t)\right]dt\right)^2. 
\end{align}
Using the embedding $\mathbb{L}^{\infty}\hookrightarrow\mathbb{L}^{4}$, we get
\begin{align}
\label{Trala3}
\mathbb{E}\left[\left(\frac{1}{\varepsilon}\int_0^T\Vert u_{h,\tau}\Vert^4_{\mathbb{L}^4}(s)ds\right)^2\right]\leq C\varepsilon^{-2}\mathbb{E}\left[\sup_{t\in[0, T]} \Vert u_{h,\tau}(t)\Vert^8_{\mathbb{L}^{\infty}}\right].
\end{align}
Combining \eqref{Trala3}, \eqref{Trala2}, and \eqref{Trala1} concludes the proof.
\end{proof}


The theorem below provides an estimate for the approximation error of the numerical scheme \eqref{scheme1} and is the main result of this paper.
\begin{theorem}
\label{mainresult3}
Let $u$ be the weak solution to \eqref{pb1}, and let $u_{h,\tau}$ be given by \eqref{Interpolant1}.
If the assumptions of \thmref{mainresult2} are satisfied, then it holds that:
\begin{align*}
&\mathbb{E}\left[\sup_{t\in[0, T]}\Vert u_{h,\tau}(t)-u(t)\Vert^2_{\mathbb{H}^{-1}}\right]+\varepsilon\int_0^T\mathbb{E}\left[\Vert\nabla\left(u_{h,\tau}(s)-u(s)\right)\Vert^2\right]ds
\nonumber\\
&\qquad\leq C\left\{\widetilde{\mathcal{R}}+\mathbb{E}\left[1\!\!1_{\Omega_{\delta, \tilde{\varepsilon}}\cap\Omega_{\gamma, \tilde{\varepsilon}}\cap\Omega_{\tilde{\varepsilon}}}\widehat{\mathcal{R}}\right]+\left(\tilde{\varepsilon}^{\delta}\varepsilon^{-3}+\tilde{h}^{-6d}\varepsilon^2\widetilde{\varepsilon}^{\gamma}+\tilde{\varepsilon}^{-1}\widetilde{\mathcal{R}}\right)^{1/2}\left(\widehat{C}_{0,h}+ \mathbb{E}[\widehat{\mathcal{R}}_{\mu}]\right)\right\},
\end{align*}
where the constant $\widehat{C}_{0,h}$ is defined as:
\begin{align}
\label{ConstantC0}
\widehat{C}_{0,h}=C\left(\varepsilon^{-2}\mathbb{E}\left[\sup_{t\in[0, T]} \Vert u_{h,\tau}(t)\Vert^8_{\mathbb{L}^{\infty}}\right]+\varepsilon^{-6}\mathbb{E}\left[\sup_{t\in[0, T]}\Vert u_{h,\tau}(t)\Vert^{12}_{\mathbb{L}^{\infty}}\right]+\varepsilon^{-2}+\tilde{h}^{-12d}\varepsilon^{-4}\right)^{\frac{1}{2}},
\end{align}
and the residual $\widehat{\mathcal{R}}_{\mu}$ is given by
\begin{align}
\label{Residualmu}
\widehat{\mathcal{R}}_{\mu}=C\varepsilon^{-2}\int_0^T\left[\widehat{\mu}_{-1}^2(t)+\widehat{\mu}^2_0(t) +\widehat{\mu}_1^2(t)\right]dt.
\end{align}

\end{theorem}
\begin{proof}
Noting that $e=\widehat{e}+\widetilde{e}$ and using the triangle inequality, we obtain:
\begin{align}
\label{FullError0}
&\mathbb{E}\left[\sup_{t\in[0, T]}\Vert e(t)\Vert^2_{\mathbb{H}^{-1}}\right]+\varepsilon\int_0^T\mathbb{E}[\Vert\nabla e(s)\Vert^2]ds\nonumber\\
 &\qquad\leq \left\{\mathbb{E}\left[\sup_{t\in[0, T]}\Vert \widetilde{e}(t)\Vert^2_{\mathbb{H}^{-1}}\right]+\varepsilon\int_0^T\mathbb{E}[\Vert\nabla \widetilde{e}(s)\Vert^2]ds\right\}\\
 &\qquad\qquad+\left\{\mathbb{E}\left[\sup_{t\in[0, T]}\Vert \widehat{e}(t)\Vert^2_{\mathbb{H}^{-1}}\right]+\varepsilon\int_0^T\mathbb{E}[\Vert\nabla \widehat{e}(s)\Vert^2]ds\right\}=:\text{VIII}_1+\text{VIII}_2.\nonumber
\end{align}
The term $\text{VIII}_1$ is estimated in \lemref{LemmaErrortilde}. To estimate $\text{VIII}_2$, we  split it as follows:
\begin{align}
\label{FullError1}
\text{VIII}_2&=\mathbb{E}\left[\sup_{t\in[0, T]}\Vert \widehat{e}(t)\Vert^2_{\mathbb{H}^{-1}}\right]+\varepsilon\int_0^T\mathbb{E}\left[\Vert\nabla \widehat{e}(s)\Vert^2\right]ds\nonumber\\
&\leq \left\{\mathbb{E}\left[\sup_{t\in[0, T]}1\!\!1_{\Omega_{\delta,\tilde{\varepsilon}}\cap\Omega_{\gamma, \tilde{\varepsilon}}\cap\Omega_{\tilde{\varepsilon}}}\Vert \widehat{e}(t)\Vert^2_{\mathbb{H}^{-1}}\right]+\varepsilon\int_0^T\mathbb{E}\left[1\!\!1_{\Omega_{\delta, \tilde{\varepsilon}}\cap\Omega_{\gamma, \tilde{\varepsilon}}\cap\Omega_{\tilde{\varepsilon}}}\Vert\nabla \widehat{e}(s)\Vert^2\right]ds\right\}\\
&\qquad+\left\{\mathbb{E}\left[\sup_{t\in[0, T]}1\!\!1_{\left(\Omega_{\delta, \tilde{\varepsilon}}\cap\Omega_{\gamma, \tilde{\varepsilon}}\cap\Omega_{\tilde{\varepsilon}}\right)^c}\Vert \widehat{e}(t)\Vert^2_{\mathbb{H}^{-1}}\right]+\varepsilon\int_0^T\mathbb{E}\left[1\!\!1_{\left(\Omega_{\delta, \tilde{\varepsilon}}\cap\Omega_{\gamma, \tilde{\varepsilon}}\cap\Omega_{\tilde{\varepsilon}}\right)^c}\Vert\nabla \widehat{e}(s)\Vert^2\right]ds\right\}\nonumber\\
&=:\text{VIII}_{21}+\text{VIII}_{22}. \nonumber
\end{align}
The term $\text{VIII}_{21}$ is estimated using \propref{mainresult2}.
 To estimate $\text{VIII}_{22}$ we note that
\begin{align*}
\sup_{t\in[0, T]}\left(1\!\!1_{\left(\Omega_{\delta, \tilde{\varepsilon}}\cap\Omega_{\gamma, \tilde{\varepsilon}}\cap\Omega_{\tilde{\varepsilon}}\right)^c}\Vert \widehat{e}(t)\Vert^2_{\mathbb{H}^{-1}}\right)\leq 1\!\!1_{\left(\Omega_{\delta, \tilde{\varepsilon}}\cap\Omega_{\gamma, \tilde{\varepsilon}}\cap\Omega_{\tilde{\varepsilon}}\right)^c}\sup_{t\in[0, T]}\Vert \widehat{e}(t)\Vert^2_{\mathbb{H}^{-1}},
\end{align*}
 and use Cauchy-Schwarz's inequality to get
\begin{align}
\label{FullError1a}
\text{VIII}_{22}&\leq \left(\mathbb{P}[\Omega_{\delta,\tilde{\varepsilon}}^c\cup\Omega_{\gamma, \tilde{\varepsilon}}^c\cup\Omega_{\tilde{\varepsilon}}^c]\right)^{\frac{1}{2}}\left(\mathbb{E}\left[\left(\sup_{t\in[0, T]}\Vert\widehat{e}(t)\Vert^2_{\mathbb{H}^{-1}}+\varepsilon\int_0^T\Vert\nabla\widehat{e}(s)\Vert^2ds\right)^{\color{red}2}\right]\right)^{\frac{1}{2}}.
\end{align}
Using Markov's inequality and \cite[Proposition 3.1]{BanasDaniel23}, we derive:
\begin{align*}
\mathbb{P}[\Omega_{\delta, \tilde{\varepsilon}}^c]\leq \tilde{\varepsilon}^{\delta}\mathbb{E}\left[\sup_{t\in[0, T]}\Vert u(t)\Vert^2_{\mathbb{H}^{-1}}+\varepsilon^{-1}\int_0^T\Vert u(s)\Vert^4_{\mathbb{L}^4}ds\right]\leq C\tilde{\varepsilon}^{\delta}\varepsilon^{-3}.
\end{align*}
Using Markov's inequality and \lemref{Normutilde} with $p=4$, we obtain:
\begin{align*}
\mathbb{P}[\Omega_{\gamma, \tilde{\varepsilon}}^c]\leq \tilde{\varepsilon}^{\gamma}\mathbb{E}\left[\sup_{t\in[0, T]}\Vert \widetilde{u}(t)\Vert^4_{\mathbb{L}^4}\right]\leq C\tilde{h}^{-6d}\varepsilon^{2}\tilde{\varepsilon}^{\gamma}.
\end{align*}
Using Markov's inequality and \lemref{LemmaErrortilde}, we derive the following estimate:
\begin{align*}
\mathbb{P}[\Omega_{\tilde{\varepsilon}}^c]\leq \tilde{\varepsilon}^{-1}\mathbb{E}\left[\sup_{t\in[0, T]}\Vert \widetilde{e}(t)\Vert^2_{\mathbb{H}^{-1}}+\varepsilon^{-1}\int_0^T\Vert \widetilde{e}(s)\Vert^4_{\mathbb{L}^4}ds\right]\leq C\tilde{\varepsilon}^{-1}\mathbb{E}[\widetilde{\mathcal{R}}]. 
\end{align*}
Using the preceding estimates, we obtain
\begin{align*}
\mathbb{P}[\Omega_{\delta, \tilde{\varepsilon}}\cup\Omega_{\gamma, \tilde{\varepsilon}}\cup\Omega_{\tilde{\varepsilon}}]\leq \mathbb{P}[\Omega_{\delta, \tilde{\varepsilon}}^c]+\mathbb{P}[\Omega_{\gamma, \tilde{\varepsilon}}^c]+\mathbb{P}[\Omega_{\tilde{\varepsilon}}^c]\leq 
C\left(\tilde{\varepsilon}^{\delta}\varepsilon^{-3}+\tilde{h}^{-6d}\varepsilon^2\tilde{\varepsilon}^{\gamma}+\tilde{\varepsilon}^{-1}\mathbb{E}[\tilde{\mathcal{R}}]\right).
\end{align*}
Hence substitute the above estimate into \eqref{FullError1a} and use \lemref{LemmaErrorhat2} to conclude
\begin{align}
\label{FullError1e}
\text{VIII}_{22}\leq C\left(\tilde{\varepsilon}^{\delta}\varepsilon^{-3}+\tilde{h}^{-6d}\varepsilon^2\tilde{\varepsilon}^{\gamma}+\tilde{\varepsilon}^{-1}\mathbb{E}[\tilde{\mathcal{R}}]\right)^{{\color{red} 1/2}}
\left(\widehat{C}_{0,h}+\revdr{\mathbb{E}[\widehat{\mathcal{R}}_{\mu}]}\right).
\end{align}
Substituting \eqref{FullError1e} into \eqref{FullError1}, and applying \propref{mainresult2}, we get:
\begin{align*}
\text{VIII}_2\leq \mathbb{E}\left[1\!\!1_{\Omega_{\delta, \tilde{\varepsilon}}\cap\Omega_{\gamma, \tilde{\varepsilon}}\cap\Omega_{\tilde{\varepsilon}}}\widehat{\mathcal{R}}\right]+C\left(\tilde{\varepsilon}^{\delta}\varepsilon^{-3}+\tilde{h}^{-6d}\varepsilon^2\tilde{\varepsilon}^{\gamma}+\tilde{\varepsilon}^{-1}\mathbb{E}[\tilde{\mathcal{R}}]\right)^{1/2}\left(\widehat{C}_{0,h}+\mathbb{E}[\widehat{\mathcal{R}}_{\mu}]\right).
\end{align*}
Substituting the estimate above into \eqref{FullError0} and applying \propref{mainresult1} completes the proof. 
\end{proof}

\section{Numerical experiments}\label{sec_num}

We consider the regularized problem \eqref{pb1} on the spatial domain $\mathcal{D}=(-1,1)^2$ with initial condition
$$
 u_0^\eps(x) = -\tanh\left(\frac{\max\{-(|x|-r_1), |x|-r_2\}}{\sqrt{2}\varepsilon}\right)\,,
$$
with $r_1 = 0.2$, $r_2 = 0.55$ and the interfacial width parameter $\eps = \frac{1}{32}$.
We consider the noise approximation \eqref{Noiseapprox2} for $\tilde{h} = \frac{1}{16}$ and $\tilde{h} = \frac{1}{32}$;
the noise term is scaled by an additional factor $\sigma = 0.4$, i.e., we use $\sigma \Delta_n\widetilde{W}$ in (\ref{scheme1}).
The simulation is performed for $T= 0.012$ and we employ a uniform time step $\tau_n = \tau = 10^{-6}$ in \eqref{scheme1}.

We employ a simple time-explicit algorithm for (pathwise) adaptive mesh refinement: we choose $h_{min}=\frac{1}{128}$ and
given the triangulation $\mathcal{T}_h^{n-1}$ we compute (the realization of) the solution $u_h^n\in \mathbb{V}_h^n(\mathcal{T}_h^{n-1})$.
The triangulation $\mathcal{T}_h^{n}$ for the next time level is then constructed using the computed value of $u_h^n$ as follows.
We set $\eta_{max} := \|\Delta_h u_h^n\|_{\mathbb{L}^{\infty}}$ and refine the mesh until $h_K \leq h_{min}$ for all triangles where $\Delta_h u_h^n|_{K} \geq 0.25 \eta_{max}$.
We coarsen all triangles $K$ where $\Delta_h u_h^n|_{K} \leq 0.1 \eta_{max}$ under the constraint that the coarsening does not violate the condition $h_K \leq \tilde{h}$
(to ensure the compatibility condition $\mathbb{V}_{\tilde{h}} \subset \mathbb{V}_h^n$). This approach results in meshes with mesh size $h_K \approx h_{min}$ along the interface of each realization
of the numerical solution (and $h_K\approx \tilde{h}$ away from the interface), see Figure~\ref{fig_mesh}

The snapshots of the computed solution at different times for $\tilde{h} = \frac{1}{32}$ are displayed in Figure~\ref{fig_uh}
and the corresponding adaptive finite element mesh is displayed in Figure~\ref{fig_mesh} (note that to simplify the implementation the noise at $t=0$ is approximated at a slightly coarser mesh away from the interface). The evolution
for $\tilde{h} = \frac{1}{16}$ exhibits no qualitatively significant differences on the graphical level.
\begin{figure}[htp!]
\includegraphics[width=0.32\textwidth]{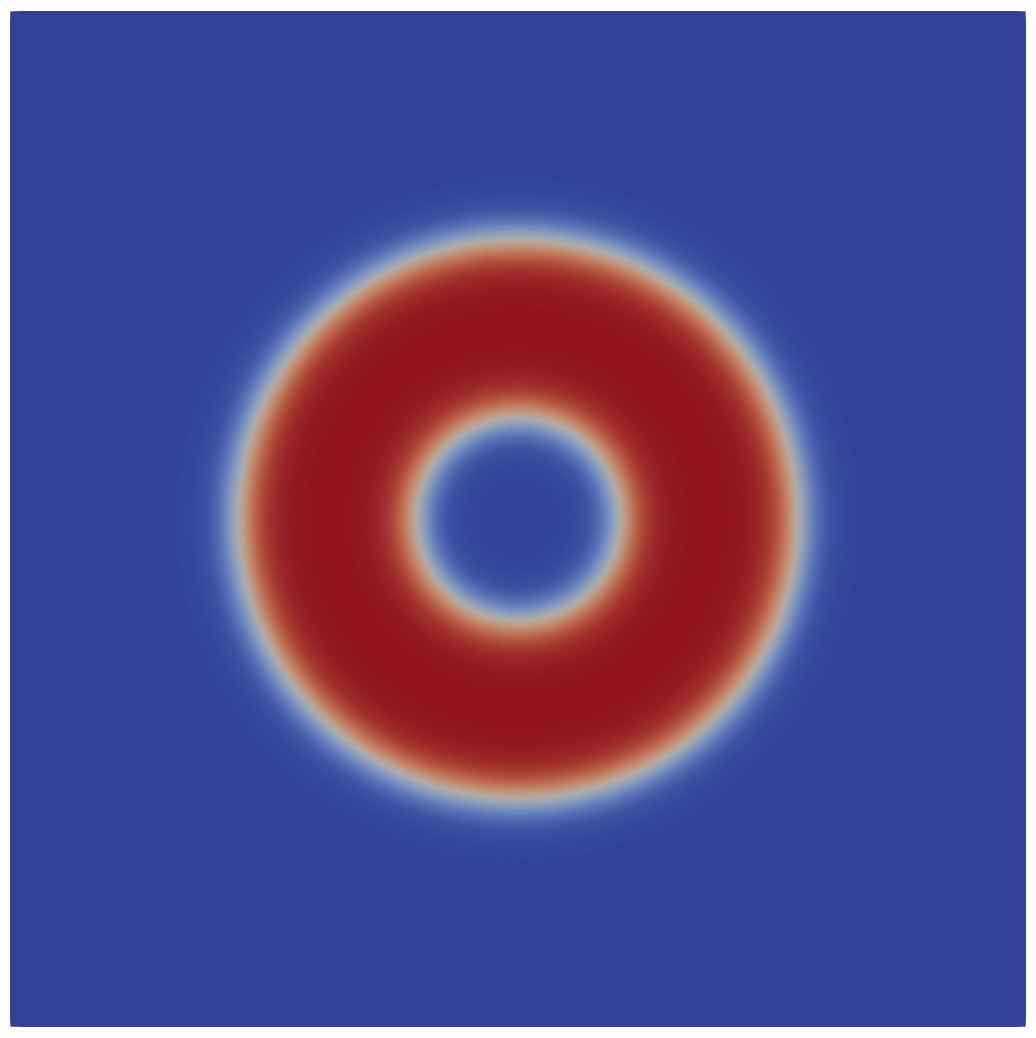}
\includegraphics[width=0.32\textwidth]{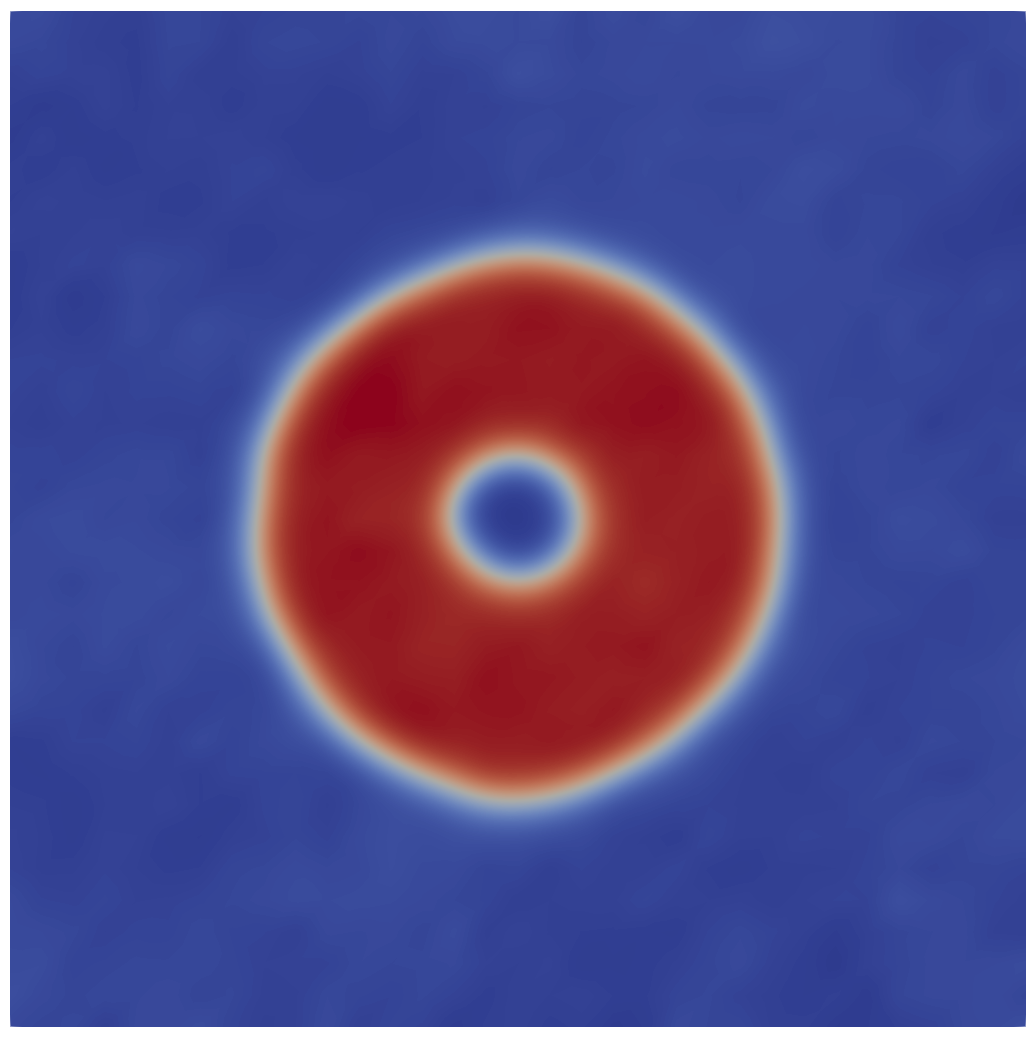}
\includegraphics[width=0.32\textwidth]{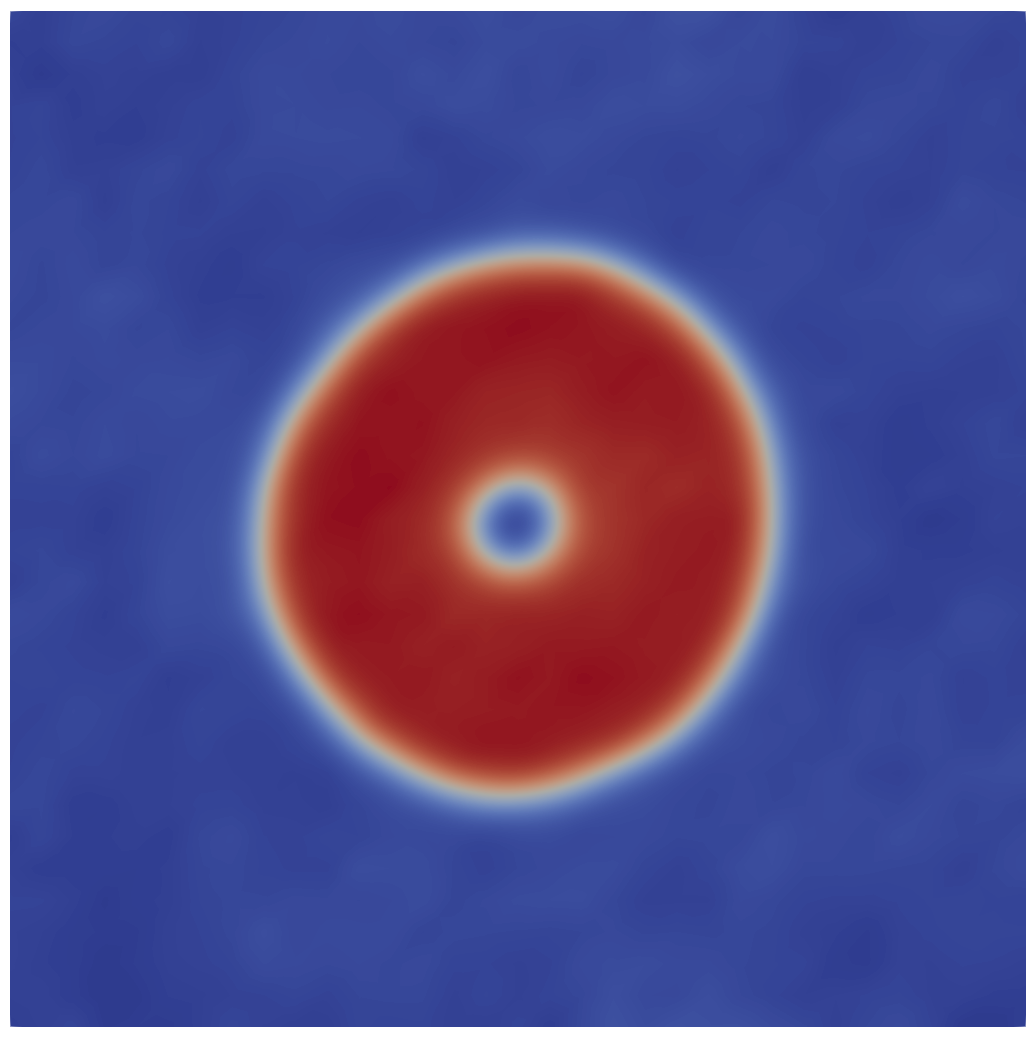}
\includegraphics[width=0.32\textwidth]{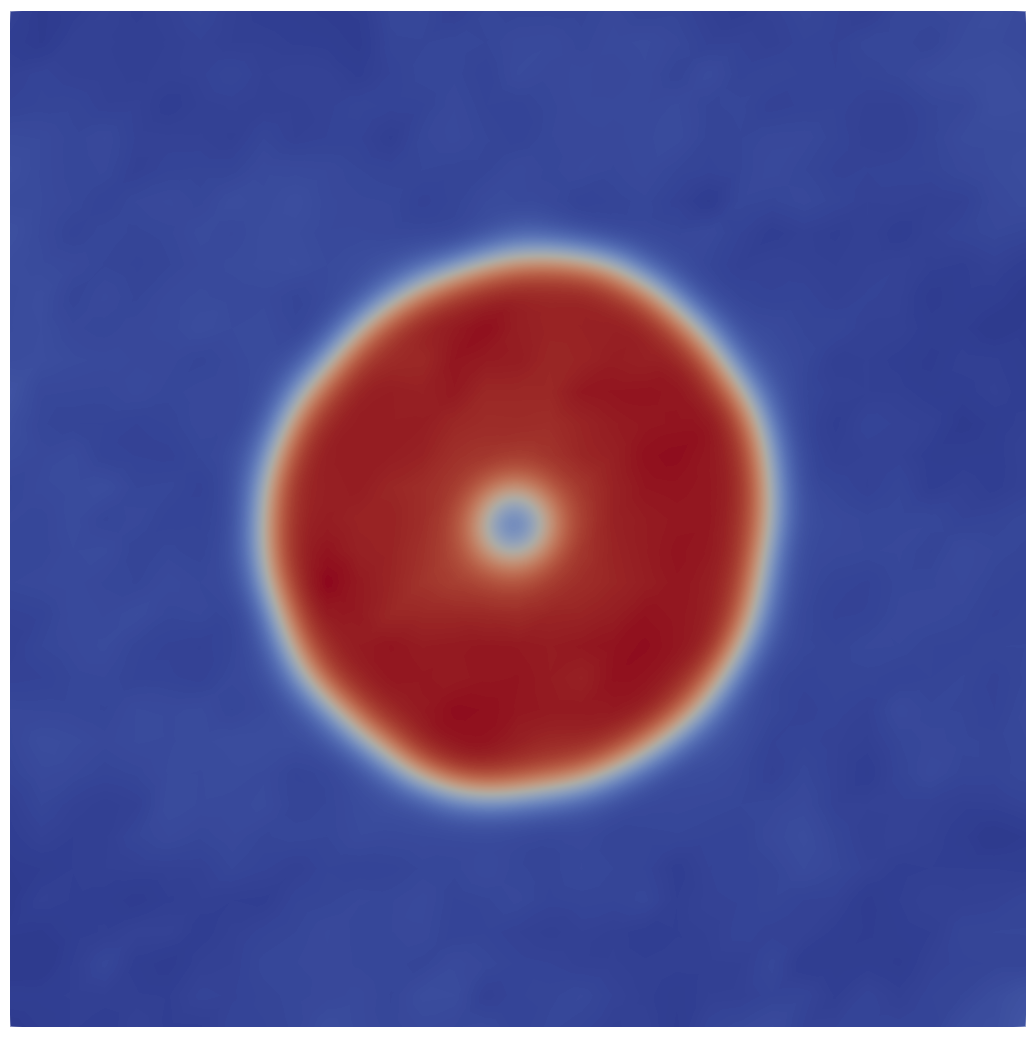}
\includegraphics[width=0.32\textwidth]{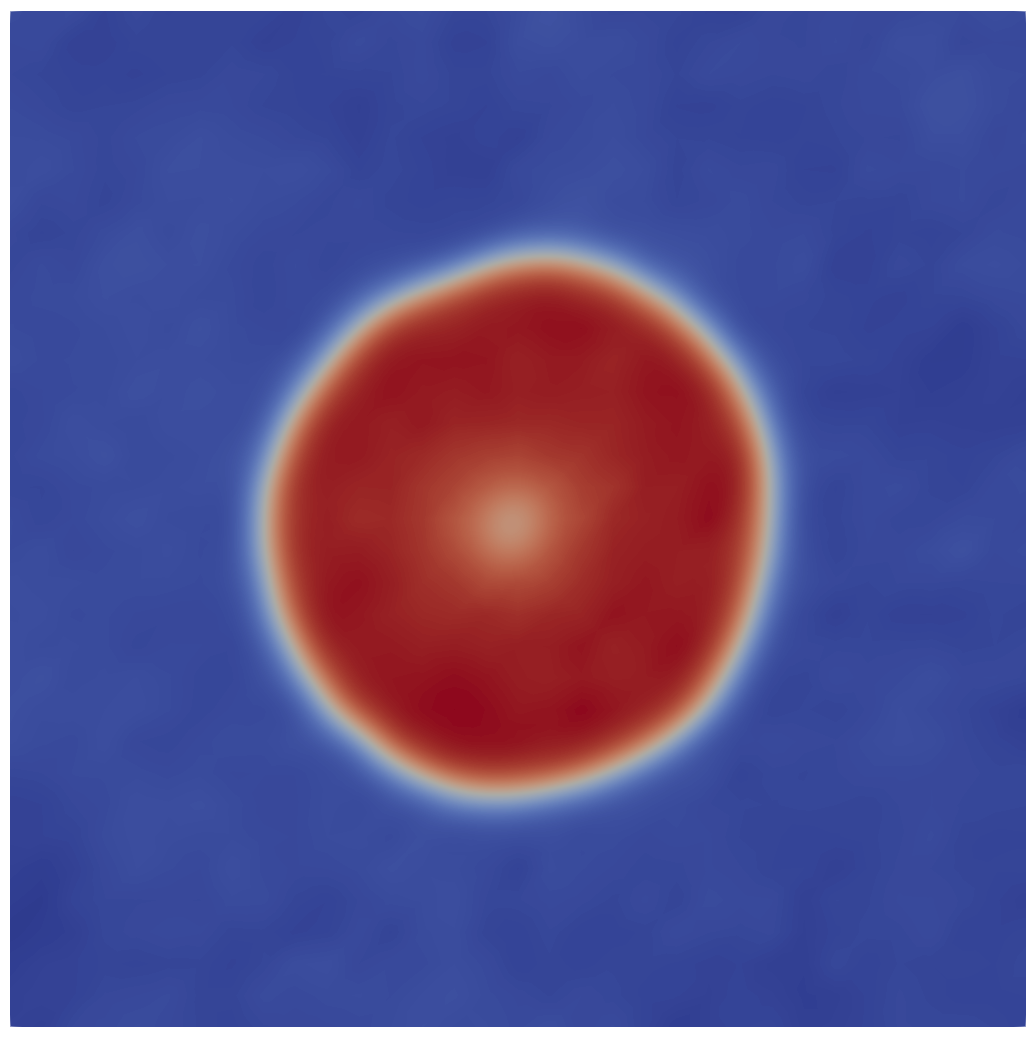}
\includegraphics[width=0.32\textwidth]{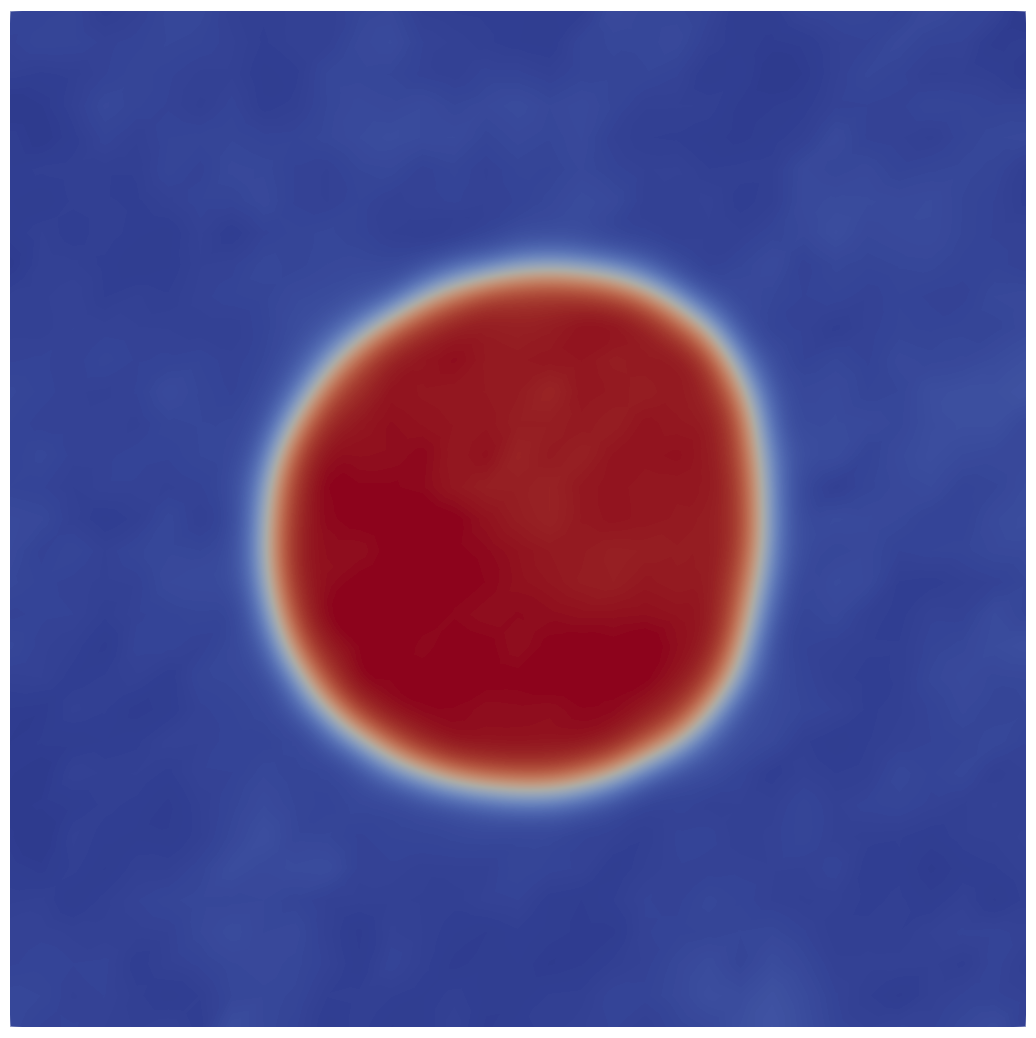}
\caption{Numerical solution at time $t=0,0.0065,0.009,0.0095,0.0097,0.012$.}
\label{fig_uh}
\end{figure}
\begin{figure}[htp!]
\includegraphics[width=0.32\textwidth]{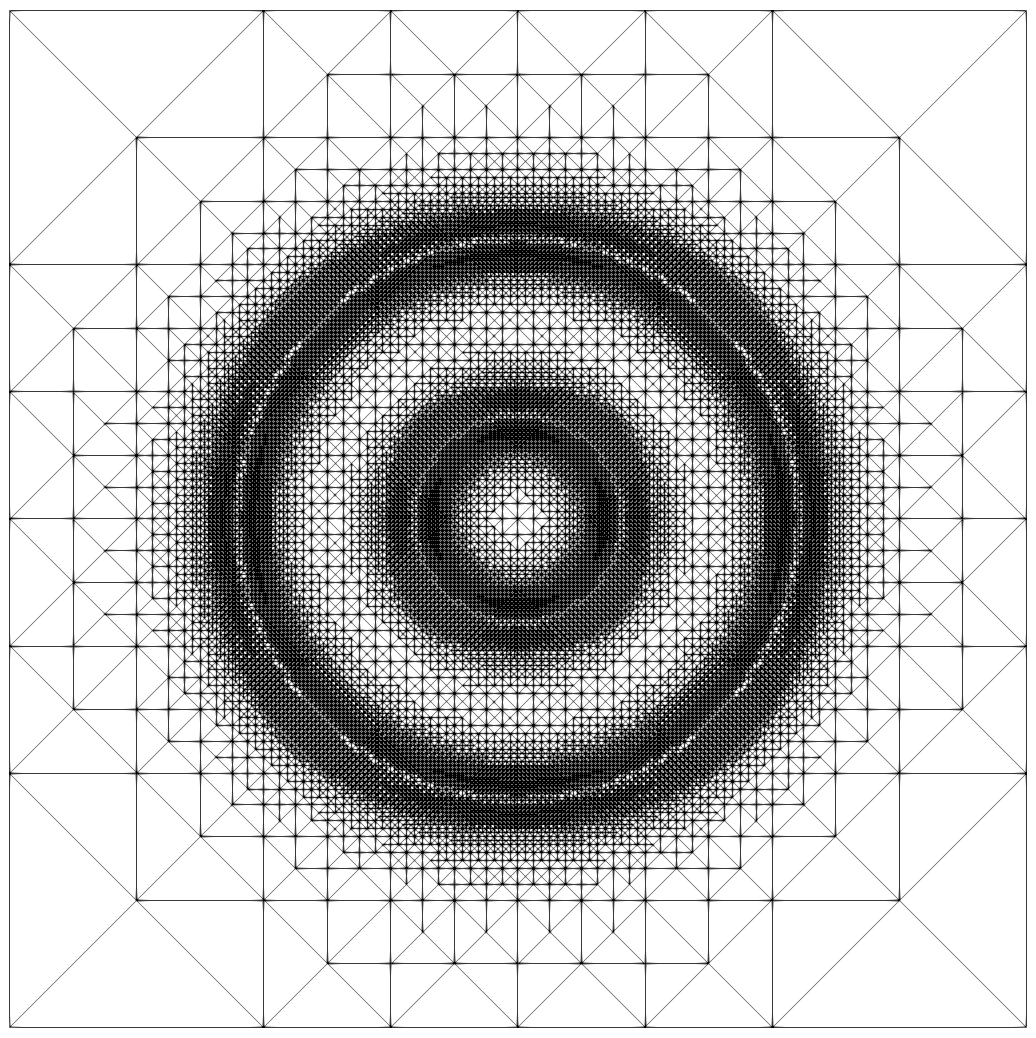}
\includegraphics[width=0.32\textwidth]{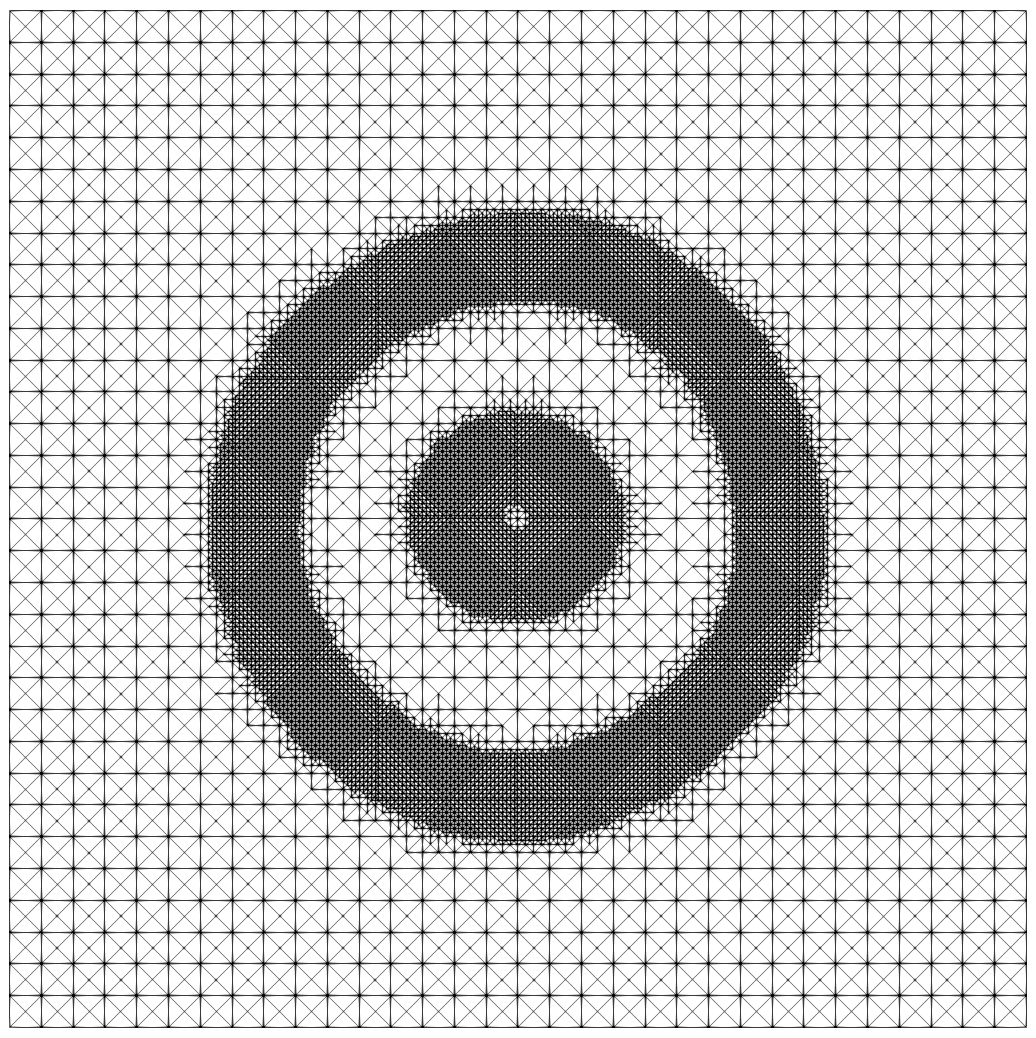}
\includegraphics[width=0.32\textwidth]{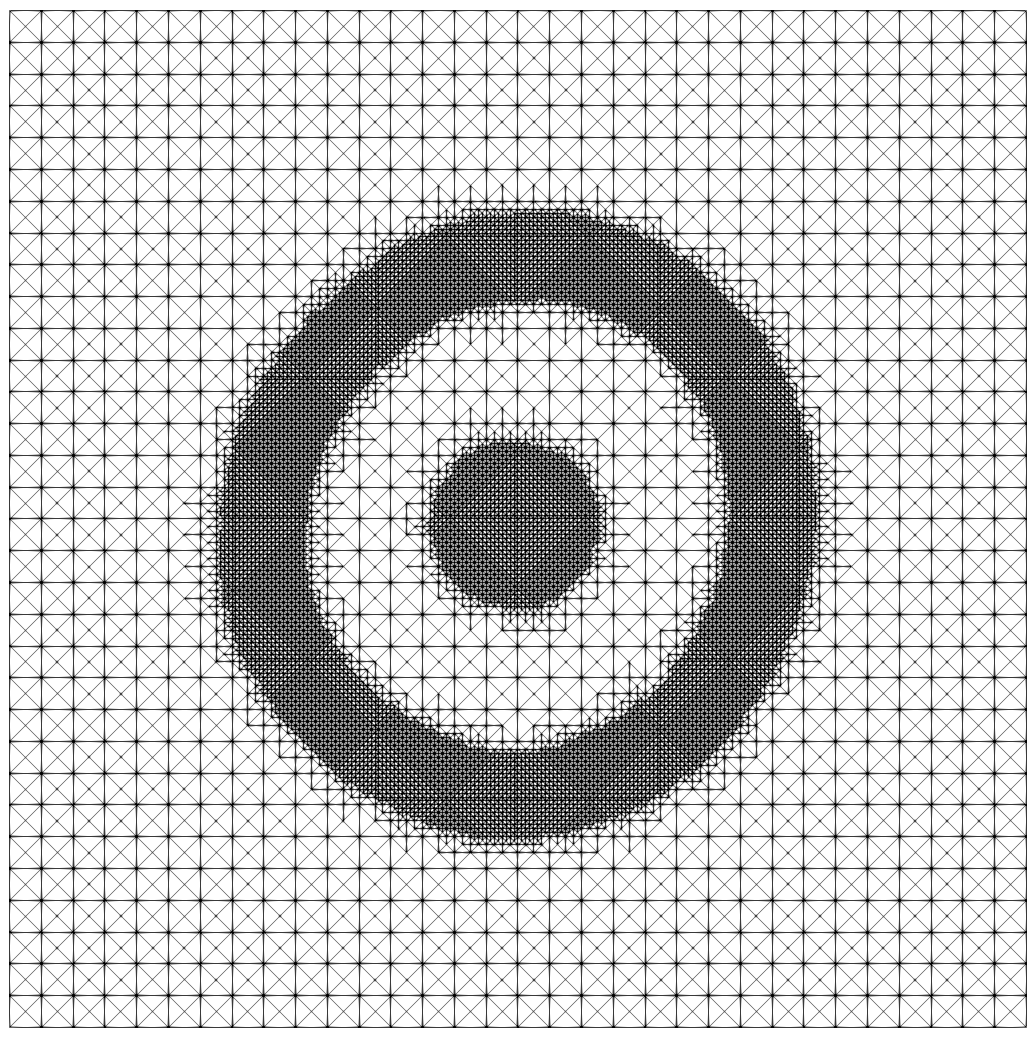}
\includegraphics[width=0.32\textwidth]{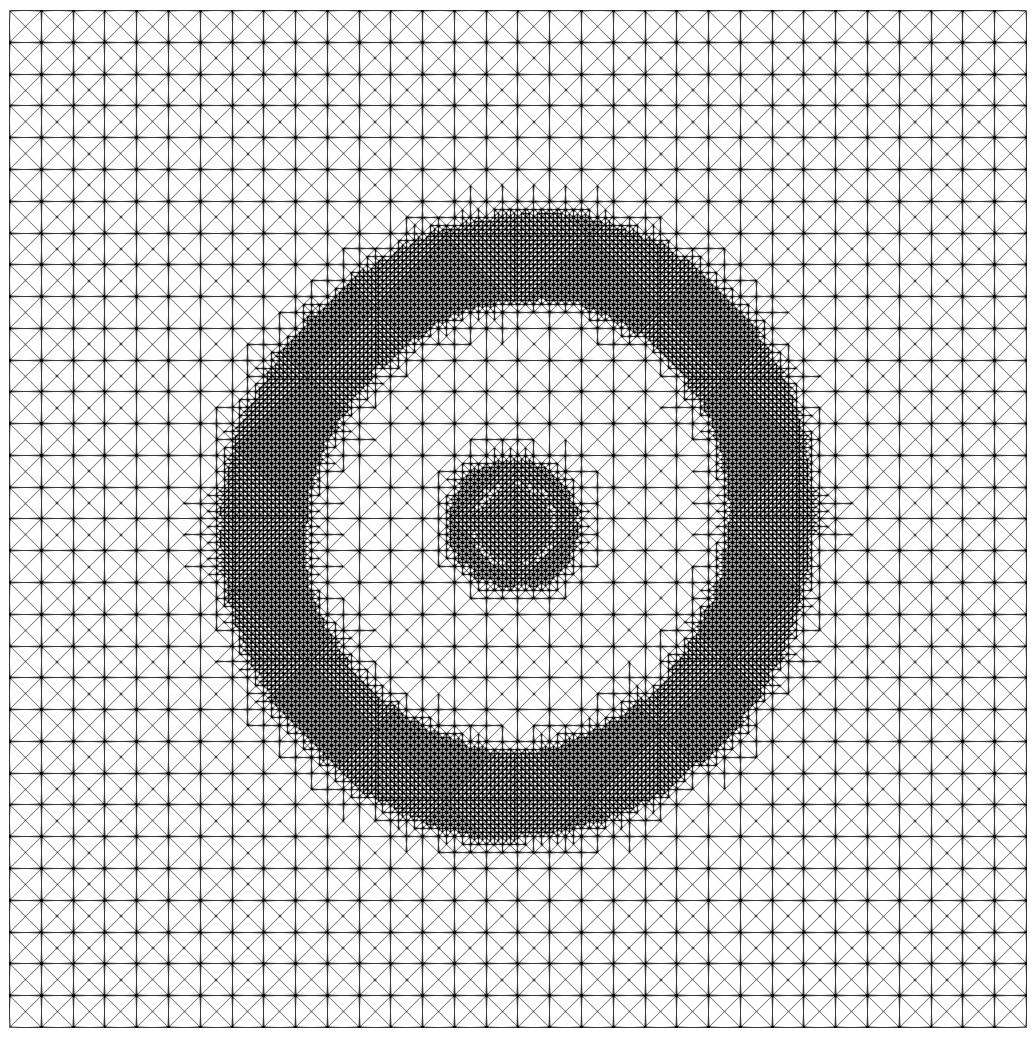}
\includegraphics[width=0.32\textwidth]{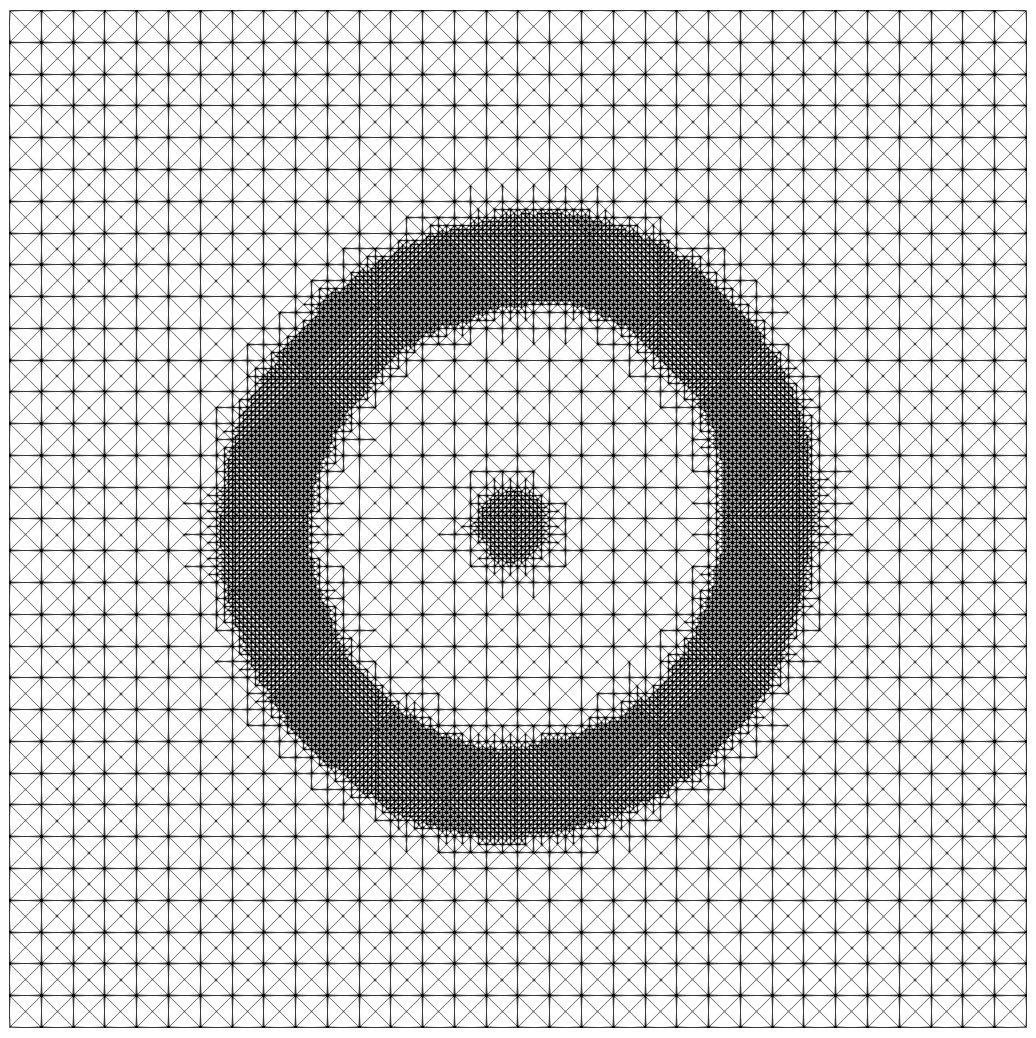}
\includegraphics[width=0.32\textwidth]{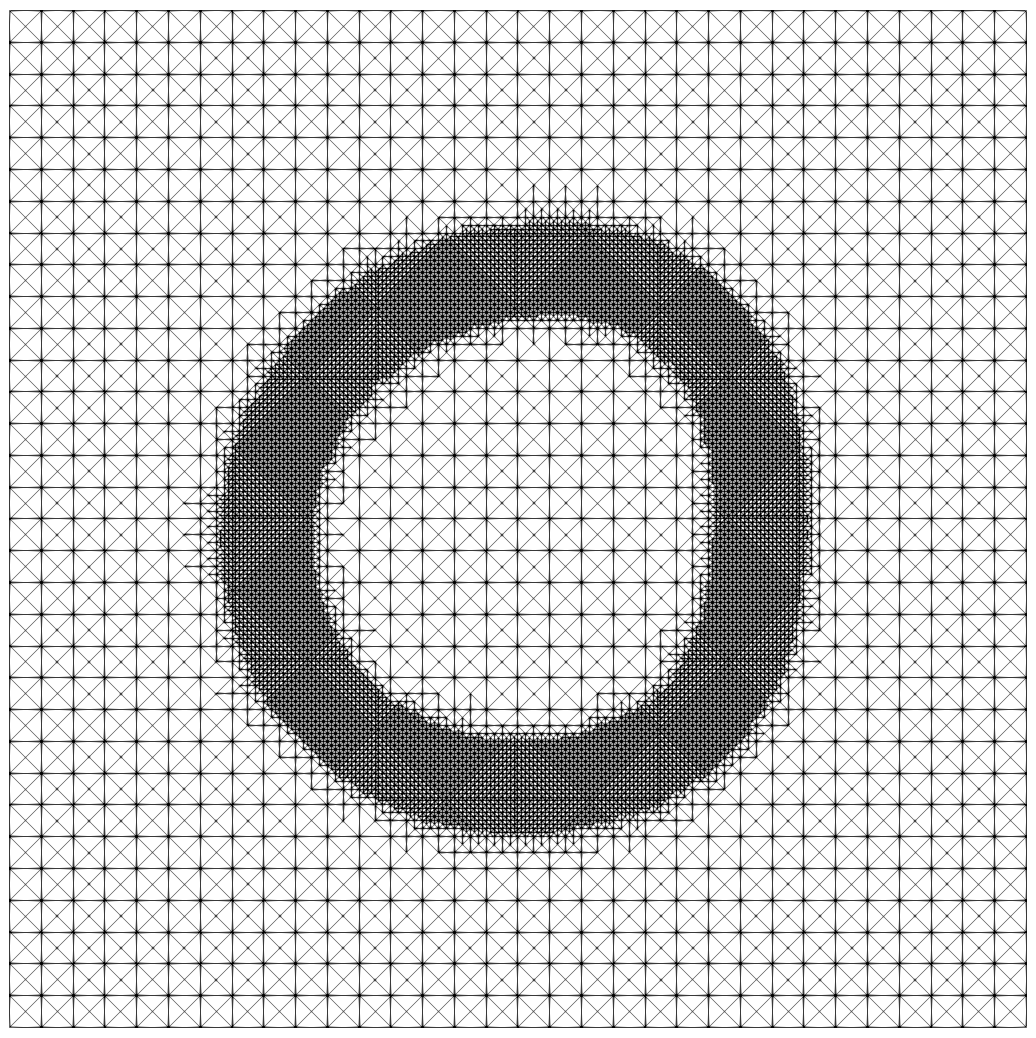}
\caption{Finite element mesh at time $t=0,0.0065,0.009,0.0095,0.0097,0.012$.}
\label{fig_mesh}
\end{figure}

The numerical solution computed for the considered initial condition evolves analogously to the deterministic setting and the stochastic setting with smooth noise, cf. \cite{BanasVieth_a_post23}:
both circles shrink until the inner circle disappear and the solution converges to a steady state which is represented by one circular interface.
The disappearance of the inner circle represents a topological change of the interface which is reflected by the peak
of the discrete principal eigenvalue (\ref{Eigenvalue1}), see Figure~\ref{fig_lambda} where we display the evolution of the principal eigenvalue for different realizations of the noise
with $\tilde{h} = \frac{1}{16}$ and $\tilde{h} = \frac{1}{32}$, respectively.
Apart from slightly larger oscillations for the finer discretisation of the noise, the evolution for both choices of $\tilde{h}$ is qualitatively similar.
\begin{figure}[htp!]
\includegraphics[width=0.45\textwidth]{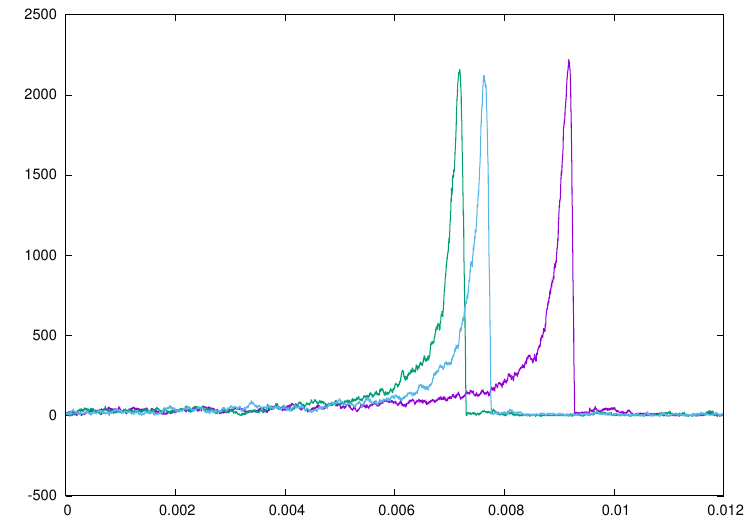}
\includegraphics[width=0.45\textwidth]{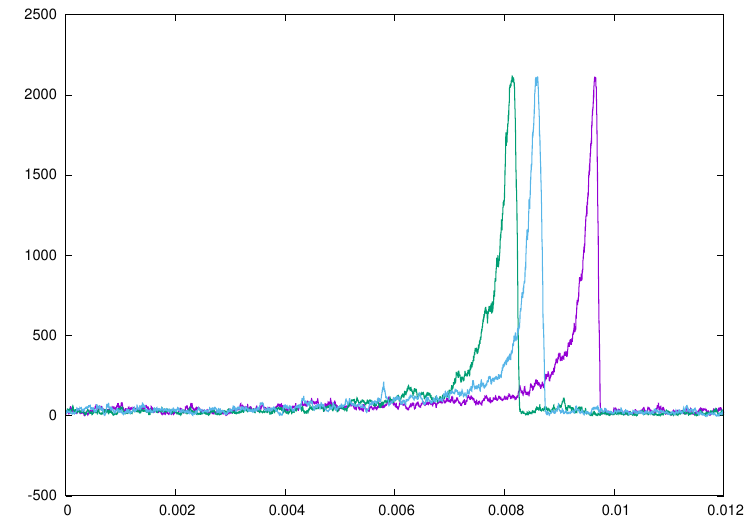}
\caption{Evolution of the discrete principal eigenvalue for different realizations of the noise for $\tilde{h}=1/16$ (left) and for $\tilde{h}=1/32$ (right)}
\label{fig_lambda}
\end{figure}

In Figure~\ref{fig_hist} we display the histogram of the peak-times of the discrete principal eigenvalue for $\tilde{h} = \frac{1}{16}, \frac{1}{32}$ (computed from  $2000$ and $4000$ realisations of the noise, respectively)
along with a (scaled) graph of the evolution of the discrete principal eigenvalue of the deterministic problem. Similarly as in the case of smooth noise \cite{BanasVieth_a_post23},
we observe that the probability of the peak-time in the stochastic case is higher close to the peak-time of the eigenvalue of the deterministic problem.
\begin{figure}[htp!]
\includegraphics[width=0.6\textwidth]{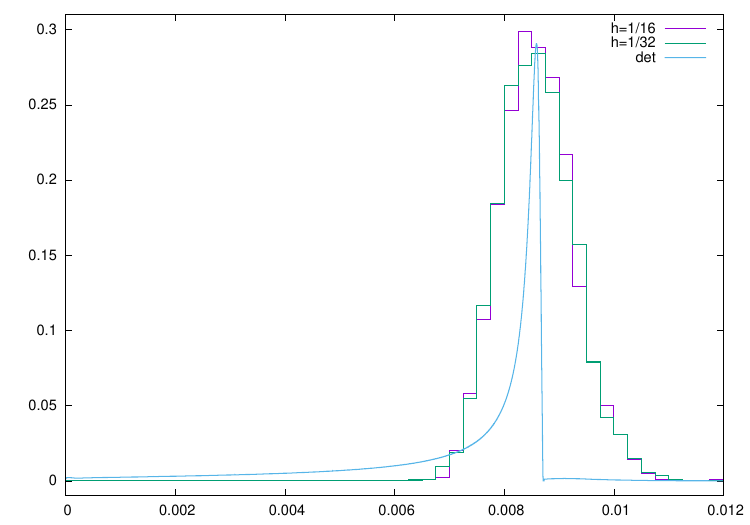}
\caption{Histogram of the peak-times of the principal eigenvalue for $\tilde{h}=1/16$, $\tilde{h}=1/32$ and the evolution of the (scaled) principal eigenvalue of the deterministic problem.}
\label{fig_hist}
\end{figure}

Finally, in Figure~\ref{fig_exp} we display the evolution of the expected value of the discrete energy $\mathcal{E}(u_h^n) = \eps \|\nabla u_h^n\|^2 + \frac{1}{\eps}\|F(u_h^n)\|_{\mathbb{L}^1}$ and of the expected value of discrete principal eigenvalue
as well as the evolution of the corresponding respective value for the deterministic problem. Analogously to the smooth noise case, cf. \cite{BanasVieth_a_post23}, the displayed results indicate that, on average, the topological change of the interface
occurs earlier than in the deterministic setting. Moreover, we observe only minor dependence of the discrete energy on the noise discretisation parameter $\tilde{h}$.
\begin{figure}[htp!]
\includegraphics[width=0.45\textwidth]{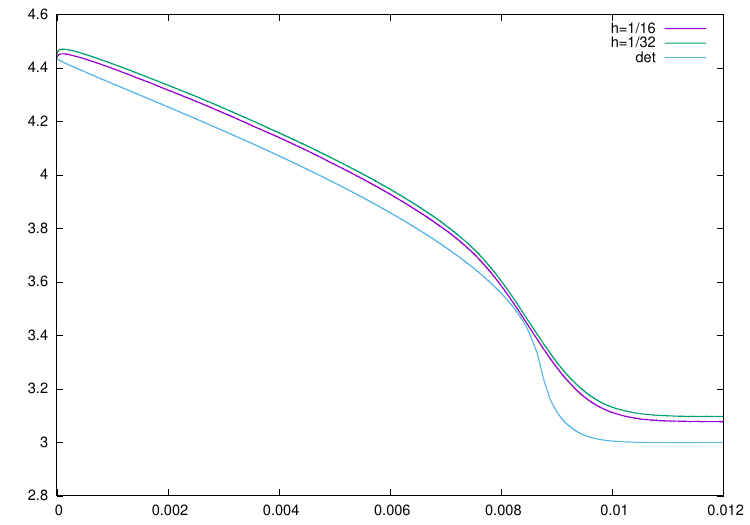}
\includegraphics[width=0.45\textwidth]{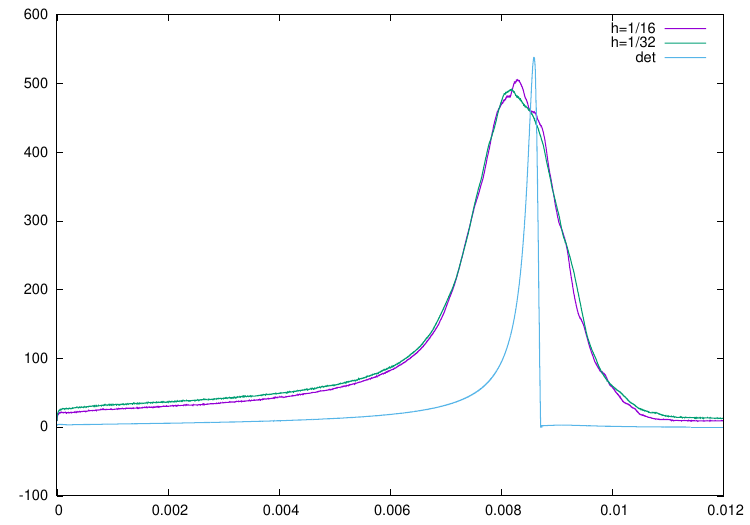}
\caption{Evolution of the expected value of the discrete energy (left) and of the principal eigenvalue (right) for $\tilde{h}=1/16$, $\tilde{h}=1/32$ and for the deterministic problem.}
\label{fig_exp}
\end{figure}

\appendix
\section{Regularity estimates of the solution to the stochastic Cahn-Hilliard equation and some useful inequalities}
\label{SectionRegularity}
In this section we prove an interpolation inequality, and regularity estimates for the solution to the stochastic Cahn-Hilliard equation.
\begin{lemma}
\label{Fundamentallemma}
Let $2<r<3$ and $C>0$. Then, there exists a positive constant $C_{\mathcal{D}}$, independent of $\varepsilon$ and $C$ such that for every $v\in \mathbb{H}^1\cap\mathbb{L}^2_0$ {and $\alpha, \beta>0$}, the following holds:
\begin{align*}
C\Vert v\Vert^3_{\mathbb{L}^3}\leq \Vert v\Vert^4_{\mathbb{L}^4}+C_{\mathcal{D}}\frac{C^{4-r}}{4-r}\varepsilon^{3-r} \Vert v\Vert^{\frac{4-r}{2}}_{\mathbb{H}^{-1}}\Vert  \nabla v\Vert^{\frac{4-r}{2}}_{\mathbb{L}^2}\Vert v\Vert_{\mathbb{L}^4}^{2r-4}.
\end{align*}
\end{lemma}

\begin{proof}
For $C>0$, $2<r<3$ the Young's inequality $ab\leq \frac{q-1}{q}a^{\frac{q}{q-1}}+\frac{b^q}{q}$ with $q=4-r$ yields
\begin{align*}
C\vert v\vert^3 = C\varepsilon^{\frac{3-r}{4-r}}(\vert v\vert^4)^{\frac{3-r}{4-r}}\varepsilon^{\frac{3-r}{4-r}}\vert v\vert^{\frac{r}{4-r}}\leq \vert v\vert^4+\frac{C^{4-r}}{4-r}\varepsilon^{3-r}\vert v\vert^r.
\end{align*}
Integrating the above estimate  over $\mathcal{D}$, we obtain
\begin{align}
\label{cubic2}
C\Vert v\Vert^3_{\mathbb{L}^3}\leq \Vert v\Vert^4_{\mathbb{L}^4}+\frac{C^{4-r}}{4-r}\varepsilon^{3-r}\Vert v\Vert^r_{\mathbb{L}^r}.
\end{align}
Let us recall the following interpolation inequality (see, for example, \cite[Proposition 6.10]{Folland})
\begin{align*}
\Vert u\Vert_{\mathbb{L}^{q'}}\leq \Vert u\Vert^{\lambda}_{\mathbb{L}^{p'}}\Vert u\Vert^{1-\lambda}_{\mathbb{L}^{r'}},\quad u\in \mathbb{L}^{r'}
\end{align*}
for $p'<q'<r'$ and $\lambda=\frac{p'}{q'}\frac{r'-q'}{r'-p'}$.
Using the preceding interpolation inequality  with $p'=2$, $q'=r$ and $r'=4$ (hence $\lambda=\frac{4-r}{r}$), we obtain
\begin{align*}
\Vert v\Vert^r_{\mathbb{L}^r}\leq \Vert v\Vert^{4-r}_{\mathbb{L}^2}\Vert v\Vert^{2r-4}_{\mathbb{L}^4}.
\end{align*}
By combining the above estimate with the interpolation inequality $\Vert v\Vert_{\mathbb{L}^2} \leq \Vert v\Vert^{\frac{1}{2}}_{\mathbb{H}^{-1}}\Vert \nabla v\Vert^{\frac{1}{2}}_{\mathbb{L}^2}$, we deduce that
\begin{align*}
\Vert v\Vert^r_{\mathbb{L}^r}\leq \Vert v\Vert^{\frac{4-r}{2}}_{\mathbb{H}^{-1}}\Vert  \nabla v\Vert^{\frac{4-r}{2}}_{\mathbb{L}^2}\Vert v\Vert_{\mathbb{L}^4}^{2r-4}.
\end{align*}
Substituting the preceding estimate into \eqref{cubic2} concludes the proof.
\end{proof}

The following generalized version of the Gronwall lemma was shown in  \cite[Lemma 2.1]{BartelsRuediger11}.
\begin{lemma}{[Generalized Gronwall's lemma]}
\label{Gronwall}
Let $T>0$ be fixed. Suppose that  $y_1\in C([0, T])$ is non-negative, $y_2, y_3\in L^1(0, T)$, $\alpha\in L^{\infty}(0, T)$, and there is $A\geq 0$ such that
\begin{align*}
y_1(t)+\int_0^ty_2(s)ds\leq A+\int_0^t\alpha(s)y_1(s)ds+\int_0^ty_3(s)ds,
\end{align*}
for all $t\in[0, T]$. Assume that for $B\geq 0$, $\beta>0$, and every $t\in[0, T]$, we have
\begin{align*}
\int_0^ty_3(s)ds\leq B\sup_{s\in[0, t]}y_1^{\beta}(s)\int_0^t\left(y_1(s)+y_2(s)\right)ds.
\end{align*}
Set $E=\exp\left(\int_0^t\alpha(s)ds\right)$ and assume that $8AE\leq \left(8B(1+T)E\right)^{-1/\beta}$. Then, it holds that
\begin{align*}
\sup_{t\in[0, T]}y_1(t)+\int_0^Ty_2(s)ds\leq 8A\exp\left(\int_0^T\alpha(s)ds\right). 
\end{align*}
\end{lemma}

In the next lemma we provide a regularity estimate of the solution to the linear SPDE \eqref{model2}.
\begin{lemma}
\label{Normutilde}
Let $\widetilde{u}$ be the solution to \eqref{model2}. 
For any $p\geq 2$, the following estimate holds:
\begin{align*}
\mathbb{E}\left[\sup_{t\in[0, T]}\Vert \widetilde{u}(t)\Vert^p_{\mathbb{L}^p}\right]\leq C\tilde{h}^{-\frac{3pd}{2}}\varepsilon^{-\frac{p}{2}}.
\end{align*}
\end{lemma}

\begin{proof}
Using the semi-group approach (cf. \cite{Debussche1, DaPratoZabczyk}), we express the solution to \eqref{model2} as:
\begin{align*}
\widetilde{u}(t)=\int_0^te^{-\varepsilon\Delta^2(t-s)}d\widetilde{W}(s)=\sum_{\ell=1}^L\frac{1}{\sqrt{(d+1)^{-1}\vert (\phi_{\ell}, 1)\vert}}\int_0^te^{-\varepsilon\Delta^2(t-s)}(\phi_{\ell}-m(\phi_{\ell}))d\beta_{\ell}(s).
\end{align*}
By applying the triangle inequality, we obtain: 
\begin{align}
\label{Refbruit1}
&\mathbb{E}\left[\sup_{t\in[0, T]}\Vert \widetilde{u}(t)\Vert^p_{\mathbb{L}^p}\right]\nonumber\\
&\quad\leq\sum_{\ell=1}^L\frac{ L^{p-1}}{(d+1)^{-\frac{p}{2}}\vert (\phi_{\ell}, 1)\vert^{\frac{p}{2}}}\mathbb{E}\left[\sup_{t\in[0, T]}\left\Vert\int_0^te^{-\varepsilon\Delta^2(t-s)}(\phi_{\ell}-m(\phi_{\ell}))d\beta_{\ell}(s)\right\Vert^p_{\mathbb{L}^p}\right].
\end{align}
Using the Burkh\"{o}lder-Davis-Gundy inequality (see, e.g., \cite[Theorem 4.36]{DaPratoZabczyk}), we obtain:
\begin{align*}
&\mathbb{E}\left[\sup_{t\in[0, T]}\left\Vert\int_0^te^{-\varepsilon\Delta^2(t-s)}(\phi_{\ell}-m(\phi_{\ell}))d\beta_{\ell}(s)\right\Vert^p_{\mathbb{L}^p}\right]\nonumber\\
&\qquad\leq C\mathbb{E}\left[\int_{\mathcal{D}}\sup_{t\in[0, T]}\left\vert\int_0^te^{-\varepsilon\Delta^2(t-s)}(\phi_{\ell}(x)-m(\phi_{\ell}))d\beta_{\ell}(s)\right\vert^pdx\right]\nonumber\\
&\qquad\leq C\mathbb{E}\left[\int_{\mathcal{D}}\sup_{t\in[0, T]}\left\vert\sum_{j\in\mathbb{N}^d}\int_0^te^{-\varepsilon\lambda_j^2(t-s)}(\phi_{\ell}(x)-m(\phi_{\ell}))e_j(x)d\beta_{\ell}(s)\right\vert^pdx\right]\\
&\qquad\leq C\int_{\mathcal{D}}\sum_{j\in\mathbb{N}^d}\mathbb{E}\left[\sup_{t\in[0, T]}\left\vert\int_0^te^{-\varepsilon\lambda_j^2(t-s)}(\phi_{\ell}(x)-m(\phi_{\ell}))e_j(x)d\beta_{\ell}(s)\right\vert^pdx\right]\nonumber\\
&\qquad\leq C\int_{\mathcal{D}}\left(\sum_{j\in\mathbb{N}^d}\int_0^Te^{-2\lambda_j^2(t-s)\varepsilon}\vert (\phi_{\ell}(x)-m(\phi_{\ell}))e_j(x)\vert^2ds\right)^{\frac{p}{2}}dx.\nonumber
\end{align*}
Using the fact that $\sum_{j\in\mathbb{N}^d}\frac{1}{\lambda_j^2}<\infty$, it follows from the estimate above that:
\begin{align*}
&\mathbb{E}\left[\sup_{t\in[0, T]}\left\Vert\int_0^te^{-\varepsilon\Delta^2(t-s)}(\phi_{\ell}-m(\phi_{\ell}))d\beta_{\ell}(s)\right\Vert^p_{\mathbb{L}^p}\right]\nonumber\\
&\quad\leq C\Vert \phi_{\ell}\Vert^p_{\mathbb{L}^{\infty}}\int_{\mathcal{D}}\left(\sum_{j\in\mathbb{N}^d}\int_0^Te^{-2\lambda_j^2(t-s)\varepsilon}ds\right)^{\frac{p}{2}}dx\leq C\int_{\mathcal{D}}\left(\sum_{j\in\mathbb{N}^d}\frac{1}{\lambda_j^2\varepsilon}\right)^{\frac{p}{2}}dx\leq C\varepsilon^{-\frac{p}{2}}. 
\end{align*}
Substituting the preceding estimate into \eqref{Refbruit1} and  using \lemref{Lemmabasis} completes the proof. 
\end{proof}

In the next Lemma we provide some regularity estimates of the solution to the stochastic Cahn-Hilliard equation \eqref{pb1}.
\begin{lemma}
\label{LemmaRegularity} 
Let $u$ be the solution to the stochastic Cahn-Hilliard equation \eqref{pb1}. Then there exists a constant $C\geq 0$, such that
\begin{align*}
\mathbb{E}\left[\sup_{t\in[0, T]}\Vert u(t)\Vert^2_{\mathbb{H}^{-1}}+\frac{1}{\varepsilon}\int_0^T\Vert u(s)\Vert^4_{\mathbb{L}^4}ds\right]\leq C\left(\Vert u_0\Vert^2_{\mathbb{H}^{-1}}+\varepsilon^{-1}+\tilde{h}^{-6d}\varepsilon^{-3}\right).
\end{align*}
\end{lemma}

\begin{proof}
 Let us recall that $u(t)=\widehat{u}(t)+\widetilde{u}(t)$, where $\widehat{u}(t)$ and $\widetilde{u}(t)$ solve \eqref{model3} and \eqref{model2} respectively. Equivalently, $\widehat{u}(t)$ satisfies the following random PDE
\begin{align*}
\frac{d\widehat{u}(t)}{dt}=-\varepsilon\Delta^2\widehat{u}(t)+\frac{1}{\varepsilon}\Delta f(u(t)),\quad
\widehat{u}(0)&=u_0,\quad t\in (0, T].
\end{align*}
Testing the above equation with $(-\Delta)^{-1}\widehat{u}(t)$ yields
\begin{align*}
\frac{1}{2}\frac{d}{dt}\Vert \widehat{u}(t)\Vert^2_{\mathbb{H}^{-1}}+\varepsilon\Vert\nabla \widehat{u}(t)\Vert^2+\frac{1}{\varepsilon}\left(f(u(t)), \widehat{u}(t)\right)=0.
\end{align*}
Using the  fact that $(f(v), v)\geq \frac{1}{2}\Vert v\Vert^4_{\mathbb{L}^4}-C$, $v\in \mathbb{L}^4$, it follows that
\begin{align}
\label{DaPratoZabczyk1}
\frac{1}{2}\frac{d}{dt}\Vert \widehat{u}(t)\Vert^2_{\mathbb{H}^{-1}}+\varepsilon\Vert\nabla \widehat{u}(t)\Vert^2+\frac{1}{2\varepsilon}\Vert u(t)\Vert^4_{\mathbb{L}^4}\leq \frac{C}{\varepsilon}+\frac{1}{\varepsilon}\left\vert \left(f(u(t)), \widetilde{u}(t)\right)\right\vert.
\end{align}
Noting that $\vert f(x)\vert\leq 2\vert x\vert^3+C_1$, using H\"{o}lder and Young's inequalities and the embbeding $\mathbb{L}^4\hookrightarrow \mathbb{L}^1$, we deduce that
\begin{align*}
\left\vert \left(f(u(t)), \widetilde{u}(t)\right)\right\vert
&\leq 2\int_{\mathcal{D}}\vert u(t)\vert^3\vert \widetilde{u}(t)\vert dx+C_1\int_{\mathcal{D}}\vert \widetilde{u}(t)\vert dx\nonumber\\
&\leq 2\left(\int_{\mathcal{D}}\vert u(t)\vert^4 dx\right)^{\frac{3}{4}}\left(\int_{\mathcal{D}}\vert \widetilde{u}(t)\vert^4 dx\right)^{\frac{1}{4}}+C_1\int_{\mathcal{D}}\vert \widetilde{u}(t)\vert dx\\
&\leq \frac{1}{4}\int_{\mathcal{D}}\vert u(t)\vert^4dx+C\int_{\mathcal{D}}\vert \widetilde{u}(t)\vert^4dx+C_1\int_{\mathcal{D}}\vert \widetilde{u}(t)\vert dx\nonumber\\
&\leq \frac{1}{4}\Vert u(t)\Vert^4_{\mathbb{L}^4}+C\Vert \widetilde{u}(t)\Vert^4_{\mathbb{L}^4}+C.\nonumber
\end{align*}
Substituting the preceding estimate into \eqref{DaPratoZabczyk1} and absorbing  $\frac{1}{4\varepsilon}\Vert u(t)\Vert^4_{\mathbb{L}^4}$ into the left hand side, yields
\begin{align*}
\frac{1}{2}\frac{d}{dt}\Vert \widehat{u}(t)\Vert^2_{\mathbb{H}^{-1}}+\varepsilon\Vert\nabla \widehat{u}(t)\Vert^2+\frac{1}{4\varepsilon}\Vert u(t)\Vert^4_{\mathbb{L}^4}\leq \frac{C}{\varepsilon}+\frac{C}{\varepsilon}\Vert \widetilde{u}(t)\Vert^4_{\mathbb{L}^4}.
\end{align*}
Integrating  over $[0, t]$ and  taking the supremum over $[0, T]$ yields
\begin{align*}
&\sup_{t\in[0, T]}\Vert\widehat{u}(t)\Vert^2_{\mathbb{H}^{-1}}+\varepsilon\int_0^T\Vert\nabla\widehat{u}(s)\Vert^2ds+\frac{1}{4\varepsilon}\int_0^T\Vert u(t)\Vert^4_{\mathbb{L}^4}ds\nonumber\\
&\quad\leq  \Vert \widehat{u}(0)\Vert^2_{\mathbb{H}^{-1}}  +\frac{C}{\varepsilon}+\frac{C}{\varepsilon}\int_0^T\Vert \widetilde{u}(s)\Vert^4_{\mathbb{L}^4}ds.
\end{align*}  
Taking the expectation on both sides  and  using \lemref{Normutilde} yields
\begin{align}
\label{depart1a}
&\mathbb{E}\left[\sup_{t\in[0, T]}\Vert \widehat{u}(t)\Vert^2_{\mathbb{H}^{-1}}\right]+\varepsilon\int_0^T\mathbb{E}\left[\Vert\nabla \widehat{u}(s)\Vert^2\right]ds+\frac{1}{4\varepsilon}\int_0^T\mathbb{E}\left[\Vert u(s)\Vert^4_{\mathbb{L}^4}\right]ds\nonumber\\
&\leq \Vert u_0\Vert^2_{\mathbb{H}^{-1}}+\frac{CT}{\varepsilon}+\frac{C}{\varepsilon}\int_0^T\mathbb{E}\left[\Vert \widetilde{u}^{\varepsilon}(s)\Vert^4_{\mathbb{L}^4}\right]ds\nonumber\\
&\leq C\left(\Vert u_0\Vert^2_{\mathbb{H}^{-1}}+\varepsilon^{-1}+\tilde{h}^{-6d}\varepsilon^{-3}\right). 
\end{align}
Using triangle inequality, the inequality $\Vert \cdot\Vert_{\mathbb{H}^{-1}}\leq C\Vert\cdot\Vert$,  \lemref{Normutilde} and \eqref{depart1a} yields
\begin{align}
\label{depart1b}
\mathbb{E}\left[\sup_{t\in[0, T]}\Vert u(t)\Vert^2_{\mathbb{H}^{-1}}\right]&\leq \mathbb{E}\left[\sup_{t\in[0, T]}\Vert \widehat{u}(t)\Vert^2_{\mathbb{H}^{-1}}\right]+\mathbb{E}\left[\sup_{t\in[0, T]}\Vert \widetilde{u}(t)\Vert^2_{\mathbb{H}^{-1}}\right]\nonumber\\
&\leq \mathbb{E}\left[\sup_{t\in[0, T]}\Vert \widehat{u}(t)\Vert^2_{\mathbb{H}^{-1}}\right]+C\mathbb{E}\left[\sup_{t\in[0, T]}\Vert \widetilde{u}(t)\Vert^2\right]\\
&\leq C\left(\Vert u_0\Vert^2_{\mathbb{H}^{-1}}+\varepsilon^{-1}+\tilde{h}^{-6d}\varepsilon^{-3}\right).\nonumber
\end{align}
Combining \eqref{depart1b} and \eqref{depart1a} ends the proof. 
\end{proof}

\section{Rate of convergence of the backward Euler method for linear stochastic Cahn-Hilliard equation with rough noise}
\label{Rateimplicit}
In this section, we examine the convergence rate of fully discrete scheme \eqref{scheme2} for the linear SPDE \eqref{model2}.
We consider a quasi-uniform triangulation $\mathcal{T}_h$ of $\mathcal{D}$, and $\mathbb{V}_h$ the associated finite element space of piecewise linear functions such that $\mathbb{V}_{\tilde{h}}\subset\mathbb{V}_h$.
 For simplicity we assume throughout this section that the finite element space $\mathbb{V}_h^n$ in \eqref{scheme2} is the same on all time levels, i.e. that $\mathbb{V}_h^n = \mathbb{V}_h$ for $n=0,\dots, N$.

The (seimi-discrete) finite element approximation of \eqref{model2} is given by:  find $\widetilde{u}_h(t), \widetilde{w}_h(t)\in \mathbb{V}_h$, for  $ t\in(0,T]$, such that:
\begin{align}
\label{FEMlinearSPDE}
\begin{array}{ll}
(\widetilde{u}_h(t), \varphi_h)+(\nabla\widetilde{u}_h(t), \nabla\varphi_h)=0&\quad  \forall \varphi\in \mathbb{V}_h,\\
(\widetilde{w}_h(t),\psi_h)=\varepsilon(\nabla\widetilde{u}_h(t), \nabla \psi_h)&\quad \forall \psi_h\in \mathbb{V}_h.
\end{array}
\end{align}
Analogously to \eqref{lineartransform} we introduce the linear transformation:
\begin{align*}
y_h(t)=\widetilde{u}_h(t)-\int_0^td\widetilde{W}(s)=\widetilde{u}_h(t)-\Sigma(t).
\end{align*}
and similarly to \eqref{Eqy} we conclude that $(y_h, \widetilde{w}_h)$ satisfies the random PDE
\begin{align}
\label{Eqydiscrete1}
\begin{array}{ll}
\langle \partial_ty_h(t), \varphi_h\rangle+(\nabla\widetilde{w}_h(t), \nabla\varphi_h)=0&\quad \forall\varphi_h\in \mathbb{V}_h,\\
(\widetilde{w}_h(t), \psi_h)=\varepsilon(\nabla\widetilde{u}_h(t), \psi)& \quad \forall \psi_h\in \mathbb{V}_h.
\end{array} 
\end{align}

\begin{lemma}
\label{ErrorNoiseLemma2}
Let $\Sigma(t)$ be the stochastic convolution given by \eqref{Sigma1}, and  let $\Sigma_{\tilde{h}, \tau}(t)$ be the continuous piecewise linear time-interpolant of $\{\Sigma_{\tilde{h}}^n\}_{n=0}^N$ given by \eqref{Sigma2}. 
Then, it  holds that:
\begin{align*}
\mathbb{E}\left[\Vert \Sigma(t)-\Sigma_{\tilde{h}, \tau}(t)\Vert^2_{\mathbb{H}^{-1}}\right]\leq C\tau_n\sum_{\ell=1}^L\frac{\Vert\phi_{\ell}-m(\phi_{\ell})\Vert^2_{\mathbb{H}^{-1}}}{(d+1)^{-1}\vert(\phi_{\ell},1)\vert}\quad \forall t\in(t_{n-1},t_n]. 
\end{align*}
\end{lemma}
\begin{proof}
From the definitions of $\Sigma(t)$ and $\Sigma_{\tilde{h}, \tau}(t)$, it follows that:
\begin{align*}
\mathbb{E}\left[\Vert \Sigma(t)-\Sigma_{\tilde{h}, \tau}(t)\Vert^2_{\mathbb{H}^{-1}}\right]&=\mathbb{E}\left[\left\Vert\int_0^td\widetilde{W}(s)-\frac{t-t_{n-1}}{\tau_n}\Delta_n\widetilde{W}-\sum_{i=1}^{n-1}\Delta_i\widetilde{W}\right\Vert^2_{\mathbb{H}^{-1}}\right]\nonumber\\
&=\mathbb{E}\left[\left\Vert\int_{t_{n-1}}^td\widetilde{W}(s)-\frac{t-t_{n-1}}{\tau_n}\int_{t_{n-1}}^{t_n}d\widetilde{W}(s)\right\Vert^2_{\mathbb{H}^{-1}}\right].
\end{align*}
Using the triangle inequality, the fact that  $\mathbb{E}[(\Delta_n\beta_{\ell})^2]=\tau_n$, $\mathbb{E}[(\Delta_n\beta_{\ell})(\Delta_n\beta_k)]=0$ for $k\neq \ell$, and the preceding equality, it follows that:
\begin{align*}
\mathbb{E}\left[\Vert \Sigma(t)-\Sigma_{\tilde{h}, \tau}(t)\Vert^2_{\mathbb{H}^{-1}}\right]&\leq  C\mathbb{E}\left[\left\Vert\int_{t_{n-1}}^td\widetilde{W}(s)\right\Vert^2_{\mathbb{H}^{-1}}\right]+C\mathbb{E}\left[\left\Vert\frac{t-t_{n-1}}{\tau_n}\Delta_n\widetilde{W}\right\Vert^2_{\mathbb{H}^{-1}}\right]\nonumber\\
&\leq  C\int_{t_{n-1}}^{t}\sum_{\ell=1}^L\frac{\Vert\phi_{\ell}-m(\phi_{\ell})\Vert^2_{\mathbb{H}^{-1}}}{(d+1)^{-1}\vert (\phi_{\ell},1)\vert}ds+C\tau_n\sum_{\ell=1}^L\frac{\Vert\phi_{\ell}-m(\phi_{\ell})\Vert^2_{\mathbb{H}^{-1}}}{(d+1)^{-1}\vert (\phi_{\ell}, 1)\vert}\nonumber\\
&\leq C\tau_n\sum_{\ell=1}^L\frac{\Vert\phi_{\ell}-m(\phi_{\ell})\Vert^2_{\mathbb{H}^{-1}}}{(d+1)^{-1}\vert (\phi_{\ell}, 1)\vert}.
\end{align*}
\end{proof}

\begin{lemma}
\label{LemmaDiscdisrete}
Let $(\widetilde{u}^n_h,\widetilde{w}^n_h)$ be the numerical solution satisfying \eqref{scheme2}. Then, there  exists a positive constant $C$ such that
\begin{align*}
\varepsilon\sum_{n=1}^N\mathbb{E}[\Vert\nabla[\widetilde{u}^n_h-\widetilde{u}^{n-1}_h]\Vert^2]+\sum_{n=1}^N\tau_n\mathbb{E}[\Vert\nabla \widetilde{w}^n_h\Vert^2]\leq C\sum_{\ell=1}^L\frac{\Vert\nabla\phi_{\ell}\Vert^2}{(d+1)^{-1}\vert (\phi_{\ell},1)\vert}.
\end{align*}
\end{lemma}
\begin{proof}
Taking $\varphi_h=\widetilde{w}^n_h$ and $\psi_h=\widetilde{u}^n_h-\widetilde{u}^{n-1}_h$ in \eqref{scheme2} we obtain
\begin{align*}
\frac{1}{\tau_n}(\widetilde{u}^n_h-\widetilde{u}^{n-1}_h, \widetilde{w}^n_h)+(\nabla\widetilde{w}^n_h, \nabla\widetilde{w}^n_h)=\frac{1}{\tau_n}(\Delta_n\widetilde{W}, \widetilde{w}^n_h)\\
(\widetilde{w}^n_h, \widetilde{u}^n_h-\widetilde{u}^{n-1}_h)=\varepsilon(\nabla\widetilde{u}^n_h, \nabla[\widetilde{u}^n_h-\widetilde{u}^{n-1}_h]).
\end{align*}
Combining the two preceding identities yields
\begin{align*}
\varepsilon(\nabla\widetilde{u}^n_h, \nabla[\widetilde{u}^n_h-\widetilde{u}^{n-1}_h])+\tau_n\Vert\nabla\widetilde{w}^n_h\Vert^2=(\Delta_n\widetilde{W},\widetilde{w}^n_h).
\end{align*}
Using the identity $2a(a-b)=a^2-b^2+(a-b)^2$ for $a,b\in\mathbb{R}$, we obtain
\begin{align}
\label{Discrepancy1}
&\frac{\varepsilon}{2}\left(\mathbb{E}[\Vert \nabla\widetilde{u}^n_h\Vert^2]-\mathbb{E}[\Vert \nabla\widetilde{u}^{n-1}_h\Vert^2]+\mathbb{E}[\Vert \nabla[\widetilde{u}^n_h-\widetilde{u}^{n-1}_h]\Vert^2]\right)+\tau_n\mathbb{E}[\Vert\nabla\widetilde{w}^n_h\Vert^2]\nonumber\\
&\quad=\mathbb{E}[(\Delta_n\widetilde{W},\widetilde{w}^n_h)].
\end{align}
Taking $\psi_h=\Delta_n\widetilde{W}$ in the second equation of \eqref{scheme2}, using the fact that $\nabla\Delta_n\widetilde{W}$ and $\nabla\widetilde{u}^{n-1}_h$ are independent, the fact that $\mathbb{E}[\nabla\Delta_n\widetilde{W}]=0$, and Young's inequality we obtain
\begin{align*}
\mathbb{E}[(\Delta_n\widetilde{W},\widetilde{w}^n_h)]&=\varepsilon\mathbb{E}[(\nabla\Delta_n\widetilde{W}, \nabla\widetilde{u}^n_h)]=\varepsilon\mathbb{E}[(\nabla\Delta_n\widetilde{W}, \nabla[\widetilde{u}^n_h-\widetilde{u}^{n-1}_h)]\nonumber\\
&\leq \frac{\varepsilon}{4}\Vert\nabla[\widetilde{u}^n_h-\widetilde{u}^{n-1}_h]\Vert^2+C\varepsilon\Vert \nabla\Delta_n\widetilde{W}\Vert^2.
\end{align*}
Substituting the preceding estimate in \eqref{Discrepancy1} and 
summing the resulting  inequality over $n\in\{1,\cdots,N\}$, we get
\begin{align}
\label{Discrepancy2}
&\frac{\varepsilon}{2}\mathbb{E}[\Vert \nabla \widetilde{u}^N_h\Vert^2]+\frac{\varepsilon}{4}\sum_{n=1}^N\mathbb{E}[\Vert\nabla[\widetilde{u}^n_h-\widetilde{u}^{n-1}_h]\Vert^2]+\sum_{n=1}^N\tau_n\mathbb{E}[\Vert\nabla \widetilde{w}^n_h\Vert^2]\nonumber\\
&\quad\leq \frac{\varepsilon}{2}\mathbb{E}[\Vert\nabla\widetilde{u}^0_h\Vert^2]+C\varepsilon\sum_{n=1}^N\mathbb{E}[\Vert\nabla\Delta_n\widetilde{W}\Vert^2].
\end{align}
From the definition of $\Delta_n\widetilde{W}$ in \eqref{Noiseapprox2}, using  the fact that $\mathbb{E}[\Delta_n\beta_j\Delta_n\beta_k]=\tau_n\delta_{j,k}$, yields
\begin{align*}
\mathbb{E}[\Vert\nabla\Delta_n\widetilde{W}\Vert^2]\leq C\mathbb{E}[\Vert\nabla\Delta_n\widetilde{W}\Vert^2]\leq C\tau_n\sum_{\ell=1}^L\frac{\Vert\nabla\phi_{\ell}\Vert^2}{(d+1)^{-1}\vert (\phi_{\ell},1)\vert}.
\end{align*}
Substituting the preceding estimate in \eqref{Discrepancy2} and using the fact that $\widetilde{u}^0_h=0$, we obtain
\begin{align*}
\frac{\varepsilon}{4}\sum_{n=1}^N\mathbb{E}[\Vert\nabla[\widetilde{u}^n_h-\widetilde{u}^{n-1}_h]\Vert^2]+\sum_{n=1}^N\tau_n\mathbb{E}[\Vert\nabla \widetilde{w}^n_h\Vert^2]\leq C\sum_{\ell=1}^L\frac{\Vert\nabla\phi_{\ell}\Vert^2}{(d+1)^{-1}\vert (\phi_{\ell},1)\vert}.
\end{align*}
\end{proof} 

We define the piecewise constant time interpolant $\bar{\widetilde{u}}_{h,\tau}$ of the numerical solution $\{\widetilde{u}_h^n\}_{n=0}^N$ in \eqref{scheme2} as:
\begin{align*}
\bar{\widetilde{u}}_{h,\tau}(t):=u^n_h\quad \text{if}\; t\in(t_{n-1}, t_n],\quad n=1,\cdots,N, \quad \text{where}\quad \tau=\max\limits_{1\leq n\leq N}\tau_n.
\end{align*}
Analogously, we define the piecewise constant time interpolant $\bar{\widetilde{w}}_{h,\tau}$ of the numerical solution $\{\widetilde{w}_h^n\}_{n=0}^N$ in \eqref{scheme2}.

\begin{lemma}
\label{LemmaErrorConstantInter}
Let $\bar{\widetilde{u}}_{h,\tau}(t)$ and $\widetilde{u}_{h,\tau}(t)$ be respectively the piecewise constant and the piecewise linear interpolants in time of the numerical solution $\{\widetilde{u}^n_h\}_{n=0}^N$. Then, the following estimate holds
\begin{align*}
\int_0^T\mathbb{E}[\Vert\nabla[\bar{\widetilde{u}}_{h,\tau}(t)-\widetilde{u}_{h,\tau}(t)]\Vert^2]dt\leq C\tau\varepsilon^{-1}\sum_{\ell=1}^L\frac{\Vert\nabla\phi_{\ell}\Vert^2}{(d+1)^{-1}\vert (\phi_{\ell},1)\vert},
\end{align*}
where $C$ is a positive constant independent of $\tau$.
\end{lemma}

\begin{proof}
Easy computations lead to
\begin{align*}
\int_0^T\mathbb{E}[\Vert\nabla[\bar{\widetilde{u}}_{h,\tau}(t)-\widetilde{u}_{h,\tau}(t)]\Vert^2]dt&=\sum_{n=1}^N\int_{t_{n-1}}^{t_n}\mathbb{E}[\Vert\nabla[\bar{\widetilde{u}}_{h,\tau}(t)-\widetilde{u}_{h,\tau}(t)]\Vert^2]dt\nonumber\\
&=\sum_{n=1}^N\frac{1}{\tau_n^2}\mathbb{E}[\Vert\nabla[\widetilde{u}^n_h-\widetilde{u}^{n-1}_h]\Vert^2]\int_{t_{n-1}}^{t_n}(t-t_n)^2dt\nonumber\\
&\leq C\sum_{n=1}^N\tau_n\mathbb{E}[\Vert\nabla[\widetilde{u}^n_h-\widetilde{u}^{n-1}_h]\Vert^2].
\end{align*}
Using \lemref{LemmaDiscdisrete}, it follows from the preceding estimate that
\begin{align*}
\int_0^T\mathbb{E}[\Vert\nabla[\bar{\widetilde{u}}_{h,\tau}(t)-\widetilde{u}_{h,\tau}(t)]\Vert^2]dt&\leq C\tau\sum_{n=1}^N\mathbb{E}[\Vert\nabla[\widetilde{u}^n_h-\widetilde{u}^{n-1}_h]\Vert^2]\nonumber\\
&\leq C\tau\varepsilon^{-1}\sum_{\ell=1}^L\frac{\Vert\nabla\phi_{\ell}\Vert^2}{(d+1)^{-1}\vert (\phi_{\ell},1)\vert}.
\end{align*}
\end{proof}

In the next lemma, we provide  an error estimate for $\widetilde{u}_h(t)-\widetilde{u}_{h,\tau}(t)$.
\begin{lemma}
\label{Errorrtatetime}
Let $\widetilde{u}_h$ be the solution to \eqref{FEMlinearSPDE}, and let $\widetilde{u}_{h, \tau}$ be the continuous piecewise linear time-interpolant of $\{\widetilde{u}_h^n\}_{n=0}^N$, satisfying \eqref{scheme2}. Then, the following error estimate holds:
\begin{align*}
&\sup_{t\in[0, T]}\mathbb{E}[\Vert \widetilde{u}_h(t)-\widetilde{u}_{h,\tau}(t)\Vert^2_{\mathbb{H}^{-1}}]+\varepsilon\int_0^T\mathbb{E}[\Vert \nabla(\widetilde{u}_h(s)-\widetilde{u}_{h,\tau}(s))\Vert^2]ds\nonumber\\
&\quad\leq  C\tau\left(\sum_{\ell=1}^L\frac{\Vert\phi_{\ell}-m(\phi_{\ell})\Vert^2_{\mathbb{H}^{-1}}+\Vert \nabla\phi_{\ell}\Vert^2}{(d+1)^{-1}\vert(\phi_{\ell},1)\vert}\right).
\end{align*}
\end{lemma}

\begin{proof}
Using \eqref{schemey} and \eqref{DerivInterpoly}, it follows that $y_{h,\tau}$ satisfies:
\begin{align}
\label{Eqydiscrete2}
\begin{array}{ll}
\langle\partial_t y_{h,\tau}(t), \varphi_h\rangle+(\nabla\bar{\widetilde{w}}_{h,\tau}(t), \nabla \varphi_h)=0&\quad \forall\varphi_h\in\mathbb{V}_h\\
(\bar{\widetilde{w}}_{h,\tau}(t), \psi_h)=\varepsilon(\nabla\bar{\widetilde{u}}_{h,\tau}(t), \nabla\psi_h)&\quad \forall \psi_h\in\mathbb{V}_h.
\end{array} 
\end{align}
Subtracting \eqref{Eqydiscrete2} from \eqref{Eqydiscrete1} yields:
\begin{align}
\label{Eqydiffsicrete1}
\begin{array}{ll}
\langle\partial_t[y_h(t)-y_{h,\tau}(t)], \varphi_h\rangle+(\nabla[\widetilde{w}_h(t)-\bar{\widetilde{w}}_{h,\tau}(t)], \nabla \varphi_h)=0 &\quad \forall \varphi_h\in\mathbb{V}_h,\\
(\widetilde{w}_h(t)-\bar{\widetilde{w}}_{h,\tau}(t), \psi_h)=\varepsilon(\nabla[\widetilde{u}_h(t)-\bar{\widetilde{u}}_{h,\tau}(t)], \nabla\psi_h)& \quad \forall \psi_h\in \mathbb{V}_h. 
\end{array}
\end{align}
Taking $\varphi_h=(-\Delta_h)^{-1}(y_h(t)-y_{h,\tau}(t))$
in the first equation of \eqref{Eqydiffsicrete1}, we obtain: 
\begin{align}
\label{Eqydiff2}
\frac{1}{2}\frac{d}{dt}\Vert y_h(t)-y_{h,\tau}(t)\Vert^2_{\mathbb{H}^{-1}}+(\widetilde{w}_h(t)-\bar{\widetilde{w}}_{h,\tau}(t), y_h(t)-y_{h,\tau})=0.
\end{align}
Integrating \eqref{Eqydiff2} over $(0, t)$, noting that $y_{h,\tau}(0)=y_h(0)=0$, and taking the expectation yields:
\begin{align}
\label{Eqydiff4}
\frac{1}{2}\mathbb{E}[\Vert y_h(t)-y_{h,\tau}(t)\Vert^2_{-1,h}]=-\int_0^t\mathbb{E}[(\widetilde{w}_h(s)-\bar{\widetilde{w}}_{h,\tau}(s), y_h(s)-y_{h,\tau}(s))]ds.
\end{align}
Taking $\psi_h=y_h(t)-y_{h,\tau}(t)$ in \eqref{Eqydiffsicrete1},  recalling that  $y_h(t)=\tilde{u}_h(t)-\Sigma(s)$,
and $y_{h,\tau}(t)=\tilde{u}_{h,\tau}(t)-\Sigma_{\tilde{h}, \tau}(t)$, we obtain:
\begin{align*}
&(\widetilde{w}_h(t)-\bar{\widetilde{w}}_{h,\tau}(t), y_h(t)-y_{h,\tau}(t))\nonumber\\
&=\varepsilon(\nabla[\widetilde{u}_h(t)-\widetilde{u}_{h,\tau}(t)], \nabla[ y_h(t)-y_{h,\tau}(t)]+\varepsilon(\nabla[\widetilde{u}_{h,\tau}(t)-\bar{\widetilde{u}}_{h,\tau}(t)], \nabla[y_h(t)-y_{h,\tau}(t)])
\nonumber\\
&=\varepsilon\Vert \nabla[y_h(t)-y_{h,\tau}(t)]\Vert^2-\varepsilon(\nabla[\Sigma(t)-\Sigma_{\tilde{h}, \tau}(t)], \nabla[y_h(t)-y_{h,\tau}(t)])\nonumber\\
&\quad +\varepsilon(\nabla[\widetilde{u}_{h,\tau}(t)-\bar{\widetilde{u}}_{h,\tau}(t)], \nabla[y_h(t)-y_{h,\tau}(t))]. 
\end{align*}
Substituting the preceding identity into \eqref{Eqydiff4} leads to:
\begin{align}
\label{Eqydiff5}
&\frac{1}{2}\mathbb{E}[\Vert y_h(t)-y_{h,\tau}(t)\Vert^2_{-1,h}]+\varepsilon\int_0^t\mathbb{E}[\Vert \nabla(y_h(s)-y_{h,\tau}(s))\Vert^2]ds\nonumber\\
&\quad=\varepsilon\int_0^t\mathbb{E}[(\nabla[\Sigma(s)-\Sigma_{\tilde{h}, \tau}(s)], \nabla[y_h(s)-y_{h,\tau}(s)])]ds\nonumber\\
&\quad \quad -\varepsilon\int_0^t\mathbb{E}[(\nabla[\widetilde{u}_{h,\tau}(s)-\bar{\widetilde{u}}_{h,\tau}(s)], \nabla[y_h(s)-y_{h,\tau}(s))])]ds.
\end{align}
Using  Cauchy-Schwarz's inequality and Young's inequality, we obtain:
\begin{align*}
&\mathbb{E}\left[\left(\nabla[\Sigma(s)-\Sigma_{\tilde{h}, \tau}(s)], \nabla[y_h(s)-y_{h,\tau}(s)]\right)\right]\nonumber\\
&\quad\leq \frac{1}{4}\mathbb{E}[\Vert \nabla(y_h(s)-y_{h,\tau}(s))\Vert^2]+C\mathbb{E}[\Vert \nabla(\Sigma(s)-\Sigma_{\tilde{h},\tau}(s))\Vert^2].
\end{align*}
Using again  Cauchy-Schwarz's inequality and Young's inequality, we estimate:
\begin{align*}
&\mathbb{E}[(\nabla[\widetilde{u}_{h,\tau}(t)-\bar{\widetilde{u}}_{h,\tau}(t)], \nabla[y_h(t)-y_{h,\tau}(t)])]\nonumber\\
&\leq \frac{1}{4}\mathbb{E}[\Vert \nabla[y_h(t)-y_{h,\tau}(t)]\Vert^2]+C\mathbb{E}[\Vert\nabla[\widetilde{u}_{h,\tau}(t)-\bar{\widetilde{u}}_{h,\tau}(t)]\Vert^2].
\end{align*}
Substituting the two preceding estimates into \eqref{Eqydiff5} and taking the supremum over $[0, T]$ we obtain:
\begin{align*}
&\sup_{t\in[0, T]}\mathbb{E}[\Vert y_h(t)-y_{h,\tau}(t)\Vert^2_{-1,h}]+\varepsilon\int_0^T\mathbb{E}[\Vert \nabla(y_h(s)-y_{h,\tau}(s))\Vert^2]ds\nonumber\\
&\quad\leq C\varepsilon\int_0^T\mathbb{E}[\Vert \nabla(\Sigma(s)-\Sigma_{\tilde{h},\tau}(s))\Vert^2]ds+C\varepsilon\int_0^T\mathbb{E}[\Vert\nabla[\widetilde{u}_{h,\tau}(s)-\bar{\widetilde{u}}_{h,\tau}(s)]\Vert^2]ds. 
\end{align*}
Noting \eqref{normeq}, using Lemmas \ref{ErrorNoiseLemma}, \ref{LemmaErrorConstantInter}, and recalling \eqref{Noise1}, it follows from the preceding estimate that:
\begin{align*}
&\sup_{t\in[0, T]}\mathbb{E}\left[\Vert y_h(t)-y_{h,\tau}(t)\Vert^2_{\mathbb{H}^{-1}}\right]+\varepsilon\int_0^T\mathbb{E}[\Vert \nabla(y_h(s)-y_{h,\tau}(s))\Vert^2]ds\nonumber\\
&\qquad\leq C\tau\sum_{\ell=1}^L\frac{\Vert\nabla\phi_{\ell}\Vert^2}{(d+1)^{-1}\vert(\phi_{\ell},1)\vert}+ C\varepsilon\sum_{n=1}^N\eta^n_{\text{NOISE}}\leq C\tau\sum_{\ell=1}^L\frac{\Vert\nabla\phi_{\ell}\Vert^2}{(d+1)^{-1}\vert(\phi_{\ell},1)\vert}.
\end{align*}
Recalling that  $y_h(t)=\tilde{u}_h(t)-\Sigma(t)$ 
and $y_{h,\tau}(t)=\tilde{u}_{h,\tau}(t)-\Sigma_{\tilde{h}, \tau}(t)$, and applying the  triangle inequality,  \lemref{ErrorNoiseLemma} and the preceding estimate, it follows that:
\begin{align}
\label{Eqydiff6}
\varepsilon\int_0^T\mathbb{E}[\Vert \nabla(\widetilde{u}_h(s)-\widetilde{u}_{h,\tau}(s))\Vert^2]ds\leq  C\tau\sum_{\ell=1}^L\frac{\Vert\nabla\phi_{\ell}\Vert^2}{(d+1)^{-1}\vert(\phi_{\ell},1)\vert}.
\end{align}
Using the triangle inequality,  \lemref{ErrorNoiseLemma2}, and the estimate \eqref{Eqydiff6}, it follows that:
\begin{align*}
\sup_{t\in[0, T]}\mathbb{E}[\Vert \widetilde{u}_h(t)-\widetilde{u}_{h,\tau}(t)\Vert^2_{\mathbb{H}^{-1}}]\leq C\tau\left(\sum_{\ell=1}^L\frac{\Vert\phi_{\ell}-m(\phi_{\ell})\Vert^2_{\mathbb{H}^{-1}}}{(d+1)^{-1}\vert(\phi_{\ell},1)\vert}+\sum_{\ell=1}^L\frac{\Vert\nabla\phi_{\ell}\Vert^2}{(d+1)^{-1}\vert(\phi_l,1)\vert}\right).
\end{align*}
Summing the two preceding estimates completes the proof. 
\end{proof}

Let us recall that \eqref{model2} can be written in the following "formal" abstract form   (see, e.g., the introduction of \cite{Larsson3}):
\begin{align}
\label{modelElliotcontinuous}
d\widetilde{u}(t)=-\varepsilon\Delta^2\widetilde{u}(t)+d\widetilde{W}(t),\quad t\in (0, T],\; \widetilde{u}(0)=0.
\end{align}
The mild solution of \eqref{modelElliotcontinuous} satisfies $\mathbb{P}$-a.s.:
\begin{align}
\label{mildElliotcont}
\widetilde{u}(t)&=\int_0^te^{-\Delta^2\varepsilon(t-s)}d\widetilde{W}(s)\nonumber\\
&=\sum_{\ell=1}^L\frac{1}{\sqrt{(d+1)^{-1}\vert(\phi_{\ell},1)\vert}}\int_0^te^{-\Delta^2\varepsilon(t-s)}(\phi_{\ell}-m(\phi_{\ell}))d\beta_{\ell}(s) \quad \forall t\in (0, T].
\end{align}
Equivalently, the finite element solution $\widetilde{u}_h(t)$   of \eqref{FEMlinearSPDE} satisfies (see, e.g.,  \cite{Larsson3}):
\begin{align}
\label{modelElliotFEM}
d\widetilde{u}_h(t)=-\varepsilon\Delta_h^2\widetilde{u}_hdt+d\widetilde{W}(t),\quad t\in(0, T],\; \widetilde{u}_h(0)=0, 
\end{align}
where the operator $\Delta_h: \mathring{\mathbb{V}}_h\rightarrow\mathring{\mathbb{V}}_h$ (the "discrete Laplacian") is defined by:
\begin{align*}
(-\Delta_h\xi_h, \eta_h)=(\nabla\xi_h,\nabla\eta_h)\quad \forall \xi_h, \eta_h\in\mathbb{V}^n_h.
\end{align*}
The mild solution $\widetilde{u}_h(t)$ can therefore be written as follows
\begin{align}
\label{mildElliotFEM}
\widetilde{u}_h(t)&=\int_{0}^te^{-\Delta^2_h\varepsilon(t-s)}P_hd\widetilde{W}(s)\nonumber\\
&=\sum_{\ell=1}^L\frac{1}{\sqrt{(d+1)^{-1}\vert(\phi_{\ell},1)\vert}}\int_0^te^{-\Delta^2_h\varepsilon(t-s)}P_h(\phi_{\ell}-m(\phi_{\ell}))d\beta_{\ell}(s)\quad \forall t\in (0, T].
\end{align}
We aim to provide an error estimate for $\widetilde{u}(t)-\widetilde{u}_h(t)$. We begin by recalling the following error estimate for the approximation of the semi-group from \cite[Lemma 5.2]{ElliotLarsson}.
\begin{lemma}
\label{ElliotLemma}
Let $r\in\{2,3\}$, and let $\alpha\in[0,r]$ be  such that $0\leq r-\alpha\leq 2$. Then, it holds:
\begin{align*}
\left\Vert \left(e^{-\Delta^2\varepsilon t}-e^{-\Delta_h^2\varepsilon t}P_h\right)v\right\Vert_k\leq Ch^{r-k}(\varepsilon
t)^{-\frac{r-\alpha}{4}}\vert v\vert_{\alpha},\quad t>0,\quad k=0,1,2, 
\end{align*}
where $\vert v\vert_{\alpha}=\Vert \Delta^{\alpha}v\Vert$ and $\Vert v\Vert_k=\Vert v\Vert_{\mathbb{H}^k}$. 
\end{lemma}
\begin{lemma}
\label{Errorratespace}
Let $\widetilde{u}(t)$ and $\widetilde{u}_h(t)$ be the mild solutions of \eqref{modelElliotcontinuous} and \eqref{modelElliotFEM}, respectively. Then, the following error estimate holds:
\begin{align*}
\mathbb{E}[\Vert \widetilde{u}(t)-\widetilde{u}_h(t)\Vert^2]+\varepsilon\mathbb{E}[\Vert\nabla(\widetilde{u}(t)-\widetilde{u}_h(t))\Vert^2]\leq C h^2\sum_{\ell=1}^L\frac{\Vert\nabla\phi_{\ell}\Vert^2}{(d+1)^{-1}\vert(\phi_{\ell},1)\vert}. 
\end{align*} 
\end{lemma}

\begin{proof}
 Subtracting \eqref{mildElliotFEM} from \eqref{mildElliotcont} yields:
\begin{align*}
&\nabla(\widetilde{u}(t)-\widetilde{u}_h(t))\nonumber\\
&\quad=\sum_{\ell=1}^L\frac{1}{\sqrt{(d+1)^{-1}\vert(\phi_{\ell},1)\vert}}\int_0^t\nabla\left(e^{-\Delta^2\varepsilon(t-s)}-e^{-\Delta^2_h\varepsilon(t-s)}P_h\right)(\phi_{\ell}-m(\phi_{\ell}))d\beta_{\ell}(s).
\end{align*} 
Using the  It\^{o} isometry, the fact that   $\mathbb{E}[(\Delta_n\beta_{\ell})^2]=\tau_n$, and $\mathbb{E}[(\Delta_n\beta_{\ell})(\Delta_n\beta_k)]=0$ for $k\neq \ell$,  we obtain: 
\begin{align*}
&\mathbb{E}[\Vert\nabla(\widetilde{u}(t)-\widetilde{u}_h(t))\Vert^2]\nonumber\\
&\quad\leq \sum_{\ell=1}^L\frac{1}{(d+1)^{-1}\vert(\phi_{\ell},1)\vert}\mathbb{E}\left[\left\Vert\int_0^t\nabla\left(e^{-\Delta^2\varepsilon(t-s)}-e^{-\Delta^2_h\varepsilon(t-s)}P_h\right)(\phi_{\ell}-m(\phi_{\ell}))d\beta_{\ell}(s)\right\Vert^2\right]\nonumber\\
&\quad\leq C\sum_{\ell=1}^L\frac{1}{(d+1)^{-1}\vert(\phi_{\ell},1)\vert}\int_0^t\left\Vert\nabla\left(e^{-\Delta^2\varepsilon(t-s)}-e^{-\Delta^2_h\varepsilon(t-s)}P_h\right)(\phi_{\ell}-m(\phi_{\ell}))\right\Vert^2ds.
\end{align*}
Using the estimate $\Vert \nabla v\Vert\leq \Vert v\Vert_1$,  \lemref{ElliotLemma} with $r=2$, $\alpha=1$ and $k=1$ yields:
\begin{align*}
&\mathbb{E}[\Vert\nabla(\widetilde{u}(t)-\widetilde{u}_h(t))\Vert^2]\nonumber\\
&\quad\leq C\sum_{\ell=1}^L\frac{1}{(d+1)^{-1}\vert(\phi_{\ell},1)\vert}\int_0^t\left\Vert\left(e^{-\Delta^2\varepsilon(t-s)}-e^{-\Delta^2_h\varepsilon(t-s)}P_h\right)(\phi_{\ell}-m(\phi_{\ell}))\right\Vert^2_{\mathbb{H}^1}ds\nonumber\\
&\quad \leq Ch^2\sum_{\ell=1}^L\frac{\Vert \nabla\phi_{\ell}\Vert^2}{(d+1)^{-1}\vert(\phi_{\ell},1)\vert}\int_0^t\varepsilon^{-\frac{1}{2}}(t-s)^{-\frac{1}{2}}ds\nonumber\\
&\quad\leq C\varepsilon^{-\frac{1}{2}} h^2\sum_{\ell=1}^L\frac{\Vert \nabla\phi_{\ell}\Vert^2}{(d+1)^{-1}\vert(\phi_{\ell},1)\vert}. 
\end{align*}
Along the same lines as above, by using \lemref{ElliotLemma} with $r=\alpha=1$ and $k=0$, we obtain:
\begin{align*}
\mathbb{E}[\Vert\widetilde{u}(t)-\widetilde{u}_h(t)\Vert^2]\leq C h^2\sum_{\ell=1}^L\frac{\Vert \nabla\phi_{\ell}\Vert^2}{(d+1)^{-1}\vert(\phi_{\ell},1)\vert}.
\end{align*}
By combining the two preceding estimates, we conclude the proof. 
\end{proof}

Using trianle inequality and Lemmas \ref{Errorratespace} and \ref{Errorrtatetime} we obtain the following error estimate.
\begin{theorem}
\label{ThmRateLinearSPDE}
Let $\widetilde{u}_h$ be the solution of \eqref{model2}, and let $\widetilde{u}_{h, \tau}$ be the continuous piecewise linear time-interpolant of $\{\widetilde{u}_h^n\}_{n=0}^N$, satisfying \eqref{scheme2}. Then, the following error estimate holds:
\begin{align*}
&\sup_{t\in[0, T]}\mathbb{E}[\Vert \widetilde{u}(t)-\widetilde{u}_{h,\tau}(t)\Vert^2_{\mathbb{H}^{-1}}]+\varepsilon\int_0^T\mathbb{E}[\Vert \nabla(\widetilde{u}(s)-\widetilde{u}_{h,\tau}(s))\Vert^2]ds\nonumber\\
&\quad\leq  C(h^2+\tau)\left(\sum_{\ell=1}^L\frac{\Vert\phi_{\ell}-m(\phi_{\ell})\Vert^2_{\mathbb{H}^{-1}}+\Vert \nabla\phi_{\ell}\Vert^2}{(d+1)^{-1}\vert(\phi_{\ell},1)\vert}\right).
\end{align*}
\end{theorem}
Using \thmref{ThmRateLinearSPDE} along with \lemref{Lemmabasis} implies the following error estimate.
\begin{corollary}
\label{RateLinearSPDE}
The following error estimate holds:
\begin{align*}
\sup_{t\in[0, T]}\mathbb{E}[\Vert \widetilde{u}(t)-\widetilde{u}_{h,\tau}(t)\Vert^2_{\mathbb{H}^{-1}}]+\varepsilon\int_0^T\mathbb{E}[\Vert \nabla(\widetilde{u}(s)-\widetilde{u}_{h,\tau}(s))\Vert^2]ds\leq C(h^2+\tau)\tilde{h}^{-2-d}. 
\end{align*}
\end{corollary}

\section*{Acknowledgement}
Funded by the Deutsche Forschungsgemeinschaft (DFG, German Research Foundation) -- Project-ID 317210226 -- SFB 1283.

\bibliographystyle{plain}
\bibliography{references}

\end{document}